\documentclass[11pt]{amsart}
\usepackage{amsmath}
\usepackage{amssymb, verbatim}
\usepackage{pstricks,pst-node,pst-plot}
\usepackage{comment}
\usepackage{color}
\usepackage{extarrows}

\title[]{Special Lagrangian submanifolds of log Calabi-Yau manifolds}
\author[T. C. Collins]{Tristan C. Collins}
  \email{tristanc@mit.edu}
  \address{Department of Mathematics, Massachusetts Institute of Technology, 77 Massachusetts Avenue, Cambridge, MA 02139}
 \thanks{T.C.C is supported in part by NSF grant DMS-1810924, NSF CAREER grant DMS-1944952 and an Alfred P. Sloan Fellowship. }

 \author[A. Jacob]{Adam Jacob}
  \email{ajacob@math.ucdavis.edu}
  \address{Department of Mathematics, University of California, Davis, 1 Shields Ave, Davis, CA, 95616}
  \thanks{A.J. is supported in part by a Simons collaboration grant}
  
  \author[Y.-S. Lin]{Yu-Shen Lin}
  \email{yslin@bu.edu}
  \address{Department of Mathematics, Boston University, 11 Cummington Mall, Boston, MA 02215}
    \dedicatory{Dedicated to S.-T. Yau with admiration on the occasion of his 70th birthday.}

\theoremstyle{plain}
\newtheorem{thm}{Theorem}[section]
\newtheorem{prop}[thm]{Proposition}
\newtheorem{defn}[thm]{Definition}
\newtheorem{lem}[thm]{Lemma}
\newtheorem{cor}[thm]{Corollary}

\theoremstyle{definition}

\newtheorem{rk}[thm]{Remark}

\numberwithin{equation}{section}

\newcommand{\del}{\partial}

\newcommand{\dbar}{\overline{\del}}
\newcommand{\ddb}{\sqrt{-1}\del\dbar}
\newcommand{\ddt}{\frac{\del}{\del t}}

\newcommand{\ti}[1]{\tilde{#1}}

\newcommand{\cC}{\mathcal{C}}

\renewcommand{\leq}{\leqslant}
\renewcommand{\geq}{\geqslant}

\renewcommand{\epsilon}{\varepsilon}
\renewcommand{\phi}{\varphi}

\begin{document}

\maketitle

\begin{abstract}
We study the existence of special Lagrangian submanifolds of log Calabi-Yau manifolds equipped with the complete Ricci-flat K\"ahler metric constructed by Tian-Yau.  We prove that if $X$ is a Tian-Yau manifold and if the compact Calabi-Yau manifold at infinty admits a single special Lagrangian, then $X$ admits infinitely many disjoint special Lagrangians.  In complex dimension $2$, we prove that if $Y$ is a del Pezzo surface or a rational elliptic surface and $D\in |-K_{Y}|$ is a smooth divisor with $D^2=d$, then $X= Y\setminus D$ admits a special Lagrangian torus fibration, as conjectured by Strominger-Yau-Zaslow and Auroux.  In fact, we show that $X$ admits twin special Lagrangian fibrations, confirming a prediction of Leung-Yau.  In the special case that $Y$ is a rational elliptic surface or $Y= \mathbb{P}^2$ we identify the singular fibers for generic data, thereby confirming two conjectures of Auroux.  Finally, we prove that after a hyper-K\"ahler rotation, $X$ can be compactified to the complement of a Kodaira type $I_{d}$ fiber appearing as a singular fiber in a rational elliptic surface $\check{\pi}: \check{Y}\rightarrow \mathbb{P}^1$.
\end{abstract}

\section{Introduction}

Mirror symmetry arose from physics as a mysterious duality between Hodge numbers of certain Calabi-Yau threefolds $Y, \check{Y}$ (see, for example, \cite{COGP, CK}).  Over the past 30 years, mirror symmetry has attracted intense interest from mathematicians.  In 1994 Kontsevich \cite{Kont} proposed that mirror symmetry could be explained as a certain duality between categories; the derived category of sheaves $D^{b}Coh(Y)$ on the one hand and the Fukaya category $Fuk(\check{Y})$ on the other.  This proposal has come to be known as homological mirror symmetry (HMS).  Strominger-Yau-Zaslow \cite{SYZ} proposed a geometric mechanism for mirror symmetry based on the prediction that, in certain limits, Calabi-Yau manifolds admit fibrations by special Lagrangian tori.  Mirror symmetry is then obtained by fiberwise $T$-duality and the categories $D^{b}Coh(Y), Fuk(\check{Y})$ are related by a real Fourier-Mukai transform \cite{LYZ}. 

 A fundamental difficulty in making progress on the SYZ proposal has been the dearth of special Lagrangian submanifolds in Calabi-Yau manifolds.  Indeed, it is unknown whether a general Calabi-Yau manifold admits even a single special Lagrangian submanifold.  In fact, the only examples of Calabi-Yau manifolds which are known to admit SYZ fibrations are essentially trivial; complex tori and hyper-K\"ahler manifolds with holomorphic fibrations by complex tori, where SYZ fibrations can be produced by a hyper-K\"ahler rotation.  Nevertheless, these examples can be used to give non-trivial evidence for the SYZ picture.  For example, when $Y$ is an elliptically fibered $K3$ surface with 24 $I_1$ singular fibers, Gross-Wilson \cite{GrWi} used the hyper-K\"ahler rotation trick, together with a careful analysis of the Calabi-Yau metrics, to confirm the SYZ picture.  We refer the reader to \cite{GTW, GTW1, Ru, Ru1, To, To1, ToZh, Zha} and the references therein for related work.
 
In the case of non-compact manifolds only slightly more is known.  In special cases,  Goldstein \cite{Gold} and Gross \cite{Gro} have constructed special Lagrangian fibrations using group actions.  However, it is important to emphasize here that the symplectic form is {\em not} the Ricci-flat symplectic form. 

Since Kontsevich's original proposal, mirror symmetry has been extended beyond the setting of Calabi-Yau manifolds thanks to the work Batryev \cite{Bat, Bat1}, Kontsevich \cite{Kont1}, Givental \cite{Giv1, Giv2, Giv3}, Hori-Vafa \cite{HoVa} and many others.  When $-K_{Y}\rightarrow Y$ is effective, the mirror can no longer be compactand instead is expected to be a Landau-Ginzburg model $(\check{Y}, W)$ consisting of a non-compact K\"ahler manifold $\check{Y}$ and a holomorphic function $W$ called the super-potential.  

A beautiful proposal of Auroux \cite{Aur1} suggests that when $-K_{Y}$ is effective, the mirror of $Y$ should be constructed by applying SYZ mirror symmetry to the complement $Y\setminus D$ where $D \in |-K_{Y}|$ is an anticanonical divisor; such pairs $(Y,D)$ are usually referred to as log Calabi-Yau manifolds.   Note that SYZ mirror symmetry makes sense since $Y\setminus D$ carries a non-vanishing holomorphic volume form with a simple pole along $D$.  In particular, by the SYZ proposal, $Y\setminus D$ should admit a special Lagrangian torus fibration and the mirror $(\check{Y}, W)$ should be constructed from $Y\setminus D$ by $T$-duality along the fibers.  Furthermore, the super potential $W$ is generated by the Floer theory of $Y\setminus D$.  Auroux's proposal is in part inspired by Seidel's ICM address \cite{SeidelICM} in which he explained how the Fukaya category of the complement of a hyperplane divisor in a projective Calabi-Yau manifold could be effectively understood. 

We note that, when $Y$ is a projective surface and $D$ a singular nodal curve,  Gross-Hacking-Keel \cite{GHK} constructed an algebraic mirror of $Y\setminus D$ using tropical techniques, along the lines of the Gross-Siebert program. We refer the reader to \cite{GSsurv} and the references therein for an introduction to this active area of research.
  
In this paper, motivated by Auroux's work, we study the SYZ proposal for log Calabi-Yau manifolds $(Y,D)$. Our main interest will be the existence of special Lagrangian submanifolds in $X= Y\setminus D$.  We will consider the following two cases
\begin{itemize}
\item[(I)] $Y$ is a Fano variety and $D \in |-K_{Y}|$ is a smooth divisor.
\item[(II)] $Y$ admits a fibration $\pi: Y\rightarrow B$ onto a smooth algebraic curve with connected fibers and $D\in |-K_{X}|$ is a smooth fiber of $\pi$.
\end{itemize}

For example, the second case can be achieved by blowing up the base points of a pencil of anti-canonical divisors in a Fano manifold. In each of these cases it is a fundamental result of Tian-Yau \cite{TY, TY2} that $X= Y\setminus D$ admits a complete Ricci-flat metric making $(X,\omega_{TY})$ a complete, non-compact Calabi-Yau manifold.  Our first main theorem is the following

\begin{thm}\label{thm: main1Intro}
Suppose $(Y,D)$ is a log Calabi-Yau manifold of Type I or II.  Suppose that the Calabi-Yau manifold $D$ admits a smooth, immersed special Lagrangian $L \subset D$.  Then $X=Y\setminus D$ equipped with the Tian-Yau metric admits a countable infinity of disjoint, immersed special Lagrangian submanifolds $\{\tilde{L}_{i}\}_{i\in \mathbb{N}}$ with topology $L\times S^1$ and having the following property: for each $i$ there is a sequence $j_i \rightarrow \infty$ such that $d(\tilde{L}_i, \tilde{L}_{j_i}) \rightarrow \infty$ as $j_i \rightarrow \infty$.
\end{thm}

\begin{rk}
The final property is meant to emphasize that the countable infinity of special Lagrangians are not small deformations of a single special Lagrangian.
\end{rk}

One way to view this result is as a lifting result for special Lagrangian submanifolds from Calabi-Yau manifolds of dimension $n-1$ to Calabi-Yau manifolds of dimension $n$.  For instance, there are several examples of projective Calabi-Yau manifolds admitting special Lagrangian submanifolds \cite{Bry, HS, MKob}, and it is well-known that elliptically fibered $K3$ surfaces can be hyper-K\"ahler rotated to produce special Lagrangian fibrations, and in some examples hyper-K\"ahler rotation remains algebraic.  By Theorem~\ref{thm: main1Intro}, we obtain non-compact Calabi-Yau threefolds admitting many distinct special Lagrangian submanifolds.  Conceivably, one could glue two such Calabi-Yau manifolds along infinity to obtain a compact Calabi-Yau manifold with finitely many distinct special Lagrangian submanfolds.  If one were exceedingly lucky, this construction could be repeated to obtain special Lagrangian submanifolds of Calabi-Yau manifolds in higher and higher dimensions.  We note Talbot \cite{Tal} has obtained gluing results of this type for non-compact, asymptotically cylindrical special Lagrangians in asymptotically cylindrical Calabi-Yau manifolds.

One case in which the existence of special Lagrangians in $D$ is trivial is when $Y$ has dimension $2$, so that $D$ is a torus.  In this case, we obtain the following result which confirms the SYZ conjecture in this case

\begin{thm}\label{thm: wtf}
Suppose that $Y$ is a del Pezzo surface or a rational elliptic surface and $D\in|-K_{Y}|$ is smooth.  Then $X=Y\setminus D$ equipped with the Tian-Yau metric admits a special Lagrangian fibration $\pi:X \rightarrow \mathbb{R}^2$ with a section.  Furthermore, near $\infty$, the fibers of $\pi$ are contained in a neighborhood of $D$ and the smooth fibers are topologically $S^1$ bundles over special Lagrangian submanifolds of $D$.
\end{thm}

Theorem~\ref{thm: wtf} resolves a conjecture of Auroux \cite[Conjecture 5.1]{Aur2} in complex dimension 2.  To our knowledge, this theorem produces the first examples of SYZ fibrations in the literature which are neither trivial, nor obtained from hyper-K\"ahler rotating from an existing holomorphic torus fibration (as in the case of $K3$ surfaces). As an application of this result we obtain

\begin{cor}[Auroux, Conjecture 2.9, \cite{Aur}]\label{cor: introAurConj1}
Let $D$ be a smooth cubic in $\mathbb{P}^2$.  Then $X= \mathbb{P}^2\setminus D$ admits a special Lagrangian fibration with respect to the Tian-Yau metric $\pi: X \rightarrow \mathbb{R}^2$.   Furthermore,  the fibration $\pi$ has $3$ singular fibers, each of which is a nodal special Lagrangian sphere; that is, of Kodaira type $I_1$.
\end{cor}

Corollary~\ref{cor: introAurConj1} resolves a conjecture of Auroux \cite[Conjecture 2.9]{Aur}.  Secondly, (in the Type II  case) we obtain the following corollary, which resolves another conjecture of Auroux

\begin{cor}[Auroux, Conjecture 2.10, \cite{Aur}]\label{cor: introAurConj2}
Let $Y$ be a rational elliptic surface and $D\in|-K_{Y}|$ a smooth divisor.  Then, for any choice of K\"ahler class $[\omega]$ on $Y$,  $X= Y\setminus D$ admits a special Lagrangian fibration $\pi: X \rightarrow \mathbb{R}^2$ with respect to the Tian-Yau metric.  For generic $(Y, [\omega], D)$, this fibration has $12$ singular fibers each of which is a nodal special Lagrangian sphere.
\end{cor}

Finally, we apply our results to mirror symmetry for del Pezzo surfaces and rational elliptic surfaces.  Before stating our result, let us explain the context.  At a homological level, mirror symmetry for del Pezzo surfaces and rational elliptic surfaces is quite well understood.  In this setting, Auroux-Katzarkov-Orlov \cite{AKO} proved one direction of the mirror correspondence, namely showing that the derived category of coherent sheaves on a del Pezzo surface $Y$ is equivalent to the derived category of vanishing cycles of a certain elliptic fibration,
\[
W : \check{X} \rightarrow \mathbb{C}.
\]
$W$ here is the superpotential of the Landau-Ginzburg mirror of the del Pezzo surface $Y$.   One of the key ideas in their work is that there is an elliptic fibration (in fact, a rational elliptic surface)
\[
\overline{W}: \check{Y} \rightarrow \mathbb{P}^1
\]
and that $\check{Y}\setminus \overline{W}^{-1}(\infty)$ is the fiber-wise compactification of $W : \check{X} \rightarrow \mathbb{C}$.  In fact, if $Y_{k}$ is the  del Pezzo surface obtained by blowing up $\mathbb{P}^2$ at $9-k$ points, then $ \overline{W}^{-1}(\infty)$ is a singular fiber of the elliptic fibration consisting of $k$ rational curves.  This correspondence is constructed by hand, exploiting the relative flexibility of the symplectic category.  Subsequently, Lunts-Przyjalkowski \cite{LunPr} showed that this construction gives mirror symmetry at the level of Hodge numbers, following a proposal by Katzarkov-Kontsevich-Pantev \cite{KKP}.  Using a different approach motivated by the Doran-Harder-Thompson conjecture \cite{DHT}, Doran-Thompson \cite{DT} showed that the mirror correspondence between del Pezzo complements and rational elliptic surfaces holds true at a lattice-theoretical level.

On a more historical note, it was originally thought that mirror symmetry for Calabi-Yau surfaces (and hyper-K\"ahler manifolds more generally) could be obtained by hyper-K\"ahler rotation \cite{BBS, BS, Huy}.  It is now understood that this is not the case in general.  Nevertheless, we prove

\begin{thm}\label{thm: main3Intro}
Let $Y$ be a del Pezzo surface or rational elliptic surface and $D \in |-K_{Y}|$ a smooth anticanonical divisor with $D^2=d$.  Let $X=Y\setminus D$ and equip $X$ with the Ricci-flat Tian-Yau metric $g_{TY}$.  Denote this complete non-compact Calabi-Yau manifold by $(X,g_{TY}, J)$.  Then, for any choice of homology class $[\gamma] \in H_1(D,\mathbb{Z})$ represented by a special Lagrangian, $(X,g_{TY}, J)$ admits a special Lagrangrian torus fibration $\pi: X\rightarrow \mathbb{R}^2$ with fibers topologically $S^1\times \gamma$.   We can perform a hyper-K\"ahler rotation to a complex structure $I$ so that the fibration $\pi: (X,g_{TY}, I) \rightarrow \mathbb{C}$ is holomorphic, with generic fiber an elliptic curve.  Furthermore, we have $(X,I) = \check{Y}\setminus \check{D}$, where $\check{\pi}:\check{Y}\rightarrow \mathbb{P}^1$ is a rational elliptic surface and $\check{D}$ is a singular fiber of $\check{\pi}$ of Kodaira type $I_d$. 
\end{thm}

It is important to remark that we do not know if the manifold $\check{Y}\setminus \check{D}$ is mirror in the sense of SYZ to $X$;  the correct mirror obtained by torus duality may have a different complex structure.  Nevertheless, it is in the correct family as suggested by the results of \cite{AKO, LunPr, KKP, DT}.   

Finally, we remark that Theorem~\ref{thm: main3Intro} in fact produces many inequivalent special Lagrangian fibrations on $X$.  Given the elliptic curve $D$ we can choose two distinct special Lagrangians $\gamma, \gamma'$  intersecting at one point and generating $H_1(D, \mathbb{Z})$.  Every such choice gives rise to a special Lagrangian fibration.  Since the mirror of $X$ is a rational elliptic surface, the existence of such {\em twin} special Lagrangian fibrations confirms a prediction of Leung-Yau \cite{LY, LLi}.

The paper is structured as follows.  In order to prove Theorem~\ref{thm: main1Intro} we use the explicit form of the Tian-Yau metrics near infinity to construct {\em approximate} special Lagrangians in the asymptotic geometry.  This construction proceeds in two steps.  Fix a point $x_0\in X$ and let $d(x_0,\cdot)$ be the distance to $x_0$ with respect to the Tian-Yau metric.  In the first step we construct explicit special Lagrangians $L_{R}$ in the model geometry to which the Tian-Yau metrics converge.  We find explicit bounds for the geometry of these special Lagrangians in terms of the parameter $R$, which roughly measures $d(x_0, L_{R})$.  Next, we transfer the special Lagrangians $L_{R}$ to {\em approximate} special Lagrangians $L'_{R}$ in the Tian-Yau manifolds, while maintaining precise control of the geometry of $L'_{R}$ in terms of $d(x_0, L'_{R})$.  

The second step is to run the Lagrangian mean curvature flow (LMCF) in order to deform $L_{R}'$ to a genuine special Lagrangian.  In the Type II case, the Tian-Yau metrics are asymptotically cylindrical and the geometry of the approximate Lagrangians $L'_{R}$ as well as the Tian-Yau metric is uniformly controlled near infinity.  This allows us to appeal to a theorem of Li \cite{Li} which in the current setting implies that for $R$ sufficiently large, the LMCF starting at $L_{R}'$ converges smoothly and exponentially fast to a special Lagrangian.  In the type I case, the situation is substantially more difficult, as the geometry of $L_{R}'$ as well as the Tian-Yau metric degenerates at  infinity.  Nevertheless, by exploiting the the precise control of the geometry achieved in the construction of $L_{R}'$ we prove that the LMCF converges smoothly and exponentially fast to a special Lagrangian submanifold; see Theorem~\ref{thm: TYLMCFconv}.  These results occupy Sections~\ref{sec: perturb},~\ref{sec: ACYL} and~\ref{sec: TY}.

Next, we focus on the surface case.  Using the deformation theory of special Lagrangians \cite{RMc}, together with the theory of $J$-holomorphic curves and a hyper-K\"ahler rotation trick, we show that the existence of two disjoint immersed special Lagrangians representing the same primitive homology class infers the existence of a special Lagrangian fibration.  Combining this result with Theorem~\ref{thm: main1Intro}, we obtain Theorem~\ref{thm: wtf}.  Finally, in Section~\ref{sec: mirror} we refine our results when $Y$ is a del Pezzo surface or a rational elliptic surface.  We prove Theorem~\ref{thm: main3Intro} as well as identify the (generic) singular fibers of the special Lagrangian fibrations, as predicted by Auroux's conjectures (see Corollaries~\ref{cor: introAurConj1} and~\ref{cor: introAurConj2}).
\\
\\
{\bf Acknowledgements}: The authors are grateful to D. Auroux and S.-T. Yau for their interest and encouragement.  The third author is grateful to P. Hacking and C. Doran for their interest and some helpful conversations.  The third author is also grateful to H.-J. Hein for explaining the work of \cite{HSVZ}.  We are also very grateful to the referees who provided many helpful comments and questions (in particular leading to Proposition~\ref{prop: MCFhomotope}) which have greatly improved the paper.

\section{Perturbation of Lagrangians}\label{sec: perturb}

In this section we collect together a few formulae for the variation of geometric quantities on a Lagrangian, or more generally a submanifold $M$, under variations in the Riemannian metric.  The primary application of these formulae will be in controlling the following perturbation problem.  Suppose $(X_{mod}, \omega_{mod}, J_{mod}, \Omega_{mod}, g_{mod})$ is a Calabi-Yau manifold (perhaps not complete) and suppose that $(X, \omega, J, \Omega, g)$ is a complete, non-compact Calabi-Yau manifold with one end.  Fix some point $p \in X$ and suppose that, for a large number $R<\infty$ there is a diffeomorphism $\Phi$ such that
 \[
X_{mod}\xlongrightarrow{\Phi}  \{ x \in X: d(p,x) >R\}  \subset X
\]
with the property that $\Phi^{*}\omega - \omega_{mod} = d\beta$ for some one form $\beta$.  Suppose that $M_{mod} \subset X_{mod}$ is a special Lagrangian.  The goal is to perturb $M_{mod}$ to a submanifold $M$ which is Lagrangian with respect to $\Phi^{*}\omega$, while maintaining control of the Riemannian geometry of $M$. 

The natural way to accomplish this is via Moser's trick. Define a time dependent family of symplectic forms  $\omega_t=(1-t)\omega_{mod}+t\Phi^*\omega$ for $t\in[0,1]$ and define the time-dependent vector field $V_t$ on $X_{mod}$ via  
\begin{equation}
i_{V_t}\omega_t=-\beta.\nonumber
\end{equation}
Let $F_t$ be the time-dependent diffeomorphism generated by the flow of $V_t$, i.e. defined by $\frac{dF_t}{dt}=V_t \circ F_t$ and $F_0=id$. Then
\begin{equation}
\frac{d}{dt}F_t^*\omega_t:=F_t^*\left(\frac{d}{dt}\omega_t+{\mathcal L}_{V_t}\omega_t\right).\nonumber
\end{equation}
Applying Cartan's formula and using that $\omega_t$ is closed, gives
\begin{equation}
\frac{d}{dt}F_t^*\omega_t:=F_t^*\left(\Phi^*\omega-\omega_{mod}+d(i_{V_t}\omega_t)\right)=F_t^*\left(d\beta+d(i_{V_t}\omega_t)\right)=0.\nonumber
\end{equation}
From this we conclude that $F_t^*\omega_t=\omega_{mod}$. Setting $t=1$ gives $F_1^*\Phi^*\omega=\omega_{mod}$. Thus, for any Lagrangian $M_{mod}$ with respect to $\omega_{mod}$, the image $F_1(M_{mod})$ will be a smooth, Lagrangian with respect to $\Phi^*\omega$. With  can now transplant $M_{mod}$ to a Lagrangian in $X$ by taking $M := \Phi (F_1(M_{mod}))$.  

To keep track of the Riemannian geometry throughout this process, we need to perturb the metrics.  Since the flow $F_{t}$ may not map $J_{mod}$ to $\Phi^{*}J$, the Riemannian structure is not naturally inherited from the flow of symplectic forms.  Instead, we will consider the one parameter family of metrics $\tilde{g}_{t} = (1-t)g_{mod} + t \Phi^{*}g$ for $t\in[0,1]$.  Note that the geometry of $F_{t}(L_{mod})$ as a submanifold of $(X_{mod}, \tilde{g}_{t})$ is just the same as the geometry of $L_{mod}$ as a submanifold of $X_{mod}$ equipped with the one-parameter metric $g_{t} = F_{t}^{*}\tilde{g}_{t}$, for $t\in [0,1]$.  In particular, we are essentially reduced to understanding how various geometric quantities vary under changes in the metric .   We begin with a simple lemma.

\begin{lem}\label{lem: perturbNorm}
Let $X$ be a Riemannian manifold.  For $t \in [0,1]$ consider time dependent Riemannian metrics $g_t$ and let $V_t$ be a time dependent vector field.  Let $F_{t}$ be the diffeomorphism of $X$ defined by $\ddt F_{t}(p) = V_{t}(F_{t}(p))$, with $F_{0}(p)= p$.  Then we have
\[
\ddt F_{t}^{*}g_t = F_{t}^{*}\left( \mathcal{L}_{V_{t}}g_t + \ddt g_t\right).
\]
In particular, we have
\[
\bigg|\ddt F_{t}^{*}g_t\bigg|_{F_{t}^{*}g_{t}} \leq 2\big|\nabla^{t}V_{t}\big|_{g_{t}} + \bigg|\ddt g_t\bigg|_{g_{t}},
\]
where $\nabla^{t}$ denotes the covariant derivative with respect to $g_t$.
\end{lem}
\begin{proof}
The formula for $\ddt F_{t}^{*}g_t$ is a straightforward computation.  The estimate follows from the formula $\mathcal{L}_{V_{t}}g_t = 2g_t(\nabla^t_{\cdot} V_{t}, \cdot)$ and the observation that, for any tensor $T$ we have $|F_{t}^{*}T|_{F_{t}^{*}g_{t}} = |T|_{g_{t}}$.
\end{proof} 

Next we consider the variation of the second fundamental form of a submanifold $M\subset X$.

\begin{lem}\label{lem: perturbAH}
Let $M^{k}\subset X^{n+k}$ be a submanifold and let $g_{t}$ be a family of Riemannian metrics on $X$ for $t\in (-\epsilon,\epsilon)$.  Let $A_{t}$ denote the second fundamental form of $M$ in $X$.  Then we have
\[
\begin{aligned}
\bigg|\ddt |A_t|^2_{g_{t}} \bigg| &\leq 10\left( |\nabla^{t} \del_tg_{t}|_{g_{t}}|A_{t}|_{g_{t}} + |\del_tg_{t}|_{g_{t}} |A_{t}|^2_{g_{t}}\right),\\
\bigg|\ddt |H_t|^2_{g_{t}} \bigg| &\leq 10\left( |\nabla^{t} \del_t{g}_{t}|_{g_{t}}|H_{t}|_{g_{t}} + |\del_t g_{t}|_{g_{t}} |H_{t}|^2_{g_{t}}\right).
\end{aligned}
\]
\end{lem}
\begin{proof} 
 
The lemma follows immediately from the variational formula for the second fundamental form.  We refer the reader to the appendix and Lemma~\ref{lem: secFunVar} for a complete proof.  Let $\overline{X} = X\times (-\epsilon, \epsilon)$ and let $\bar{g} = g(t) +dt^2$.  If $\overline{\nabla}$ denotes the covariant derivative of $\bar{g}$ on $\overline{X}$, then Lemma~\ref{lem: secFunVar} gives
\[
\bigg|\ddt |A|^2_{g_t} \bigg|=2\bigg|\langle A, \overline{\nabla} _{\del_t} A\rangle\bigg| \leq 10\left(|A|_{g_t}|\nabla^t \del_t g|_{g_t} + |A|_{g_t}^2|\del_t g|_{g_t}\right)
\]
and similarly
\[
\bigg|\ddt |H|^2_{g_t}\bigg|\leq 10\left(|H|_{g_t}|\nabla^t \del_t g|_{g_t} + |H|_{g_t}^2|\del_t g|_{g_t}\right),
\]
which is the desired result.
\end{proof}

Finally we examine how the first non-zero eigenfunction of the Laplacian changes under a change in the metric.

\begin{lem}\label{lem: perturbLambda}
Let $M$ be a compact Riemannian manifold and let $g_t$ be a smooth family of Riemannian metrics for $t\in (-\epsilon, \epsilon)$.  Let $\lambda_1(t)$ denote the first non-zero eigenvalue of the Laplacian on $(M, g_t)$.  Then we have
\[
e^{-3\mu(t)}\lambda_1(0) \leq  \lambda_1(t) \leq e^{3\mu(t)}\lambda_1(0)
\]
where $\mu(t) = \int_0^t \sup_{M}|\del_sg_s|_{g_s} ds$.
\end{lem}
\begin{proof}
We define the eigenvalues of the Laplacian by $\Delta f + \lambda f=0$.  Recall that the first non-zero eigenvalue of the Laplacian on $M$ is given by the Rayleigh quotient
\[
\lambda_1(t) = \inf \frac{\int_{M} |\nabla^t f_t|^2_{g_t} dVol_{t}}{\int_{M} f_t^2 dVol_t},
\]
where the infimum is taken over all smooth functions with $\int_{M} f dVol_t =0$.  Given a function $f$ such that $\int_{M} f dVol_0 =0$, define
\[
f_t= f- \frac{1}{{\rm Vol}(M,g_t)} \int_{M} f d Vol_{t}
\]
and note that $df_t = df$ for all $t$.  Then
\[
\ddt \int_{M} |\nabla^t f_t|^2_{g_t} dVol_{t} = \int_{M}\bigg(\frac{1}{2}\left( g_{t}^{ij} \del_t g_{ij}\right) \cdot |\nabla^t f_t|^2_{g_t} - \langle \del_t g, df_t \otimes df_t \rangle_{g_t}\bigg) dVol_{t}
\]
and so
\[
\bigg|\ddt \int_{M} |\nabla^t f|^2_{g_t} dVol_{t}\bigg| \leq 2\sup_{M}|\del_t g|_{g_t}\left(\int_{M} |\nabla^t f_t|^2_{g_{t}}dVol_t\right).
\]
Similarly, we have
\[
\begin{aligned}
\ddt \int_{M} f_t^2 dVol_t &= \frac{1}{2}\int_{M} f_{t}^2 \left( g_{t}^{ij} \del_t g_{ij}\right) dVol_t + 2 \ddt{f}_t \int_{M} f_t dVol_{t}\\
& =  \frac{1}{2}\int_{M} f_{t}^2 \left( g_{t}^{ij} \del_t g_{ij}\right) dVol_t,
\end{aligned}
\]
where we used that $\ddt f_t$ is constant on $M$ and $\int_{M} f_t dVol_t =0$.  Therefore
\[
\bigg|\ddt \int_{M} f_t^2 dVol_t \bigg| \leq \sup_{M}|\del_tg_t|_{g_{t}} \int_{M} f_t^2 dVol_t.
\]
Define 
\[
\mu(t) = \int_{0}^{t}\sup_{M}|\del_sg_s|_{g_{s}} ds,
\]
and denote the quotient by $F(t):=\frac{\int_{M} |\nabla^t f_t|^2_{g_t} dVol_{t}}{\int_{M} f_t^2 dVol_t}$.
The above inequalities imply the function $e^{-3\mu(t)}F(t)$ is non-increasing in time, while $e^{3\mu(t)}F(t)$ is non-decreasing. Thus, we conclude
\[
e^{-3\mu(t)} F(0) \leq F(t)\leq e^{3\mu(t)}F(0),
\]
from which the result follows.
\end{proof}

\section{Special Lagrangians in Asymptotically Cylindrical Calabi-Yau manifolds}\label{sec: ACYL}

We first turn to the case of asymptotically cylindrical Calabi-Yau manifolds and prove existence of infinitely many special Lagrangian submanifolds assuming the existence of one special Lagrangian in the asymptotic cross section. This case is much simpler than what we prove in the subsequent section, although the basic idea is similar.

\begin{defn}
A complete Riemannian manifold $(X,g)$ is called {\em asymptotically cylindrical} (ACyl) if there exists a compact subset $V\subset X$, a closed Riemannian manifold $(Y,h)$ and a diffeomorphism $\Phi:{\mathbb R}^+\times Y\rightarrow X\setminus V$ satisfying
\begin{equation}
\label{metricdecay}
 |\nabla^k_{g_\infty}(\Phi^*g-g_\infty)|_{g_{\infty}}=O(e^{-\delta \ell})
\end{equation}
for some $\delta>0$ and all $k\in{\mathbb N}$. The limiting metric is given by $g_\infty=d\ell^2\oplus h$, where $\ell$ is the coordinate on $\mathbb R^+$.
\end{defn}

Assume that $(X,\omega,J,\Omega,g)$ is a simply-connected, irreducible ACyl Calabi-Yau manifold and thus Ricci-flat .  By the Cheeger-Gromoll splitting theorem, unless $X$ is a product cylinder, it can only have a single cylindrical end. 

\subsection{The model cylindrical geometry}
We discuss the geometry of the limiting cylindrical end and construct a one-parameter family of special Lagrangians in this model space. Let $n={\rm dim}_{\mathbb C}X$ and first assume $n>2$. In this case Theorem B of \cite{HHN} applies:
\begin{thm}[Haskins-Hein-Nordstrom, \cite{HHN}]
\label{HHNB}
Let $X$ be a simply-connected irreducible ACyl Calabi-Yau manifold with $n  > 2$. There exists a compact Calabi-Yau manifold $D$ with a K\"ahler isometry $\iota$ of finite order $m$ such
that the cross-section $Y$ of $X$ can be written as $Y = ({S}^1\times D)/\iota$, where $\iota$ acts on the product
via $\iota(\theta,x)=(\theta+\frac{2\pi}{m},\iota(x))$. Moreover, $\iota$ preserves the holomorphic volume form $\Omega_{D}$ on $D$.
\end{thm}
To construct the model cylinder, first consider  $\ti X_\infty:={\mathbb R^+}\times{S}^1\times D$ with the product complex structure $J_\infty$ and Hermitian metric $g_\infty=d\ell^2+d\theta^2+g_D$. Here $g_D$ is a Ricci-flat  metric on $D$, $\theta\in{S}^1$ and $J_{\infty}(\frac{\partial}{\partial \ell})=\frac{\partial}{\partial\theta}$. The associated K\"ahler form and  holomorphic $(n,0)$-form are given by 
\begin{equation}
\omega_\infty=d\ell\wedge d\theta+\omega_D\qquad{\rm and}\qquad \Omega_\infty=(d\ell+id\theta)\wedge \Omega_{D} \nonumber
\end{equation} 
respectively.  As in Theorem \ref{HHNB}, the action of $\iota$ on $D$ extends to ${S}^1\times D$,
and furthermore $\omega_\infty$ and $\Omega_\infty$ are fixed under this action. Thus both forms descend to the smooth K\"ahler manifold
\begin{equation}
X_\infty:={\mathbb R^+}\times({S}^1\times D)/\iota,\nonumber
\end{equation}
which serves as the cylindrical model for the end of $X$.

We now construct our one-parameter family of special Lagrangians in $X_\infty$. Assume $D$ admits a special Lagrangian submanifold $N\subset D$ which does not contain any fixed points of   $\iota$. Because $\iota$ is an isometry which preserves $\Omega_{D}$,    $\iota^k(N)$ is a special Lagrangian for $1\leq k\leq m-1$. For any $\rho\in\mathbb R^+,$ consider the following union, which consists of $m$ submanifolds with boundary sitting inside $\ti X_\infty$:
\begin{equation}
\ti M_{\rho}:=\bigcup_{k=1}^m\{\rho\}\times\iota^{k-1}\left(\left[0,\frac{2\pi}m\right]\times N\right).\nonumber
\end{equation}
$\ti M_{\rho}$ descends to a smooth submanifold of $X_\infty$, which we denote by $M_\rho$. The following lemma is immediately clear from our construction.
\begin{lem}
$M_\rho$ is isometric to $S^1\times N$ with the product metric, where the $S^1$ factor has length $\frac {2\pi}m$.  In particular, if
$N$ is a torus, then so is $M_\rho$.
\end{lem}

Let $\{\frac{\partial}{\partial x^i}\}_{i=1}^{n-1}$ form a basis for $TN$. At any point $p\in  M_\rho$ the tangent space $T_pM_\rho$ is spanned by the vectors $\{\frac{\partial}{\partial \theta},\frac{\partial}{\partial x^1},...,\frac{\partial}{\partial x^{n-1}}\}$. Because we assumed $N\subset D$ is Lagrangian, it is clear that $f|_{M_\rho}=0$ as well. Furthermore
\begin{equation}
\Omega_\infty|_{M_\rho}=i d\theta\left(\frac{\partial}{\partial \theta}\right)\Omega_{D}\left(\frac{\partial}{\partial x^1},...,\frac{\partial}{\partial x^{n-1}}\right)=i\Omega_{D}|_{N},\nonumber
\end{equation}
which is constant if $N\subset D$ is in addition a special Lagrangian. Additionally, the induced metric on $M_\rho$ is given by
\begin{equation}
g_{M_\rho}=d\theta^2+g_D|_{N}.\nonumber
\end{equation}
In particular we note that this metric is independent of $\rho$. This gives:
\begin{lem}
\label{cylindricalmodel}
Assume $N\subset D$ is a special Lagrangian which does not contain fixed points of $\iota$. Then $M_\rho$ is a special Lagrangian submanifold of $(X_\infty,\omega_\infty,\Omega_\infty)$ for all $\rho$.
Furthermore, the induced metric on $M_\rho$ is independent of $\rho\in\mathbb R^+$ and thus the geometric quantities associated to $M_\rho$, including the second fundamental form and the first eigenvalue of the Laplacian, 
depend only on $N$ and are independent of $\rho$.
\end{lem}

\subsection{Perturbation} 
Given our one-parameter family of special Lagrangians in the model space, we demonstrate how to perturb this family into a family of Lagrangians with respect to the ACyl metric. By assumption, our Calabi-Yau manifold $(X,\omega,J,\Omega,g)$ comes equipped with a cylindrical diffeomorphism  $\Phi:X_\infty\rightarrow X\setminus V$ for a given compact subset $V\subset X$. Pulling back via $\Phi$ allows us to work on the half cylinder $X_\infty$, where in addition to \eqref{metricdecay} we have the following decay for all $k\geq 0$: 
\begin{equation}
\label{formdecay}
|\nabla^k_{g_\infty}(\Phi^*\omega-\omega_\infty)|_{g_\infty}=O(e^{-\delta\ell})
\end{equation}
\begin{equation}
\label{Jdecay}
|\nabla^k_{g_\infty}(\Phi^*J-J_\infty)|_{g_\infty}=O(e^{-\delta\ell})
\end{equation}
\begin{equation}
\label{holvoldecay}
|\nabla^k_{g_\infty}(\Phi^*\Omega-\Omega_\infty)|_{g_\infty}=O(e^{-\delta\ell}).
\end{equation}

We will apply the perturbation results from Section \ref{sec: perturb}. First, we demonstrate that the two K\"ahler forms $\Phi^*\omega$ and $\omega_\infty$, are cohomologous, using an argument from \cite{HHN}.

Denote the difference in K\"ahler forms by $\eta:=\omega_\infty-\Phi^*\omega$ and decompose this two form as $\eta=d\ell\wedge \eta_1+\eta_2$. Since  $\Phi^*\omega$ and $\omega_\infty$ are closed, $d\eta=0$ and so
\begin{equation}
d\eta_2=d\ell\wedge \ti d\eta_1,\nonumber
\end{equation}
where $\ti d$ denotes the differential on the cross section $({S}^1\times D)/\iota$ only. 
Since the right hand side above contains $d\ell$, the exterior derivative of $\eta_2$ must be of the form $d\eta_2=d\ell\wedge\frac{\partial}{\partial \ell}\eta_2$, which implies that $\ti d\eta_1=\frac{\partial}{\partial \ell}\eta_2.$ 

Define a one form $\tau$ on $X_\infty$ by integrating $\eta_1$ in the $\ell$ direction:
\begin{equation}
\tau(\ell,p)=-\int_{\ell}^\infty\eta_1(s,p)ds,\nonumber
\end{equation}
which is finite by \eqref{formdecay}. Taking the exterior derivative gives
\begin{equation}
d\tau=d\ell\wedge\eta_1-\int_{\ell}^\infty \ti d\eta_1ds=d\ell\wedge\eta_1-\int_{\ell}^\infty \frac{\partial}{\partial \ell}\eta_2ds=d\ell\wedge\eta_1+\eta_2=\eta,\nonumber
\end{equation}
Thus $d\tau=\omega_\infty-\Phi^*\omega$ and $\tau$ is determined by the initial data.

Next we employ the standard Moser trick from Section~\ref{sec: perturb}, setting  $\omega_{mod}$ as $\omega_\infty$.  This gives a family of diffeomorphisms $F_t$ of $X_\infty$  satisfying $F_t^*\omega_t=\omega_\infty$. Setting $t=1$ gives $F_1^*\Phi^*\omega=\omega_\infty$. Thus, for any Lagrangian $M_\rho$ with respect to $\omega_\infty$, the image $F_1(M_\rho)$ will be Lagrangian with respect to $\Phi^*\omega$. In addition to being a Lagrangian, we have the following control the geometry of $\hat M_{\rho}:=F_1(M_{\rho})$.
\begin{prop}
\label{perturbedcyl}
For any $\rho>0$, the submanifold $\hat M_{\rho}:=F_1(M_{\rho})\subset X_\infty$ is a Lagrangian with respect to $\Phi^*\omega$ with vanishing Maslov class. Furthermore, for $\rho$ sufficiently large, there are uniform constants $C>2$ and $\delta_0>0$, depending on $N$, \eqref{metricdecay},\eqref{formdecay},\eqref{Jdecay} and \eqref{holvoldecay}, so that:
\begin{enumerate}
\item The coordinate $\ell$ on ${\mathbb R}^+$ restricted to $\hat M_{\rho}$ satisfies
$$\rho-C < \ell|_{\hat M_{\rho}} < \rho+C, $$
\item The second fundamental form satisfies
$$|A|^2 \leq C,$$
\item The mean curvature satisfies
$$|H|^2 \leq Ce^{-\delta_0\rho},$$
\item The volume satisfies 
$$(1-C^{-1}){\rm Vol}_{g_{D}}(N) \leq {\rm Vol}_{\Phi^*g}(\hat M_{\rho}) \leq (1+C^{-1}){\rm Vol}_{g_{D}}(N),$$
\item The first positive eigenvalue $\lambda_1(\hat M_{\rho})$ satisfies
$$C^{-1}\lambda_1(N) \leq \lambda_1(\hat M_{\rho}) \leq C\lambda_1(N).$$
\end{enumerate}

\end{prop}

\begin{proof}
  We have already seen that $\hat M_{\rho}$ is Lagrangian with respect to $\Phi^*\omega$. Furthermore, because $\hat M_{\rho}$ is homotopic to $M_{\rho}$ via a one-parameter family of diffeomorphisms,  the Maslov class of $\hat M_{\rho}$ vanishes.
  
Since $V_t$ is uniquely determined by $\tau$, it depends only on $\Phi^*\omega-\omega_\infty$. Thus, by \eqref{formdecay} we see
\begin{equation}
|\nabla^k_{g_\infty}V_t|_{g_\infty}=O(e^{-\delta\ell}).\nonumber
\end{equation}
Furthermore, given the Riemannian metrics $g_t=(1-t)g_\infty+t\Phi^*g$, \eqref{metricdecay} implies
\begin{equation}
\frac12 g_\infty\leq g_t\leq 2g_\infty,\qquad |\nabla^k_{g_\infty}g_t|_{g_\infty}=O(e^{-\delta\ell}).\nonumber
\end{equation}
Putting these together gives
\begin{equation}
\label{Vtbound}
|\nabla^k_{g_t}V_t|_{g_t}=O(e^{-\delta\ell}).
\end{equation}
Now, by definition we have the restriction $\ell|_{M_{\rho}}=\rho$. Our control of  $V_t$ bounds how far the points in $M_{\rho}$ can move by the diffeomorphism $F_1$, proving $(1)$. 

To bound the remaining quantities, we rely on our analysis from Section~\ref{sec: perturb}. First we see control of  $(F_{t}(M_{\rho}), g_{t})$ follows from  control of the geometry of  $M_{\rho} \subset X_\infty$ with respect to the Riemannian metrics  $\tilde{g}_{t} := F_{t}^{*}g_{t}$.  The above bound  \eqref{Vtbound}, together with Lemma~\ref{lem: perturbNorm}, gives
\begin{equation}\label{eq: perturbEstCylmod}
\begin{aligned}
|\del_t \tilde{g}_{t}|_{\tilde{g_{t}}} &\leq C_0 e^{-\delta\ell}\\
|\nabla_{\tilde{g}_{t}}\del_{t}\tilde{g}_{t}|_{\tilde{g_{t}}}& \leq C_0 e^{-\delta\ell}.
\end{aligned}
\end{equation}
Turning to the second fundamental form and the mean curvature of $ M_{\rho}$ with respect to $\tilde{g}_{t}$, we consider the ODE
\begin{equation}\label{eq: perturbVarEqCyl}
\frac{df}{dt} = c(f^{\frac{1}{2}} + f), \qquad f(0)\geq 0,
\end{equation}
whose solution is $f(t) = \left(-1+ [1+f(0)^{\frac{1}{2}}]e^{\frac{ct}{2}}\right)^2$.  By Lemma~\ref{lem: perturbAH}, equation~\eqref{eq: perturbEstCylmod} and part $(1)$ of the proposition, both $|H|^2(t)$ and $|A|^2(t)$ are subsolutions of ~\eqref{eq: perturbVarEqCyl} with constant $c= C'e^{-\delta'\rho}$ where $C', \delta' >0$ are uniform constants.  For $\rho$ sufficiently large depending only on $C', \delta',$ we obtain
\begin{eqnarray}
|A|^2(t) &\leq 100(C')^2e^{-2\delta'\rho}+ 4|A|^2(0)  \nonumber \\
|H|^2(t) &\leq  100(C')^2e^{-2\delta'\rho} + 4|H|^2(0).\nonumber
\end{eqnarray}
Lemma \ref{cylindricalmodel} describes the geometry of $M_{\rho}$ with the metric induced  by $\tilde{g}_0$, from which we see $|A|(0)$ is controlled by the supremum of the second fundamental form of the compact Lagrangian $N\subset D$. This establishes $(2)$. Also, since $M_{\rho}$ is minimal with respect to $\tilde{g}_0$ we have $|H|(0)=0$, establishing $(3)$.

Estimate $(4)$ follows immediately from~\eqref{eq: perturbEstCylmod}, while $(5)$ follows from~\eqref{eq: perturbEstCylmod} and Lemma~\ref{lem: perturbLambda}, again using our understanding of the geometry of  $M_{\rho}$ for $t=0$ given by Lemma \ref{cylindricalmodel}. This completes the proof.
\end{proof}

\subsection{The mean curvature flow} 
We now evolve the perturbed Lagrangian $\hat M_{\rho}$ by mean curvature flow. We apply the following result of  Li \cite{Li}, based on work of Chen-Li in the K\"ahler-Ricci flow \cite{ChLi}:

\begin{thm}[Li \cite{Li}] \label{Liconverge} Let $(X, g)$ be a complete K\"ahler-Einstein manifold with  scalar curvature $\overline{R}\geq 0$ and $M$ a compact Lagrangian submanifold smoothly immersed in $X$ with vanishing Maslov class. For any $V_0,\Lambda_0,$ and $\delta_0>0$, there exists an $\epsilon>0$, depending on $V_0,\Lambda_0,\delta_0,\overline{R}$,  a lower bound for the injectivity radius of $X$ and an upper bound for $\sum_{k=0}^5 |\nabla^kRm|_{g}$  so that if 
\begin{equation}
\lambda_1\geq\frac{\overline R}{2n}+\delta_0, {\rm Vol}(M)\leq V_0, |A|\leq \Lambda_0, \int_M|H|^2\leq \epsilon,\nonumber
\end{equation}
then the Lagrangian mean curvature flow with initial data $M$ will converge exponentially fast to a minimal Lagrangian submanifold in $X$. 
\end{thm}

To apply the above result to the Lagrangian $\hat M_{\rho}$, note that the constants $\delta_0$, $V_0$ and $ \Lambda_0$ can be specified by Proposition \ref{perturbedcyl}. Asymptotic decay \eqref{metricdecay} gives the desired control of background Riemannian curvature tensor and also allows us to bound the injectivity radius of $(X_\infty,\Phi^*g)$  below by the injectivity radius of the cross section $({S}^1\times D)/\iota$. All of these quantities are independent of $\rho$ and thus using $(3)$ from Proposition \ref{perturbedcyl}, for any $\epsilon>0$ we can choose $\rho$ large enough so
$$|H|^2 \leq Ce^{-\delta_0\rho}<\epsilon.$$
The hypothesis of Theorem \ref{Liconverge} now apply and as a result the mean curvature flow beginning with $\hat M_{\rho}$ converges to a special Lagrangian submanifold, which we denote by $\mathcal M_{\rho}$.

We now show there is a countable family of special Lagrangians. Let $C$ be the fixed constant from  Proposition \ref{perturbedcyl}. Fix $\rho$ so $\hat M_{\rho}$ converges to $\mathcal M_{\rho}$ along the mean curvature flow. Now, conclusion $(1)$ from Proposition \ref{perturbedcyl}  allows us to choose $\rho_1>\rho+2C+1$ large enough so that the corresponding perturbed Lagrangians  $\hat M_{\rho}$ and  $\hat M_{\rho_1}$ are at least distance $1$ apart. To see that they stay distinct along the flow, we only need to observe that    the mean curvature vector decays exponentially along the flow. Specifically,  Lemma 5.2 in \cite{Li} gives
$$ |H(t)|\leq Ce^{\frac{-\delta_0\rho}{n+2}}e^{\frac{-\beta_0}{2n+4}t},$$
 which controls how far each Lagrangian can travel. Thus, for $\rho$ large enough, we can  construct a sequence $\rho_{i+1}=\rho_{i}+2(C+1)$, so that the corresponding  limiting special Lagrangians $\mathcal M_{\rho_i}$ are distinct.

We conclude this section with a note about the case $n=2$.  Recall that Theorem~\ref{HHNB} stipulates $n>2$, but this is only required to obtain a compactification; see \cite[Remark 1.6]{HHN}).  However, in our setting we are not working on an arbitrary ACyl Calabi-Yau manifold.  Rather, we consider $Y$ a compact K\"ahler surface which admits a holomorphic fibration and let $X=Y\setminus D$ be the complement of a smooth fiber in $|-K_Y|$. The complete metric $\omega$ constructed by Tian-Yau \cite{TY} (see also \cite{Hein, HHN}) on $X$ is ACyl with cross section $\mathbb T^3=S^1\times \mathbb T^2$, equipped with a flat metric $h$ of the form 
 $$h=\gamma d\theta^2+g_{\epsilon,\tau_0}.$$
 Here $\theta\in S^1$, $\gamma$ is a fixed length and $g_{\epsilon,\tau_0}$ is the unique flat metric of area $\epsilon$ and modulus $\tau_0$ on $\mathbb T^2$. Our above construction of a special Lagrangian can now be carried out, with the model Lagrangian $M_\rho$ given by the product of the   $S^1$ factor with any line of rational slope in $\mathbb T^2$. Thus 
 
 \begin{cor}
Let $Y$ be a rational elliptic surface and $D\in |-K_{Y}|$ a smooth divisor.  Then, for any closed loop $[\gamma] \in H_{1}(D,\mathbb{Z})$, the non-compact Calabi-Yau manifold $Y\setminus D$ admits infinitely many special Lagrangian submanifolds which are topologically $[\gamma] \times S^1$.
\end{cor}

\section{Special Lagrangians in Tian-Yau spaces of Type I}\label{sec: TY}

In this section we prove the existence of infinitely many special Lagrangian submanifolds in Tian-Yau spaces, under the assumption that the Calabi-Yau manifold at infinity admits one smooth special Lagrangian.  We begin by reviewing the model geometry for the Tian-Yau spaces.

\subsection{The Calabi model geometry}\label{sec: TYmodLag}

Suppose that $D$ is a projective Calabi-Yau manifold of dimension $n-1$, with $K_{D} = \mathcal{O}_{D}$ and let $\pi : L\rightarrow D$ be an ample line bundle.  By Yau's Theorem \cite{Y} we can find a hermitian metric $h$ on $L$, unique up to scaling, so that $\omega_{D} = -\ddb\log (h)$ defines a Ricci-flat K\"ahler metric.  Consider the open $n$ dimensional complex manifold
\[
\mathcal{C} = \{ \xi \in L : 0<|\xi|_{h} <1 \}.
\]
The space $\mathcal{C}$ has a natural, non-vanishing holomorphic $(n,0)$-form induced in the following way.  Fix local holomorphic coordinates $(z_1,\ldots,z_{n-1})$ on $D$ and a local trivialization $\xi$ of $L$.  We get coordinates $(z_1,\ldots,z_{n-1}, w)$ on $L$ by
\[
(z_1,\ldots,z_{n-1}, w) \mapsto (z_1,\ldots,z_{n-1}, w\xi).
\]
Let $\Omega_{D}$ be a holomorphic volume form on $D$.  We fix a scale for $\Omega_{D}$ by requiring that
\begin{equation}\label{eq: normOmegaD}
\frac{1}{2}\int_{D}(\sqrt{-1})^{(n-1)^2} \Omega_{D}\wedge\overline{\Omega}_{D} = \int_{D}(2\pi c_1(L))^{n-1}.
\end{equation}
$\Omega_{D}$ can be locally written as $f(z)dz_1\wedge\cdots\wedge dz_{n-1}$.  Then
\[
\Omega_{\mathcal{C}} = \frac{f(z)}{w}dz_1\wedge\cdots\wedge dz_{n-1}\wedge dw
\]
is a local, non-vanishing holomorphic $(n,0)$-form on $\mathcal{C}$.  It is easy to see that this expression is independent of the choice of trivialization of $L$ and hence $\Omega_{\mathcal{C}}$ glues to a trivialization of $K_{\mathcal{C}}$. We will assume that $\Omega_D$ is normalized such that 
   \begin{align} \label{normalization of omega_D}
       \omega_D^{n-1}=\frac{1}{2}(\sqrt{-1})^{(n-1)^2}\Omega_D\wedge \bar{\Omega}_D. 
   \end{align}

Define a function on $\mathcal{C}$ by
\[
\cC \ni \xi \longmapsto \frac{n}{n+1}(-\log |\xi|^{2}_{h})^{\frac{n+1}{n}}
\]
and let
\[
\omega_{\cC} = \ddb \frac{n}{n+1}(-\log |\xi|^{2}_{h})^{\frac{n+1}{n}}.
\]
By direct computation one can verify that $\omega_{\cC}$ defines a K\"ahler metric on $\cC$, which is complete as $|\xi|_{h} \rightarrow 0$, but incomplete as $|\xi|_{h}\rightarrow 1$.  Fix a point $p\in D$ and let $(z_1,\ldots,z_{n-1})$ be coordinates centered at $p$.  Choose a trivialization of $L$ so that $h(p)=1$, $dh(p)=0$ and write $h= e^{-\phi}$ for a locally defined function $\phi$.  If we write $w=re^{\sqrt{-1}\theta}$, then the K\"ahler form and the Riemannian metric are given by
\begin{equation}\label{eq: calModSymp}
\begin{aligned}
\omega_{\cC} &= \sqrt{-1}\big(-\log(|w|^2e^{-\phi})\big)^{\frac{1}{n}-1} \frac{1}{n}\left(\frac{dw}{w}+ \del \phi\right)\wedge \left(\frac{d\bar{w}}{\bar{w}} + \dbar\phi\right)\\
&\quad + \big(-\log(|w|^2e^{-\phi})\big)^{\frac{1}{n}} \pi^{*}\omega_{D},\\
g_{\cC} &= \big(-\log(r^2e^{-\phi})\big)^{\frac{1}{n}-1} \frac{1}{n}\left((\frac{dr}{r} - \frac{1}{2}d\phi)^2 + (d\theta + \frac{1}{2}Jd\phi)^2 \right)\\
&\quad + \big(-\log(r^2e^{-\phi})\big)^{\frac{1}{n}} \pi^{*}g_{D}.
\end{aligned}
\end{equation}
At any point $(0,\ldots, 0, re^{\sqrt{-1}\theta})$ this reduces to
\begin{equation}\label{eq: calModSympPoint}
\begin{aligned}
\omega_{C} &= \sqrt{-1}\frac{\big(-\log(|w|^2)\big)^{\frac{1}{n}-1}}{n}\frac{dw\wedge d\bar{w}}{|w|^2} + (-\log(|w|^2))^{\frac{1}{n}} \pi^{*}\omega_{D}\\
g_{C} &= \big(-\log(r^2)\big)^{\frac{1}{n}-1} \frac{1}{n}\left(\frac{dr^2}{r^2} + d\theta^2\right) + \big(-\log(r^2)\big)^{\frac{1}{n}} \pi^{*}g_{D}.
\end{aligned}
\end{equation}
Completeness as $r\rightarrow 0$ easily follows from this formula.  Furthermore, from \eqref{normalization of omega_D}, we have 
\[
\omega_{\cC}^{n} = \sqrt{-1}\frac{dw\wedge d\bar{w}}{|w|^2}\wedge \pi^{*}\omega_{D}^{n-1} = \frac{1}{2}(\sqrt{-1})^{n^2}\Omega_{\cC}\wedge \overline{\Omega_{\cC}},
\]
so $\omega_{\cC}$ is Ricci-flat.  Let us introduce the following terminology;

\begin{defn}\label{def: scaleFunc}
We define the {\em scale function} on $\cC$ to be
\[
\ell_0= (-\log|\xi|^2_{h})^{\frac{1}{2n}}.
\]
\end{defn}

\begin{rk}\label{rek: scaleRk}
Note that the scale function satisfies
\[
|\nabla_{g_{\cC}}\ell_0^{n+1}|^2 = \frac{(n+1)^2}{4n}.
\]
Furthermore, if $p \in \cC$ is a fixed point there there is a constant $C_1$ so that
\[
C_{1}^{-1} \ell_0^{n+1} \leq d_{\cC}(p, x) \leq C_{1}\ell_0^{n+1}.
\]
\end{rk}

By direct computation we have
\begin{prop}\label{prop: calModDec}
Let $g_{\cC}$ be the Riemannian metric on $\cC$ induced by $\omega_{\cC}$ and suppose that $n\geq 3$.  Then, for all $k \in \mathbb{N}\cup \{0\}$ there is a constant $C_k$ so that
\[
|\nabla^{k}Rm(g_{\cC})|_{g_{\cC}} \leq C_{k}\ell_0^{-(k+2)}.
\]
If $n=2$, then we have
\[
|\nabla^{k}Rm(g_{\cC})|_{g_{\cC}} \leq C_{k}\ell_0^{-(k+6)}.
\]
The injectivity radius satisfies
\[
C_{\iota}^{-1}\ell_0^{1-n}\leq {\rm inj}\,g_{\cC} \leq C_{\iota}\ell_0^{1-n}
\]
for a uniform constant $C_{\iota}$.
\end{prop}

Having now understood the Riemannian geometry of $\cC$, we move on to the study of the special Lagrangian submanifolds of $\cC$.  To this end, suppose that $N \subset D$ is a special Lagrangian submanifold of $D$.  Fix $\epsilon >0$ and consider the $S^1$ bundle over $D$
\[
\cC_{\epsilon} = \{ \xi \in \cC : |\xi|^2_{h} = \epsilon\} \xlongrightarrow{\pi} D.
\]
For each $\epsilon \in (0,1)$ we define a smooth, real codimension $n$ submanifold of $\cC$ by
\[
M_{\epsilon} = \pi^{-1}(N) \cap \cC_{\epsilon}.
\]
In other words, $M_{\epsilon}$ is the manifold obtained by restricting the circle bundle $\cC_{\epsilon}$ to $N$.  First, we describe the topology of $M_{\epsilon}$.
\begin{lem}
The manifold $M_{\epsilon}$ is topologically $S^1\times N$.  In particular, if $N = (S^1)^{n-1}$, is a torus, then so is $M_{\epsilon}$.
\end{lem}
\begin{proof}
Since $N$ is Lagrangian, we have $\omega|_{N} = 0$.  On the other hand, $\omega$ is the first Chern class of $L\rightarrow N$.  Since $M_{\epsilon}$ is the circle bundle in $L$, it follows that the Euler class of the circle bundle $M_{\epsilon} \rightarrow N$ vanishes and hence $M_{\epsilon}$ is topologically a trivial $S^1$-bundle over $N$.
\end{proof}

Our primary case of interest will be when $N$ is a torus and hence the above lemma ensures that, in this case $M_{\epsilon}$ will also be a torus.

Next we claim that $M_{\epsilon}$ is Lagrangian.  This can be achieved by a pointwise calculation, using the coordinate expression for $\omega_\cC$ in~\eqref{eq: calModSympPoint}.  Fix a point $p\in N$. As before, let $\xi$ be a local trivialization of $L$ so that $h(p)=1, dh(p)=0$ and write $h = e^{-\phi}$.  Let $w$ the corresponding local coordinate on $L$ and write $w=re^{\sqrt{-1}\theta}$ where $\theta \in S^1$.  Locally we have
\[
M_{\epsilon} = \{ (p,w) : p\in N, \quad |w|^2e^{-\phi(p)}= \epsilon\}.
\]
Therefore $\frac{dw}{w} = \frac{dr}{r} + \sqrt{-1}d\theta$ and we have
\[
\frac{dw}{w}\bigg|_{M_{\epsilon}} = \frac{1}{2}d\phi + \sqrt{-1}d\theta.
\]
By our choice of trivialization using \eqref{eq: calModSympPoint} we have $d\phi(p) =0$ and hence
\[
\frac{dw\wedge d\bar{w}}{|w|^2}\bigg|_{M_{\epsilon}} =0, \qquad \Omega_{\cC}|_{M_{\epsilon}} = \sqrt{-1}\pi^{*}\Omega_{D} \wedge d\theta\bigg|_{M_\epsilon}.
\]
It follows from \eqref{eq: calModSympPoint} that $M_{\epsilon}$ is Lagrangian.  Furthermore,
\begin{equation}
\label{eq: holVolMeps}
e^{-\sqrt{-1}\frac{\pi}{2}}\Omega_{\cC}\bigg|_{M_{\epsilon}}=  \pi^{*}{\rm Im}\Omega_{D}\bigg|_{N}\wedge d\theta =0.
\end{equation}
If necessary, rotate $\Omega_{D}$ by a unit complex number so that ${\rm Im}\left(\Omega_{D}\big|_{N}\right)=0$ and ${\rm Re}\left(\Omega_{D}\big|_{N}\right)=dVol_{N}$.  Then we have ${\rm Im}\left(e^{-\sqrt{-1}\frac{\pi}{2}}\Omega_{C}\big|_{M_{\epsilon}}\right) =0$.  Summarizing we have
\begin{lem}
If $N \subset D$ is special Lagrangian, then, for each $\epsilon>0$ the manifold 
\[
M_{\epsilon} = \{ \xi \in L : \pi(\xi)\in N,\,\, |\xi|^2_{h}=\epsilon\}
\]
is a special Lagrangian submanifold of $(\cC,\omega_\cC)$ which is topologically $N\times S^1$.
\end{lem}

As described in the introduction and already executed in Section~\ref{sec: ACYL}, we will transplant these special Lagrangians into the Tian-Yau spaces and run the Lagrangian mean curvature flow in order to produced special Lagrangians. In order to prove the convergence of the mean curvature flow, we will need to understand the Riemannian geometry of $M_{\epsilon}$ in some detail. This is what we take up next.  First, we compute the volume of $M_{\epsilon}$ with respect to the induced metric.
\begin{lem}\label{lem: TYmodVol}
The volume of $M_{\epsilon} \subset (\cC, g_{\cC})$ is independent of $\epsilon$ and given by
\[
{\rm Vol}(M_{\epsilon}, g_{\cC}) = \frac{2\pi}{\sqrt{n}} {\rm Vol}(N, g_{D}).
\]
\end{lem}
\begin{proof}
One can compute directly the volume form of the induced metric on $M_{\epsilon}$ from the formula of $\omega_{\cC}$~\eqref{eq: calModSymp}, to find 
\[
dVol_{M_{\epsilon}} = \frac{1}{\sqrt{n}}\pi^{*}dVol_{N}\wedge d\theta.
\]
The lemma follows by integration over $M_{\epsilon}$.
\end{proof}

Next we examine the bottom of the spectrum of the Laplacian on $M_{\epsilon}$.  Recall that the first non-zero eigenvalue of the Laplacian on a compact, boundaryless Riemannian manifold $(M,g)$ is characterized by the Rayleigh quotient
\[
\lambda_1(M,g) = \inf \frac{\int_{M}|\nabla f|^2 }{\int_{M} f^2 },
\]
where the infimum is taken over all real valued $L^2$ functions on $M$ with $\int_{M} f=0$.  Suppose $f : N \rightarrow \mathbb{R}$ and $\int_{N}f =0$.  By the formula for the volume form of $M_{\epsilon}$ we get $\int_{M_{\epsilon}}\pi^{*}f =0$ and
\[
\int_{M_{\epsilon}} (\pi^{*}f)^2 = 2\pi \int_{N} f^2.
\]
Furthermore, by a local computation using~\eqref{eq: calModSympPoint} we have
\[
|\nabla \pi^{*}f|_{g_{\cC}}^2 = (-\log(\epsilon))^{-1/n}\pi^{*} |\nabla f|_{g_{D}}^2.
\]
It follows immediately from the Rayleigh quotient formula that
\begin{equation}\label{eq: evBdMod}
\lambda_1(M_{\epsilon}) \leq \frac{\lambda_1(N)}{(-\log(\epsilon))^{1/n}}.
\end{equation}
We claim that in fact we have equality in~\eqref{eq: evBdMod}, provided $\epsilon$ is sufficiently small.  This essentially follows from work of B\'erard-Bergery- Bourguignon \cite{BB} (see also \cite{BeBo}), but we will give a simple explicit proof for the reader's convenience. 

\begin{lem}\label{lem: TYmodLambda}
The first non-zero eigenvalue of the Laplacian on $M_{\epsilon}$ satisfies
\[
\lambda_1(M_{\epsilon}) = \frac{\lambda_1(N)}{(-\log(\epsilon))^{1/n}}
\]
provided $\epsilon$ is sufficiently small, depending only on $\lambda_1(N), n$.
\end{lem}
\begin{proof}

Fix a point $p \in N$ and choose normal coordinates $(x_1,\ldots, x_{n-1})$ centered at $p$.  As before, we use $(x_1,\ldots, x_{n-1},\theta)$ as coordinates on $M_{\epsilon}$ choosing the trivialization of $L$ so that,
\begin{equation}\label{eq: metSimpCoord}
g_{\cC}\big|_{M_{\epsilon}} = (-\log(\epsilon))^{\frac{1}{n}-1} \frac{1}{n}d\theta^2 + (-\log(\epsilon))^{\frac{1}{n}}\pi^{*}g_{D}\big|_{N}
\end{equation}
at any point $q \in \pi^{-1}(p)$.  It is easy to see that the circle $\pi^{-1}(p)$ is a geodesic and so the Laplacian takes the form
\[
\begin{aligned}
\Delta_{M_{\epsilon}} f&= g^{\alpha\beta}\nabla_{\alpha}\nabla_{\beta} f - \nabla_{\nabla_{\del_{\alpha}\del_\beta}}f\\
&= n(-\log(\epsilon))^{1-\frac{1}{n}}\nabla_{\theta}\nabla_{\theta} f + \Delta_{H} f
\end{aligned}
\]
where
\[
\Delta_{H}f = (-\log(\epsilon))^{-\frac{1}{n}}\sum_{1\leq i,j \leq n-1} (g_{D})^{ij}\nabla_i\nabla_j f - (-\log(\epsilon))^{-\frac{1}{n}}(g_{D})^{ij}\Gamma_{ij}^{\theta} \frac{\del f}{\del \theta}.
\]
From this formula it is clear that if $f$ is an eigenfunction on $N$ with eigenvalue $\lambda$, then $\pi^{*}f$ is an eigenfunction on $M_{\epsilon}$ with eigenvalue $(-\log(\epsilon))^{-1/n}\lambda$.  In particular, $\frac{\lambda_1(N)}{(-\log(\epsilon))^{1/n}}$ is an eigenvalue of $\Delta_{M_{\epsilon}}$.

Now suppose $f$ is an eigenfunction of $\Delta_{M_{\epsilon}}$ with eigenvalue $\mu$.  Define (local) smooth functions $a_k(x)$ by
\[
a_k(x) = \frac{1}{2\pi}\int_{S^1} f(x,\theta) e^{-\sqrt{-1}k\theta}d\theta,
\]
and let $f_{k}(x,\theta) = a_k(x)e^{\sqrt{-1}k\theta}$.  Then, by standard results in Fourier analysis we have
\[
f = \sum_{k\in \mathbb{Z}} f_k(x,\theta)
\]
and this series converges locally smoothly.  In fact, $f_k(x,\theta)$ are globally defined smooth functions corresponding to the decomposition of $L^2(M_{\epsilon})$ into weight spaces induced by the natural isometric $U(1)$ action on $M_{\epsilon}$.  Since the $U(1)$ action is Killing we have, for each $\ell \in \mathbb{Z}$
\[
\begin{aligned}
\mu f_{\ell} &= \Delta_{M_{\epsilon}} f_{\ell}\\
&= 2n(-\log(\epsilon))^{1-\frac{1}{n}} (-\ell^2)f_{\ell} + \Delta_{H}f_{\ell}.\\
\end{aligned}
\]
Therefore we have
\begin{equation}\label{eq: horLapEig}
\Delta_{H}f_{\ell} = (\mu + 2n(-\log(\epsilon))^{1-\frac{1}{2}}\ell^2) f_{\ell}.
\end{equation}
On the other hand, for any complex valued function $h$ we have
\begin{equation}\label{eq: horLapNeg}
\begin{aligned}
\int_{M_{\epsilon}} \Delta_{H}h \overline{h}  &= \int_{M_{\epsilon}} \Delta h \overline{h} -  2n(-\log(\epsilon))^{1-\frac{1}{n}}\nabla_{\theta}\nabla_{\theta} h \overline{h} \\
&= -\int_{M_{\epsilon}} |\nabla h|^2 + \int_{M_{\epsilon}} |\nabla_{\theta} h|^2 \leq 0.
\end{aligned}
\end{equation}
Combining~\eqref{eq: horLapEig} and~\eqref{eq: horLapNeg} we conclude that, if $f_{\ell} \ne 0$ then
\[
(\mu + 2n(-\log(\epsilon))^{1-\frac{1}{n}}\ell^2) \leq 0.
\]
If $\epsilon$ is sufficiently small, then
\[
-2n(-\log(\epsilon))^{1-\frac{1}{n}} \ll \frac{\lambda_1(N)}{(-\log(\epsilon))^{1/n}},
\]
and so to be a competitor for $\lambda_1(M_{\epsilon})$ we must have $f_{\ell}=0$ for all $\ell \ne 0$.  But in this case it is clear that $f= \pi^{*}\tilde{f}$ for some eigenfunction $\tilde{f}$ of the Laplacian on  $N$.  The lemma follows immediately. 
\end{proof}

Next we will compute the second fundamental form of $M_{\epsilon} \subset (\cC,g_{\cC})$.  To this end, choose real coordinates $(x_1,\ldots,x_{n-1}, y_1,\ldots, y_{n-1})$ for $D$ centered at a point $p\in N$ so that $(x_{1}, \ldots, x_{n-1})$ are coordinates on $N$ and at $p$ we have $\{\frac{\del}{\del x_1},\ldots, \frac{\del}{\del x_{n-1}}\}$ orthonormal and 
\[
J_{D}(p) \frac{\del}{\del x_i}= \frac{\del}{\del y_i}.
\]
Since $g_{D}$ is Hermitian,  $\{\frac{\del}{\del y_1},\ldots, \frac{\del}{\del y_{n-1}}\}$ form an orthonormal basis for $T_{p}N^{\perp} \subset T_{p}D$.  As before we choose a trivialization of $L$ so that $h(p)= 1$ and $dh(p)=0$ and write $h= e^{-\phi}$.  Let $w$ be the induced coordinate on $L$ and write $w=re^{\sqrt{-1}\theta}$.  Define a new coordinate $u$ by
\[
u= \log(r)-\frac{1}{2}\phi.
\]
Then $(x_1,\ldots,x_{n-1}, y_1,\ldots, y_{n-1}, u, \theta)$ form local coordinates for any point near $(0,u,\theta)$.  In these coordinates the metric is given by
\[
g_{\cC} = (-2u)^{\frac{1}{n}-1}\frac{1}{n}\left(du^2 + \left(d\theta + \frac{1}{2}Jd\phi\right)^2\right) + (-2u)^{\frac{1}{n}}\pi^{*}g_{D},
\]
and this simplifies at $(0,u,\theta)$ since $d\phi=0$, so in particular, the metric is block diagonal there.  Fix a point $q\in M_{\epsilon}\subset \{ u=\frac{1}{2}\log(\epsilon)\}$ with $\pi(q)=p$.  To compute the second fundamental form we note that
\[
Y_k = (-\log(\epsilon))^{-\frac{1}{2n}}\frac{\del}{\del y_i},  \qquad U = \sqrt{n} (-\log(\epsilon))^{n-1}{2n} \frac{\del}{\del u}
\]
form an orthonormal basis for $T_{q}M_{\epsilon}$.  It suffices to compute $\langle \nabla_{X}Y, Z\rangle$ where $X,Y$ run over $\{\frac{\del}{\del x_1},\ldots, \frac{\del}{\del x_{n-1}}, \frac{\del}{\del\theta}\}$ and $Z$ runs over  $\{\frac{\del}{\del y_1},\ldots, \frac{\del}{\del y_{n-1}}, \frac{\del}{\del u}\}$.  In other words, we need to compute some of the Christoffel symbols of $g_{\cC}$ at $q$.  We begin by computing $\Gamma_{x_ix_j}^{\alpha}$, where $\alpha = y_k, u$. We have 
\[
\Gamma^{u}_{x_ix_j} = \frac{1}{2}n(-\log(\epsilon))^{1-\frac{1}{n}} \left(-\del_ug_{x_ix_j}\right).
\]
Now $g_{x_ix_j} = (-2u)^{\frac{1}{n}}(g_{D})_{x_ix_j} + (-2u)^{\frac{1}{n}-1} O((d\phi)^2)$ so that
\[
\Gamma^{u}_{x_ix_j} = \delta_{ij}, \qquad \langle \nabla_{\del_{x_i}}\del_{x_j}, U \rangle = \frac{(-\log(\epsilon))^{\frac{1-n}{2n}}}{\sqrt{n}} \delta_{ij}.
\]
Next, since $d\phi=0$ at $q$, it is straightforward to show that $\Gamma^{y_k}_{x_ix_j} = (\Gamma_{D})_{x_ix_j}^{y_k}$ where $\Gamma_{D}$ are the Christoffel symbols at $p\in D$.  Thus,
\[
 \langle \nabla_{\del_{x_i}}\del_{x_j}, Y_{k} \rangle = (-\log(\epsilon))^{\frac{1}{2n}}\langle\nabla^{D}_{\del_{x_i}}\del_{x_{j}}, \del_{y_{k}}\rangle_{g_{D}},
\]
which is a rescaling of the second fundamental form of $N\subset (D, g_{D})$.  Next, consider $\nabla_{\del_{\theta}}\del_{\theta}$.  We begin by computing
\[
\Gamma_{\theta\theta}^{u} = -\frac{1}{2}\del_{u}\log\left(\frac{(-2u)^{\frac{1}{n}-1}}{n}\right) = \frac{n-1}{2nu}.
\]
Therefore,
\[
\langle \nabla _{\del_{\theta}}\del_{\theta}, U \rangle = \frac{-(n-1)}{n\sqrt{n}} (-2u)^{\frac{1-3n}{2n}}.
\]
One easily checks that $\langle \nabla_{\del_\theta} \del_\theta ,Y_k\rangle =0$ and so it only remains to compute the contribution from $\nabla_{\del_{\theta}} \del_{x_{i}}$.  We have
\[
\Gamma_{\theta x_i}^{u} = \frac{n}{2}(-2u)^{1-\frac{1}{n}} \big(-\del_u (Jd\phi)(\frac{\del}{\del x_i})\big) =0,
\]
since $Jd\phi$ is independent of $u$.  It therefore suffices to compute
\[
\Gamma_{\theta x_i}^{y_k} = \frac{1}{2}(-\log(\epsilon))^{-\frac{1}{n}}(\del_{x_i}g_{y_{k}\theta} + \del_{\theta}g_{x_iy_k} - \del_{y_{k}}g_{x_i \theta}).
\]
Now $\del_{\theta}g_{x_iy_k} = 0$ since the metric is $\theta$-independent.  On the other hand,
\[
\del_{x_i}g_{y_{k}\theta} = (-2u)^{\frac{1}{n}-1}\frac{1}{2n}\frac{\del}{\del x_i}\left[(Jd\phi)\bigg(\frac{\del}{\del y_k}\bigg)\right].
\]
Since $p= \pi(q)$ satisfies $d\phi(p)=0$, we have $Jd\phi( \del_{y_k}) = -\del_{x_k}\phi + O(x^2)$, and so we get
\[
\del_{x_i}g_{y_{k}\theta} =- (-2u)^{\frac{1}{n}-1}\frac{1}{2n}\frac{\del^2 \phi}{\del x_ix_k}.
\]
Similarly we have $\del_{y_k}g_{x_{i}\theta} = (-2u)^{\frac{1}{n}-1}\frac{1}{2n}\frac{\del^2 \phi}{\del y_iy_k}$.  Therefore,
\[
\Gamma_{\theta x_i}^{y_k}= \frac{-1}{n}(-\log(\epsilon))^{-1}\frac{1}{4}\left(\frac{\del^2 \phi}{\del x_ix_k}+\frac{\del^2 \phi}{\del y_iy_k}\right).
\]
On the other hand, we have
\[
g_{i\bar{j}} = \frac{\del^2 \phi}{\del z_i\del\bar{z}_{j}} = \frac{1}{4} \left(\frac{\del^2 \phi}{\del x_i\del x_j} + \frac{\del^2\phi}{\del y_i \del y_j}\right) + \frac{\sqrt{-1}}{4}\left(\frac{\del^2 \phi}{\del x_i \del y_j} - \frac{\del^2}{\del y_i\del x_j}\right),
\]
and so, by our choice of coordinates, we have
\[
\Gamma_{\theta x_i}^{y_k}= \frac{-1}{n}(-\log(\epsilon))^{-1} \delta_{ik}.
\]
Therefore,
\[
\langle \nabla_{\del_{\theta}}\del_{x_i}, Y_k\rangle =  \frac{-1}{n}(-\log(\epsilon))^{\frac{1}{2n}-1} \delta_{ik}.
\]
Taking the norm of the second fundamental form, we obtain
\begin{lem}\label{lem: TYmodA}
The special Lagrangian submanifold $M_{\epsilon} \subset (\cC, g_{\cC})$ satisfies
\[
|A|^2_{g} \leq C(N,n)(-\log(\epsilon))^{-\frac{1}{n}},
\]
for a constant $C(N,n)$ depending only on the dimension and the second fundamental form of $N\subset(D, g_{D})$.
\end{lem}

\begin{defn}
	Let $(M,g)$ be a Riemannian manifold of dimension $n$ and let $\kappa, r_0>0$. We say $(M,g)$ is $\kappa$ non-collapsing at scale $r_0$ if for every $0<r<r_0$ and for every $p\in M$, we have
	\[
	{\rm Vol}(B(p,r)) \geq \kappa r^{n},
	\]
	where ${\rm Vol}(B(p,r))$ denotes the volume with respect to $g$ of the geodesic ball of radius $r$ about $p$ in $(M,g)$.
\end{defn}

Finally, we have the following non-collapsing result.

\begin{lem}\label{lem: TYmodNonCol} Assume that $N$ is $\kappa_N$ non-collapsing at the scale $r_N$. 
Define constants $\kappa, r_{\epsilon} >0$ by
\[
\kappa = \frac{\kappa_N }{2^{n-1}}, \qquad  r_{\epsilon} = \frac{2\pi}{\sqrt{n}}(-\log(\epsilon))^{\frac{1-n}{2n}}.
\]
For $\epsilon$ sufficiently small, depending only on $(N, g_{D})$,  $M_{\epsilon}$ is $\kappa$ non-collapsing at scale $r_\epsilon$. 

\end{lem}
\begin{proof}
Fix a point $p \in M_{\epsilon}$.  Choose $\epsilon$ sufficiently small so that
\[
\frac{2\pi}{\sqrt{n}} (-\log(\epsilon))^{-\frac{1}{2}} < r_{N}.
\]
  Let $d_{N}$ denote the distance function on $(N, g_{D})$.  For $r < \frac{2\pi}{\sqrt{n}}$, consider the set $Q = \{ q\in N: d_{N}(\pi(p),q)<  \frac{r}{2}(-\log \epsilon)^{-1/2}\}$.  For each point $q \in Q$ with $\pi(p) \ne q$ we can choose a unit speed geodesic $\gamma(t)$ in $(N, g_{D})$ connecting $\pi(p)$ to $q$ and having length $d_{N}(\pi(p),q) < \frac{r}{2}(-\log \epsilon)^{-1/2}$.  We now take the horizontal lift of this curve through $p$.  Recall that the metric on $M_{\epsilon}$ is given in local coordinates by
  \[
  g_{\cC}\big|_{M_{\epsilon}} = (-\log(\epsilon))^{\frac{1}{n}-1} \frac{1}{n}(d\theta +\frac{1}{2}Jd\phi)^2 + (-\log(\epsilon))^{\frac{1}{n}}\pi^{*}g_{D}\big|_{N}
   \]
   and the tangent space of the fibers $\pi:M_{\epsilon}\rightarrow N$ is spanned by $\frac{\del}{\del \theta}$.  Thus, the horizontal lift $\bar{\gamma}$ is a lift of $\gamma$ to $M_{\epsilon}$ such that
   \[
   \iota_{\dot{\overline{\gamma}}}(d\theta + \frac{1}{2}Jd \phi )=0.
   \]
   In particular, if $\bar{\gamma}$ denotes the horizontal lift of $\gamma$ to $M_{\epsilon}$ passing through $p$, then $\bar{\gamma}$ connects $p$ to a point $\bar{\gamma}(1) \in \pi^{-1}(q)$ and $\bar{\gamma}$ satisfies
   \[
   |\dot{\overline{\gamma}}|^2=(-\log(\epsilon))^{\frac{1}{n}}|\dot{\gamma}|^2_{g_{D}}.
   \]
   Thus, $\bar{\gamma}$ has length in $M_{\epsilon}$ given by
  \[
  {\rm Length}_{(M_{\epsilon}, g_{\cC})}(\bar{\gamma}) =(-\log(\epsilon))^{\frac{1}{2n}}d_{N}(\pi(p),q) <   \frac{r}{2}(-\log(\epsilon))^{\frac{1-n}{2n}}.
  \]
  
Consider the ball $B: = B(p,r(-\log(\epsilon))^{\frac{1-n}{2n}}) \subset M_{\epsilon}$.  Choose coordinates $(x_1,\ldots, x_{n-1}, \theta)$ centered at $\bar{\gamma}(1)$ so that the metric takes the form~\eqref{eq: metSimpCoord} at any point in the fiber containing $\overline{\gamma}(1)$.  If $q'$ is any point in the fiber containing $\overline{\gamma}(1)$ with $|\theta(q')|< \frac{\sqrt{n}r}{2}$ then we claim that $q' \in B$.  To see this, connect $q'$ to $p$ by concantenating the curve $\overline{\gamma}$ with the curve in the fiber from $\overline{\gamma}(1)$ to $q'$. The length of this curve is at most $r(-\log(\epsilon))^{\frac{1-n}{2n}}$.  Recall that, by assumption, $\frac{\sqrt{n}r}{2} < \pi$.  Since this holds for every point $q\in Q$ and the volume form of $M_{\epsilon}$ is given by $\frac{1}{\sqrt{n}}d\theta \wedge \pi^{*}dVol_{N}$, we have
\[
\begin{aligned}
\int_{B} dVol_{M_{\epsilon}} \geq r \int_{Q} dVol_{N} &> r\kappa_N \left(\frac{r}{2}\right)^{n-1}(-\log(\epsilon))^{\frac{1-n}{2}}\\
&= \frac{\kappa_N }{2^{n-1}}\left(r(-\log(\epsilon))^{\frac{1-n}{2n}}\right)^{n}
\end{aligned}
\]
which is the desired result. 
\end{proof}

We now essentially understand the geometry of $M_{\epsilon}$ as a subset of the model space $(\cC, g_{\cC})$.  

\subsection{Transplantation and perturbation}

The next step in the construction is to transplant the special Lagrangians in the Calabi model into the Tian-Yau spaces to produce approximate special Lagrangians using the calculations in Section~\ref{sec: perturb}.  To begin, we recall the identification of the end of the Tian-Yau spaces with the Calabi model; our discussion follows closely \cite{HSVZ}.  Thus, we fix a Fano K\"ahler manifold $X$ and assume that $D = \{\sigma =0\}$ is a smooth anti-canonical divisor in the linear system $|-K_{X}|$. Denote by $L=-K_X|_D$ the normal bundle of $D$ in $Y$.  We fix a holomorphic volume form on $D$ satisfying the normalization~\eqref{eq: normOmegaD}. We view $\frac{1}{\sigma} = \Omega_{X}$ as a holomorphic $(n,0)$-form on $X$ with its normalization chosen so that the residue of $\Omega_{X}$ on $D$ is $\Omega_{D}$.  Let $h_{D}$ be the unique up-to-scale positively curved metric on $-K_{X}|_{D}\rightarrow D$ such that
\[
-\ddb \log h_{D} = \omega_{D}
\]
is Ricci-flat on $D$.  Let $h_{X}$ be a positively curved hermitian metric on $-K_{X}$ extending $h_{D}$ and put
\[
\omega_{X} = \ddb\frac{n}{n+1}(-\log|\sigma|^2_{h_{X}})^{\frac{n+1}{n}}.
\]
After possibly scaling $h_{X}$ by a sufficiently small positive constant, we can assume that $\omega_{X}$ defines a smooth positive K\"ahler metric on $X\setminus D$ and evidently $\omega_{X}$ is asymptotic to the Calabi model.  The following theorem is due to Tian-Yau \cite{TY} with the exponential decay estimates due to Hein \cite{Hein}.

\begin{thm}\label{thm: TY}
There exists a function $\phi: X\setminus D \rightarrow \mathbb{R}$ such that $\omega_{TY} := \omega_{X} + \ddb \phi$ is a complete Ricci-flat K\"ahler metric on $X\setminus D$ solving the Monge-Amp\`ere equation
\begin{align} \label{cpx Monge}
\omega_{TY}^{n} = \frac{(\sqrt{-1})^{n^2}}{2}\Omega_{X}\wedge \overline{\Omega}_{X}.
\end{align}

Furthermore, there is a constant $\delta_0= \delta_0(M,D)$ such that
\[
|\nabla^{k}_{g_{X}} \phi|_{g_{X}} = O(e^{-\delta_0 \ell_0^{n+1}})
\]
where $\ell_0$ is the scale function of Definition~\ref{def: scaleFunc}
\end{thm}

Fix a smooth K\"ahler metric $g$ on $X$.  Using $g$ we can identify $L$ with $(T^{1,0}D)^{\perp}$ as $C^{\infty}$ complex line bundles.  Using the $g$-exponential map, we get a diffeomorphism $\Phi$ between a neighborhood of the zero section in $L$ and a neighborhood of $D$ in $X$.  Under this identification, we have the following estimates \cite{Hein, HSVZ}

\begin{prop}\label{prop: HSVZdecay}
There is a diffeomorphism $\Phi: \mathcal{C}\setminus K' \rightarrow X\setminus K$, where $K\subset X$ is compact, and $K' := \{ |\xi|_{h} \geq \frac{1}{2} \}$, such that the following estimates hold uniformly for all large enough values of $\ell_0$:
\begin{enumerate}
\item $|\nabla^k_{g_{C}}(\Phi^{*}J_{X} - J_C)|_{g_{C}} = O(e^{-(\frac{1}{2}-\epsilon)\ell_0^{2n}})$ for all $k\geq 0, \epsilon>0$,
\item $|\nabla^k_{g_{C}}(\Phi^{*}\Omega_{X} - \Omega_C)|_{g_{C}} = O(e^{-(\frac{1}{2}-\epsilon)\ell_0^{2n}})$ for all $k\geq 0, \epsilon>0$,
\item $|\nabla^k_{g_{C}}(\Phi^{*}\omega_{X} - \omega_C)|_{g_{C}} = O(e^{-(\frac{1}{2}-\epsilon)\ell_0^{2n}}))$ for all $k\geq 0, \epsilon>0$,
\item $|\nabla^k_{g_{C}}(\Phi^{*}(-\log|\sigma|^2_{h_{X}})^{\frac{n+1}{n}} - (-\log|\xi|^2_{h_{D}})^{\frac{n+1}{n}})|_{g_{C}} =O(e^{-(\frac{1}{2}-\epsilon)\ell_0^{2n}})$ for all $k\geq 0, \epsilon>0$,
\item There is a number $\underline{\delta}>0$ such that for all $k\geq 0$, we have 
\[
\begin{aligned}
|\nabla^k_{g_{C}}(\Phi^{*}\omega_{TY} - \omega_C)|_{g_{C}} &= O(e^{-\underline{\delta}\ell_0^{2n}}),\\
|\nabla^k_{g_{C}}(\Phi^{*}g_{TY} - g_C)|_{g_{C}} &= O(e^{-\underline{\delta}\ell_0^{2n}}).
\end{aligned}
\]
\end{enumerate}
\end{prop}

We will (somewhat abusively) use $\ell_0$ to denote the scale function pulled back by $\Phi^{-1}$ to $X$.  Note that by Proposition~\ref{prop: HSVZdecay}, there is a uniform constant $C$ such that
\[
C^{-1}(-\log|\sigma|^2_{h_{X}})^{\frac{1}{2n}} \leq \ell_0 \leq C(-\log|\sigma|^2_{h_{X}})^{\frac{1}{2n}}.
\]
Furthermore, Proposition~\ref{prop: HSVZdecay} together with Proposition~\ref{prop: calModDec} can be equivalently reformulated as
\begin{cor}\label{cor: TYcurvDec}
For all $k \geq 0$ and $n \geq 3$, we have
\[
|\nabla^{k}_{g_{TY}}Rm(g_{TY})|_{g_{TY}}  = O(\ell_0^{-(2+k)}).
\]
When $n=2$, then, for all $k\geq 0$ we have
\[
|\nabla^{k}_{g_{TY}}Rm(g_{TY})|_{g_{TY}}  =O(\ell_0^{-(6+k)}).
\]
The injectivity radius satisfies
\[
C_{\iota}^{-1}\ell_0^{1-n}\leq {\rm inj}\,g_{TY} \leq C_{\iota}\ell_0^{1-n}
\]
for a uniform constant $C_{\iota}$.
\end{cor}

Note that
\[
\begin{aligned}
&\Phi^{*}\omega_{TY}- \omega_{C}\\
 &= d\left( \Phi^{*} \left(\sqrt{-1}\,\dbar_{X}\frac{n}{n+1}(-\log|\sigma|^2_{h_{X}})^{\frac{n+1}{n}} + \dbar_{X}\phi \right)- \dbar_{C}\frac{n}{n+1}(-\log |\xi|^{2}_{h})^{\frac{n+1}{n}}\right)\\
&= \frac{-n}{2(n+1)}d\left( (\Phi^{*}J_{X})d\Phi^{*}\left[(-\log|\sigma|^2_{h_{X}})^{\frac{n+1}{n}} +\frac{n+1}{n}\phi\right] - J_{C}d(-\log |\xi|^{2}_{h})^{\frac{n+1}{n}}\right).
\end{aligned}
\]
Thus, by Proposition~\ref{prop: HSVZdecay} parts  $(1)$ and $(4)$, together with Theorem~\ref{thm: TY} we get a $1$-form $\beta$ satisfying
\[
\Phi^{*}\omega_{TY}- \omega_{C} = d\beta, \qquad  \sum_{\ell=0}^{\ell}|\nabla^{\ell}_{g_{C}} \beta|_{g_{C}} = O(e^{-(\frac{1}{2}-\epsilon)\ell_0^{2n}}),
\]
for all $k\geq 0$ and $\epsilon >0$.  Define symplectic forms $\omega_{t} = (1-t)\omega_{C} + t\Phi^{*}\omega_{TY}$ for $t\in [0,1]$ and a time dependent vector field $V_{t}$ by
\begin{equation}\label{eq: TYmodV}
\iota_{V_{t}}\omega_{t} = -\beta.
\end{equation}
It follows from the estimates in Proposition~\ref{prop: HSVZdecay}, that, for all $k\geq 0$,
\begin{equation}\label{eq: calModVest}
\sum_{\ell=0}^{k}|\nabla^{\ell}_{g_{C}} V_t|_{g_{C}} = O(e^{-\underline{\delta}\ell_0^{2n}}).
\end{equation}
Define Riemannian metrics $g_{t} = (1-t)g_{C} + t\Phi^{*}g_{TY}$.  Proposition~\ref{prop: HSVZdecay} implies that for all $k \geq 0$, we have  
\[
\frac{1}{2}g_{C} \leq g_{t} \leq 2 g_{C},\quad  |\nabla^{k}_{g_{C}}g_{t}|_{g_{C}} = O(e^{-\underline{\delta}\ell_0^{2n}}).
\]
We obtain
\begin{cor}\label{cor: TYmodVest}
Let $V_{t}, g_{t}$ be as above.  Then there is a number $\underline{\delta} >0$, so that, for all $t\in [0,1]$ and all $k\geq 0$ we have 
\[
\sum_{\ell=0}^{k} |\nabla^{\ell}_{g_{t}} V_t|_{g_{t}} =  O(e^{-\underline{\delta}\ell_0^{2n}}).
\]
\end{cor}
We can now apply the analysis of Section~\ref{sec: perturb} to conclude

\begin{prop}\label{prop: TYalmostSLAG}
Suppose that $N \subset (D, \omega_{D})$ is a special Lagrangian submanifold.  Then, for all $K\gg 0$ sufficiently large there exists a Lagrangian submanifold $M_{K} \subset (X, \omega_{TY}, g_{TY}, J_{TY}, \Omega_{TY})$, which is topologically $N\times S^1$, and has vanishing Maslov class.  Furthermore, there are uniform constants $C>2, \delta'>0$, depending only on $N$ and the estimates in Proposition~\ref{prop: HSVZdecay}, such that
\begin{enumerate}
\item The function $\ell_0$ satisfies
\[
C^{-1}K < \ell_0|_{M_{K}} < CK.
\]
\item The second fundamental form satisfies
\[
|A|_{g_{TY}}^2 \leq CK^{-2}.
\]
\item The mean curvature satisfies
\[
|H|_{g_{TY}}^2 \leq Ce^{-\delta'K^{2n}}.
\]
\item The volume satisfies 
\[
(1-C^{-1}){\rm Vol}_{g_{D}}(N) \leq {\rm Vol}_{g_{TY}}(M_{K}) \leq (1+C^{-1}){\rm Vol}_{g_{D}}(N).
\]
\item The first positive eigenvalue $\lambda_1(M_K)$, of $(M_{K}, g_{TY})$ satisfies
\[
C^{-1}\lambda_1(N)\ell_0^{-2} \leq \lambda_1(M_K) \leq C\lambda_1(N)\ell_0^{-2}.
\]
\item  $(M_{K}, g_{TY})$ is $\kappa_0$ non-collapsing on scale $r_K$ where 
\[
\kappa_0 = C^{-1}, \qquad r_{K} = C^{-1} K^{1-n},
\]
\end{enumerate}
where $C, \kappa_{N}$ are uniform constants depending only on $N$ and the estimates in Proposition~\ref{prop: HSVZdecay}.
\end{prop}
\begin{proof}
The proof is the culmination of the estimates in Section~\ref{sec: TYmodLag} together with the arguments in Section~\ref{sec: perturb}.  Given $K$ large, define  $\epsilon >0$ by $K = (-\log(\epsilon))^{\frac{1}{2n}}$.  Let $M_{\epsilon}$ be the special Lagrangian in $(C,\omega_{C}, J_{C}, \Omega_{C})$ constructed in Section~\ref{sec: TYmodLag}.  By the calculations in Section~\ref{sec: TYmodLag} the above estimates hold, with constants depending only on $n, N$ for $M_{\epsilon}$.  We now follow the arguments in Section~\ref{sec: perturb} to transplant and perturb $M_{\epsilon}$ to a Lagrangian in the Tian-Yau space.  To this end, let $F_{t}$ the the time $t$ flow of the vector field $V_{t}$ defined in~\eqref{eq: TYmodV} and let $g_{t} = (1-t)g_{C} + t \Phi^{*}g_{TY}$ as above.  By definition we have that $F_{t}^{*}\omega_{t} = \omega_{C}$, so $F_{t}(M_{\epsilon})$ is Lagrangian with Maslov class zero with respect to $\omega_{t}$.  It follows that
\[
M_{K} := \Phi(F_{1}(M_{\epsilon})) \subset (X,\omega_{TY})
\]
is Lagrangian with vanishing Maslov class.

To control the geometry, we begin by estimating the function $\ell_0\big|_{F_{t}(M_{\epsilon})}$.  By definition we have $\ell_{0}\big|_{F_{0}(M_{\epsilon})} = K$.  By Remark~\ref{rek: scaleRk} and the estimates in Proposition~\ref{prop: HSVZdecay}, we have
\[
|\nabla_{g_t}\ell_0^{n+1}|_{g_t}^2 \leq \frac{(n+1)^2}{4n} + Ce^{-\underline{\delta}\ell_0^{2n}} \leq (n+1)^2
\]
provided $K$ is sufficiently large. Therefore,
\[
\bigg|\ddt \ell_0^{n+1}\bigg| = \big|\langle\nabla_{g_{t}} \ell_0^{n+1}, V_{t}\rangle_{g_{t}}\big| \leq (n+1)e^{-\underline{\delta}\ell_0^{2n}}.
\]
Thus, if $K$ is sufficiently large, we will have
\[
\frac{K}{2}  \leq \ell_0(M_{K}) \leq 2 K.
\]

It remains to control the geometry of $(F_{t}(M_{\epsilon}), g_{t})$.  By the discussion in Section~\ref{sec: perturb} it suffices to control the geometry of $M_{\epsilon} \subset C$ with respect to the Riemannian metrics  $\tilde{g}_{t} := F_{t}^{*}g_{t}$.  By Corollary~\ref{cor: TYmodVest}, together with Lemma~\ref{lem: perturbNorm} we have that
\begin{equation}\label{eq: perturbEstTYmod}
\begin{aligned}
|\del_t \tilde{g}_{t}|_{\tilde{g_{t}}} &\leq C_0 e^{-\underline{\delta}\ell_0^{2n}},\\
|\nabla_{\tilde{g}_{t}}\del_{t}\tilde{g}_{t}|_{\tilde{g_{t}}}& \leq C_0 e^{-\underline{\delta}\ell_0^{2n}}.
\end{aligned}
\end{equation}
Let us first consider the second fundamental form and the mean curvature of $M_{\epsilon}$ with respect to $\tilde{g}_{t}$.  Consider the ODE
\begin{equation}\label{eq: perturbVarEq}
\frac{df}{dt} = c(f^{\frac{1}{2}} + f), \qquad f(0)\geq 0
\end{equation}
whose solution is $f(t) = \left(-1+ [1+f(0)^{\frac{1}{2}}]e^{\frac{ct}{2}}\right)^2$.  By Lemma~\ref{lem: perturbAH}, equation~\eqref{eq: perturbEstTYmod} and part  $(1)$ of the Proposition both $|H|^2(t)$ and $|A|^2(t)$ are subsolutions of ~\eqref{eq: perturbVarEq} with constant $c= C'e^{-\delta'K^{2n}}$ where $C', \delta' >0$ are uniform constants.  For $K$ sufficiently large depending only on $C', \delta'$ we obtain
\[
\begin{aligned}
|A|^2(t) &\leq 100(C')^2e^{-2\delta'K^{2n}}+ 4|A|^2(0),   \\
|H|^2(t) &\leq  100(C')^2e^{-2\delta'K^{2n}} + 4|H|^2(0).
\end{aligned}
\]
Since $M_{\epsilon}$ is minimal with respect to $\tilde{g}_0$ we have $|H|(0)=0$, while by Lemma~\ref{lem: TYmodA} we have $|A|(0) \leq C_1K^{-2}$ for a constant $C_1$ depending only on $n$ and $N$.  This establishes $(2)$.

Estimate $(4)$ follows immediately from~\eqref{eq: perturbEstTYmod}, while (5) follows from~\eqref{eq: perturbEstTYmod} and Lemma~\ref{lem: perturbLambda}. Finally, estimate $(6)$ follows from the Lemma~\ref{lem: perturbLambda}, since~\eqref{eq: perturbEstTYmod}, together with $(1)$ implies that for $K$ sufficiently large, depending only on $N, n$ and the constants in Proposition~\ref{prop: HSVZdecay},
\[
\frac{1}{2}\tilde{g}_{t} \leq \tilde{g}_{0} \leq 2 \tilde{g}_{t}
\]
for all $t\in[0,1]$.  Combining these calculations with the discussion at the beginning of Section~\ref{sec: perturb} we obtain the result.
\end{proof}

\subsection{The mean curvature flow}

The only remaining task is to prove that the almost special Lagrangian manifold $M_{K}$ constructed in Proposition~\ref{prop: TYalmostSLAG} can be perturbed to a special Lagrangian.  To do this, we will show that for $K$ sufficiently large the Lagrangian mean curvature flow starting from $M_{K}$ converges to a special Lagrangian.  Furthermore, by controlling the scale function along the flow, we will show that we can construct infinitely many distinct special Lagrangians.

In Section~\ref{sec: ACYL} we appealed to a theorem of Li \cite{Li}; see Theorem~\ref{Liconverge}.  However, it is clear that this result does not apply in the Tian-Yau spaces since the injectivity radius is not bounded below.  More crucially, however, in order to make the mean curvature of the initial manifold $M_{K}$ very small, we may have to take $K$ very large.  In turn, by the estimates in Proposition~\ref{prop: TYalmostSLAG}, this causes $\lambda_1(M_K)$ and the non-collapsing scale to become {\em even smaller}.

We therefore need an effective version of Li's result tailored to our situation which exploits the fact that the mean curvature of $M_K$ decays exponentially in $K$, while the quantities $r_K, \lambda_{1}(M_{K})$ decay only polynomially.

Before beginning the proof, let us fix some notation.  In what follows we will use unbarred quantities, $g, \nabla, Rm, \nabla Rm$ and so forth, to denote quantities associated with the manifold $(X, g_{TY})$.  The corresponding barred quantities $\overline{g}, \overline{\nabla}, \overline{Rm}, \overline{\nabla}\overline{Rm}$ will denote quantities computed on $M_t$ with respect to the induced metric.

Let us briefly explain the idea of the proof.  The key result is the following.

\begin{lem}\label{lem: HexpDec}
Let $M_t$ be compact Lagrangian submanifolds of vanishing Maslov class moving by the LMCF.  Then the mean curvature satisfies
\[
\ddt \int_{M_t} |H|_{g}^2 dVol_{g} \leq -2\left(\lambda_1(M_t) - \sup_{M_t}|A|_{g}|H|_{g}\right)\int_{M_{t}}|H|_{g}^2 dVol_{g},
\]
where $g$ denotes the metric induced by the Tian-Yau metric on $X$ and $\lambda_1(M_t)$ is the first positive eigenvalue of the Laplacian on $(M_t,g|_{M_{t}})$.  In particular, if $\lambda_1(M_{t}) > \epsilon$ and $\epsilon >  2\sup_{M_t}|A|_{g}|H|_{g}$ on some interval $[0,T]$, then we have
\[
\int_{M_{t}} |H|_{g}^2 dVol_{g} \leq e^{-\epsilon t}\int_{M_0}|H|_{g}^2 dVol_{g}
\]
on $[0,T]$.
\end{lem}
\begin{proof}
The proof is straightforward.  To ease notation, we will suppress the metric $g$, with the understanding that it is the metric induced by $g_{TY}$.  A standard computation \cite{Smo} shows that the mean curvature one-form satisfies
\[
\ddt H_{j} = \nabla_{j} \nabla^{i}H_{i},
\]
along the flow.  Combining this formula with the evolution for the metric $\bar{g}$ yields
\[
\ddt \int_{M_t} |H|^2 dVol  \leq 2 \int_{M_t} \bigg(\overline{g}^{j\ell} H_{\ell} \nabla_{j} \nabla^{i}H_{i} + 2|A||H|^3 - |H|^4 \bigg) dVol.
\]
Integration by parts on the first term yields
\[
\ddt \int_{M_t} |H|^2 dVol  \leq -2 \int_{M_t}|\nabla^{i}H_{i}|^2 + 2\sup_{M_{t}}|A||H|\int_{M_{t}}|H|^2dVol. 
\]
Now, by the Maslov class zero assumption there is a function $\theta(t)$ so that $H_{j} = \nabla_{j}\theta(t)$.  In particular we have
\[
\int_{M_t}|\nabla^{i}H_{i}|^2 = \int_{M_t} |\Delta_{\overline{g}}\theta|^2.
\]
Write $\theta = \sum_{i}f_{i}$ where $f_{i} = \alpha _i \psi_i$ for $\alpha_i \in \mathbb{R}$ and $\psi_i$ is an orthonormal basis of $L^2$ consisting of eigenfunctions of $\Delta_{\bar{g}}$; we say that $\psi_i$ has eigenvalue $\lambda_i$ if $\Delta \psi_i + \lambda_i \psi_i=0$.  Then we have (suppressing the volume form)
\[
\int_{M_t} |\Delta_{\overline{g}}\theta|^2 = \sum_{i}\lambda_i^2\int_{M_{t}} f_{i}^2 \geq \lambda_1\sum_{i}\lambda_i\int_{M_{t}} f_{i}^2 = -\lambda_1\int_{M_{t}} \theta \Delta_{\bar{g}}\theta = \lambda_1 \int_{M_{t}} |\nabla \theta|^2.
\]
As a consequence we have
\[
\int_{M_{t}}|\nabla^{i}H_{i}|^2dVol \geq \lambda_1 \int_{M_{t}} |H|^2 dVol
\]
and the lemma follows immediately.
\end{proof}

The general idea of the proof is that if $\lambda_1(M_t)$ is large compared to $|A|, |H|$, and bounded from below on some time interval $[0,T]$, then the previous lemma implies the exponential decay of the $L^2$ norm of the mean curvature.  This implies pointwise exponential decay for $|H|^2$ provided that $M_t$ is non-collapsing and $|\nabla H|^2$ is controlled. The exponential decay of the mean curvature strongly controls the geometry of the flow on $[0,T]$ and yields exponential decay on an even larger interval.  To ease the presentation, we make the following definition.
\begin{defn}
A Maslov class zero Lagrangian submanifold $M \subset (X,\omega_{TY}, g_{TY})$ has $(C,K, \delta')$-bounded geometry if
\begin{enumerate}
\item The function $\ell_0$ satisfies 
\[
C^{-1}K < \ell_0|_{M} < CK.
\]
\item The second fundamental form satisfies 
\[
|A|^2 \leq CK^{-2}.
\]
\item The mean curvature satisfies 
\[
|H|^2 \leq Ce^{-\delta'K^{2n}}.
\]
\item The volume satisfies 
\[
C^{-1} \leq {\rm Vol}(M) \leq C.
\]
\item The first positive eigenvalue $\lambda_1(M)$ satisfies
\[
C^{-1}K^{-2} \leq \lambda_1(M_K) \leq CK^{-2}.
\]
\item  $(M, g_{TY})$ is $\kappa_0$-non-collapsing on scale $r_0$ where 
\[
\kappa_0 \geq C^{-1}, \qquad r_{0} \geq C^{-1} K^{1-n}.
\]
\end{enumerate}
\end{defn}

As a first step, we show that control of $|A|^2, |H|^2$ and $\ell_0$ at time $t=0$ implies control on a suitably large time interval.

\begin{lem}\label{lem: TYlmcfAest}
Suppose that $M_0$ is a Maslov clas zero Lagrangian in $(X,\omega_{TY})$ with $(C,K, \delta')$-bounded geometry.  Let $M_{t}$ be a solution of the mean curvature flow starting at $M_0$.  Then, for all $\delta \in(0,10)$ there is a constant $\alpha= \alpha(C,\delta) >0$ so that $M_{t}$ has $((1+\delta)C, K, \delta')$-bounded geometry for $t\in [0,\alpha K^2)$.
\end{lem}
\begin{proof}
Define three times $T_{S}, T_{A}, T_{H}>0$ by
\[
\begin{aligned}
T_{S} &:= \sup \{s>0:  \frac{1}{(1+\delta)}C^{-1}K \leq \ell_0|_{M_t} < (1+\delta)CK\, \text{ for all } t\in [0,s)\},\\
T_{A} &:= \sup \{s>0:  |A|^2(t) < (1+\delta)CK^{-2}\, \text{ for all } t\in [0,s)\},\\
T_{H} &:= \sup \{s>0:  |H|^2(t) < (1+\delta)Ce^{-\delta' K^{2n}}\, \text{ for all } t\in [0,s)\}.
\end{aligned}
\]
We first estimate $T_{A}$.  Recall the evolution equation for the norm of the second fundamental form \cite{Smo}
\[
\ddt |A| \leq \Delta |A| + 8|A|^3 + 20|Rm||A| + 4|\nabla Rm|.
\]
By Corollary~\ref{cor: TYcurvDec}, there is a uniform constant $D$ so that on $[0, T_{S})$ we have
\[
\sup_{t\in[0, T_{S})} \sup_{M_{t}} |Rm| \leq D((1+\delta)CK^{-2}), \quad \sup_{t\in[0, T_{S})} \sup_{M_{t}} |\nabla Rm| \leq D((1+\delta)CK^{-3}).
\]
Combining this with the definition of $T_{A}$, we conclude that on the interval $[0, \min\{T_{S}, T_{A}\})$ there holds
\[
\ddt |A| \leq \Delta |A| + 10^3(C^{3/2}K^{-3} + DC^{3/2}K^{-3} + DCK^{-3}).
\]
By the comparison principle there is a constant $c_{A}>0$, depending only on $C,D, \delta$ so that
\[
T_{A} \geq \min\{ c_{A}K^2, T_{S}\}.
\]
We estimate $T_{H}$ in a similar way.  Recall that along the LMCF $|H|$ satisfies the inequality
\[
\ddt |H| \leq \Delta |H| +2 |A|^2|H| + |Rm||H|.
\]
Arguing as above, we have that, as long as $0 <t < \min \{T_{S}, c_{A}K^{2}\}$ there holds
\[
\ddt |H| \leq \Delta |H| +2 |A|^2|H| + |Rm||H| \leq \Delta |H| + (1+\delta)CK^{-2}(1+D)|H|
\]
and so by the comparison principle
\[
|H(t)|^2 \leq |H(0)|^2e^{2(1+\delta)CK^{-2}(1+D)t} \leq Ce^{-\delta' K^{2n}}e^{100CK^{-2}(1+D)t}.
\]
As a result, there is a constant $c_{H}$ depending only on $C,D, \delta$ so that $T_{H} \geq \min \{ c_{H}K^2, c_{A}K^2,  T_{S}\}$. 

On the other hand, it is easy to see that $T_{S} > T_{H}$.  Arguing as in the proof of Proposition~\ref{prop: TYalmostSLAG}, by the equivalence of $\Phi^{*}g_{TY}, g_{C}$ near infinity, we have
\[
\bigg|\ddt \ell^{n+1}_0(M_t)\bigg| \leq |\nabla \ell_0^{n+1}||H| \leq (n+1)Ce^{-\delta'K^{2n}}
\]
as long as $t \leq T_{H}$.  In particular, we can certainly choose a constant $c_{S}$ depending only on $n, \delta,\delta', C$ so that if $t < \min\{ T_{H}, c_{S}K^{2}\}$ then we have
\[
\frac{1}{(1+\delta)}C^{-1}K \leq \ell_0|_{M_{t}} \leq (1+\delta) C K
\]
and so $T_{S} > \min\{ T_{H}, c_{S}K^{2}\}$. Combining these estimates, we conclude that there is a constant $\alpha'$ such that
\[
\min\{T_{S}, T_{H}, T_{A}\} \geq \alpha' K^{2}.
\]
It remains only to prove that $\lambda_1(M_t)$ and the non-collapsing scale are under control.  This is straightforward.  A standard computation shows that the induced metric $\overline{g}_{t}$ on $M_{t}$ satisfies 
\[
\frac{d \overline{g}_{t}}{dt} = -2\langle H(t), A(t)\rangle_{g_{t}}.
\]
Define
\[
\mu(t) = 2\int_{0}^{t} (\sup_{M_{t}}|H|)\cdot (\sup_{M_{t}}|A|) dt.
\]
Then by Lemma~\ref{lem: perturbLambda}, we have
\[
\lambda_1(M_0)e^{-3\mu(t)} \leq \lambda_1(M_t) \leq \lambda_1(M_0)e^{3\mu(t)}
\]
and trivially $e^{-\mu(t)} \overline{g}(0) \leq \overline{g}(t) \leq e^{\mu(t)} \overline{g}(0)$.  If $M_0$ is $\kappa_0$-non-collapsing on scale $r_K$, then $M_{t}$ will be $\kappa_0e^{-(n+1)\mu(t)}$ non-collapsing on scale $r_K$.  Now, if $t< \alpha K^2$ for $\alpha \leq  \alpha'$ then we have
\[
\mu(t) \leq t (1+\delta)^2C^2K^{-2}e^{-\delta' K^{2n}} \leq \alpha(1+\delta)^2C^2e^{-\delta'K^{2n}}  \leq \frac{1}{3(n+1)}\log(1 +\delta)
\]
provided $\alpha$ is taken sufficiently small depending only on $C, \delta$.  The lemma follows.
\end{proof}

We extract the bounds for $\lambda_1$ and the non-collapsing constant in the following elementary corollary which follows from \cite[Lemma 3.4]{Li} and the computations in Section~\ref{sec: perturb}.

\begin{cor}\label{cor: TYkapLam}
Suppose $M_{t}$ evolves by LMCF and $M_0$ is $\kappa_0$ non-collapsing at scale $r_0$. Denote by $\lambda_1(M_t)$ the first positive eigenvalue of the Laplacian on $M_{t}$ and define
\[
\mu(t) =  2\int_{0}^{t} (\sup_{M_{t}}|H|)\cdot (\sup_{M_{t}}|A|) dt.
\]
Then $M_{t}$ is $\kappa_0e^{-(n+1)\mu(t)}$-non-collapsing at scale $r_0$ and satisfies
\[
\lambda_1(M_0)e^{-3\mu(t)} \leq \lambda_1(M_t) \leq \lambda_1(M_0)e^{3\mu(t)}.
\]
\end{cor}

Once we have control of the second fundamental form, we get control of all higher derivatives along the flow by the smoothing estimates for the mean curvature flow.  In our setting, we can state these estimates succinctly as

\begin{lem}\label{lem: smoothingEstTY}
Suppose $M_t$ has $(C, K, \delta')$-bounded geometry for all $t\in[0,\alpha K^2)$.  Then for all $\ell \geq 0$, there is a constant $C(\ell)$ depending only on $C, \alpha$ and the constants in Corollary~\ref{cor: TYcurvDec} so that, for all $t\in [0, \alpha K^2)$ we have
\[
|\nabla^{\ell}A|^2 \leq C(\ell)\frac{K^{-2}}{t^{\ell}}.
\]
\end{lem}

This result is well-known, but since we have not been able to find a reference in the literature with the dependence we need, we have included a proof in the appendix.  With the smoothing estimates we can turn integral estimates of geometric quantities into pointwise estimates by the following simple lemma; see, for example \cite[Lemma 3.5]{Li}.

\begin{lem}\label{lem: int-pt}
Suppose a Riemannian manifold $(M,g)$ is $\kappa_0$ non-collapsing on scale $r_0$ and suppose $S$ is a tensor with
\[
\int_{M} |S|^2 \leq \epsilon, \qquad |\nabla S| \leq C.
\]
If $\epsilon < r_0^{n+2}$, then
\[
\sup_{M} |S| \leq \left(\frac{1}{\sqrt{\kappa_0}} + C\right)\epsilon^{\frac{1}{n+2}}.
\]
\end{lem}

The next step is to show that exponential decay of the mean curvature, together with a bound on the second fundamental form, implies {\em improved} estimates on the second fundamental form.

\begin{lem}\label{lem: TYAimprove}
Suppose that $M_{0}$ has $(C,K,\delta')$-bounded geometry and let $M_t$ be the solution of the LMCF with initial data $M_0$.  Suppose that, on some interval $[0,T]$, $M_{t}$ has $(4C,K, \frac{\delta'}{n+2})$-bounded geometry and furthermore that there is constant $a>0$ so that
\[
|H(t)|^2 < e^{-\frac{\delta'}{n+2}K^{2n} -at}.
\]
Then, for $K$ sufficiently large depending only on $C, n, \delta'$ we have
\[
|A(t)|^2 \leq  2CK^{-2} +\frac{1}{a}e^{-\frac{\delta'}{n+2}K^{2n}} 
\]
for all $t\in [0,T]$.
\end{lem}
\begin{proof}
By Lemma~\ref{lem: TYlmcfAest} we can assume that $T>\alpha K^2$ for $\alpha$ depending only on $C$.  If $K$ is sufficiently large, depending only on $C$, then we may assume that $T>1$ and that
\begin{equation}\label{eq: TYshortEst}
|A(t)|^2 < 2CK^{-2} \qquad \text{ for } t\in [0,1].
\end{equation}
By the smoothing estimates in Lemma~\ref{lem: smoothingEstTY}, for all $\ell \geq 0$ and $t\in [1, T)$ we have
\begin{equation}\label{eq: TYsmoothApp}
|\nabla^{\ell}A|^2 \leq C(\ell)K^{-2}
\end{equation}
for a constant $C(\ell)$ depending only on $\ell$ and $C$.

The second fundamental form satisfies the following inequality along the flow \cite{Smo}
\begin{equation}\label{eq: ddtNormA}
\ddt |A|^2 \leq 100\left( |A||\nabla^{2}H| + |A|^3|H| + |Rm||H| \right).
\end{equation}
Since the mean curvature decays exponentially, the only problematic term is $|\nabla^{2}H|$.  On the other hand, using~\eqref{eq: TYsmoothApp} and integrating by parts gives
\[
\begin{aligned}
\int |\nabla^2 H|^2 \leq \int |H| |\nabla^{4}H| &\leq \sqrt{C(4)}K^{-1}e^{-\frac{\delta'}{2(n+2)}K^{2n} -\frac{a}{2}t}\\
&\leq e^{-\frac{\delta'}{2(n+2)}K^{2n} -\frac{a}{2}t}
\end{aligned}
\]
for $t>1$, provided $K$ is sufficiently large depending only on $C$.  Since $M_{t}$ has $(4C, K, \frac{\delta'}{n+2})$-bounded geometry, we can apply Lemma~\ref{lem: int-pt} to conclude that, for all $t\in [1,T)$ there holds
\[
|\nabla^2H| \leq (2\sqrt{C} + \sqrt{C(3)}K^{-1})e^{-\frac{\delta'}{2(n+2)^2}K^{2n} -\frac{a}{2(n+2)}t}
\]
provided we choose $K$ sufficiently large, depending only on $\delta', C, n$ so that
\[
e^{-\frac{\delta'}{2(n+2)}K^{2n} -\frac{a}{2}t} \leq \frac{K^{(1-n)(n+2)}}{(4C)^{n+2}}.
\]
Plugging this estimate into~\eqref{eq: ddtNormA} we conclude that, for $K$ sufficiently large, depending only on $\delta', C, n$ and the constants in Corollary~\ref{cor: TYcurvDec} we have
\[
\ddt |A|^2(t) \leq e^{-\frac{\delta'}{2(n+2)^2}K^{2n} -\frac{a}{2(n+2)}t}.
\]
Combining this estimate with~\eqref{eq: TYshortEst} and integrating in time yields
\[
|A(t)|^2 \leq 2CK^{-2} + \frac{2(n+2)}{a} e^{-\frac{\delta'}{2(n+2)^2}K^{2n}} \qquad \text{ for all } t\in [0,T].
\]
\end{proof}

We can now prove the main theorem.

\begin{thm}\label{thm: TYLMCFconv}
Fix constants $C>1$ and $\delta' >0$.  There is a constant $K_0>0$ such that, for all $K \geq K_0$, if $M_{K}$ is a Lagrangian submanifold of $(X, \omega_{TY})$ with Maslov class zero and $(C,K,\delta')$-bounded geometry, then the LMCF starting at $M_K$ converges smoothly and exponentially fast to a special Lagrangian $M_{\infty}$ with $(4C, K, \frac{\delta'}{n+2})$-bounded geometry.
\end{thm}
\begin{proof}
Let $M_t$ be the solution of the LMCF with initial data $M_0= M_{K}$.  First we show that, for $K$ sufficiently large, as long as $M_t$ has $(4C, K, \frac{\delta'}{n+2})$-bounded geometry, the mean curvature decays exponentially.  

Suppose that $M_{t}$ has $(4C, K, \frac{\delta'}{n+2})$-bounded geometry on $[0,T]$.  By Lemma~\ref{lem: TYlmcfAest} we can assume that $T>1$ and $M_t$ has $(2C, K, \delta')$-bounded geometry on $[0,1]$ provided $K$ is large enough depending only on $C$.    For all $t\in [0,T]$ we have
\[
\begin{aligned}
\lambda_1(M_{t}) - \sup_{M_t}|A||H| &\geq (4C)^{-1}K^{-2} - 4CK^{-1}e^{-\frac{\delta'}{2(n+2)}K^2}\\
& > (8C)^{-1}K^{-2} =: aK^{-2}
\end{aligned}
\]
provided $K$ is sufficiently large depending on $C, \delta', n$.  Therefore, by Lemma~\ref{lem: HexpDec} we have
\[
\int_{M_{t}}|H|^2 \leq e^{-aK^{-2}t}\int_{M_0}|H|^2 \leq C^2e^{-\delta'K^{2n}-aK^{-2}t},
\]
since $M_0$ has $(C, K, \delta ')$-bounded geometry.  We now use Lemma~\ref{lem: int-pt} to turn this estimate into a pointwise bound.  By the smoothing estimates in Lemma~\ref{lem: smoothingEstTY} we have
\[
|\nabla A|^2 \leq C(1)K^{-2}
\]
for $t\in[1,T]$ for a constant $C(1)$ depending only on $C, (X, g_{TY})$.  Therefore, Lemma~\ref{lem: int-pt} yields the estimate
\[
|H|(t) \leq (2\sqrt{C} + \sqrt{C(1)}K^{-1})e^{-\frac{a}{n+2}K^{-2}t}C^{\frac{2}{n+2}}e^{-\frac{\delta'}{n+2}K^{2n}}
\]
as long as $K$ is chosen sufficiently large, depending only on $C, \delta'$, so that
\[
e^{-aK^{-2}t}C^2e^{-\delta'K^{2n}} < (4C)^{-(n+2)}K^{-(n-1)(n+2)}.
\]
Increasing $K$ if necessary, we can assume that
\begin{equation}\label{eq: TYexpDecay}
|H|^2(t) \leq e^{-\frac{\delta'}{n+2}K^{2n}-\frac{2a}{n+2}K^{-2}t}
\end{equation}
for all $t\in[1,T]$.  In particular, we have shown that $|H|^2(t)$ decays exponentially as long as $M_t$ has $(4C, K, \frac{\delta'}{n+2})$-bounded geometry and $K$ is chosen sufficiently large depending on $n, \delta', C$ and $(X, g_{TY})$. 

Next we claim that this exponential decay implies that $M_{t}$ has $(4C, K, \frac{\delta'}{n+2})$-bounded geometry for all time.  Define
\[
 T_{max} = \sup\left\{T : M_{t} \text{ has } (4C, K, \frac{\delta'}{n+2}) \text{-bounded geometry } \forall\, t\in [0,T) \right\}
 \]
 First, for $t \in [0, T_{max})$ we estimate
\[
\begin{aligned}
\mu(t) &= 2\int_{0}^{t} (\sup_{M_{t}}|H|)\cdot (\sup_{M_{t}}|A|) dt \\
&\leq 8Ce^{-\frac{\delta'}{2(n+2)}K^{2n}}K^{-1} + 2\sqrt{C}K^{-1}\int_{1}^{T} e^{-\frac{\delta'}{n+2}K^{2n}-\frac{2a}{n+2}K^{-2}t}dt\\
&<\frac{1}{3(n+1)}\log(2)
\end{aligned}
\]
provided $K$ is sufficiently large depending only on $n, C, \delta'$.  In particular, by Corollary~\ref{cor: TYkapLam}, on $M_{t}$ we have
\[
\kappa_0(M_t) > \frac{1}{2}C^{-1}, \qquad r_0(M_{t})= r_0, \qquad \lambda_1(M_t) \geq \frac{1}{2}C^{-1}K^{-2}.
\]
Similarly, we have
\[
\ddt {\rm Vol}(M_{t}) = -\int_{M_t} |H_{t}|^2 \geq -4Ce^{-\frac{\delta'}{n+2}K^{2n}-\frac{2a}{n+2}K^{-2}t}.
\]
Thus, if $K$ is sufficiently large depending only on $n, C, \delta'$, then
\[
C \geq {\rm Vol}(M_0) \geq {\rm Vol}(M_{t}) \geq \frac{1}{2}{\rm Vol}(M_{0}) \geq \frac{1}{2C}.
\]
To control $\ell_0$ we argue as in the proof of Proposition~\ref{prop: TYalmostSLAG}.  Thanks to Remark~\ref{rek: scaleRk} and Proposition~\ref{prop: HSVZdecay}, for $K$ sufficiently large depending only on $(X,g_{TY})$ we have $|\nabla \ell_0^{n+1}|_{g_{TY}}^2 \leq (n+1)^2$.  Therefore
\[
\bigg|\ddt \ell_0^{n+1}\bigg| \leq (n+1) |H(t)| \leq (n+1)e^{-\frac{\delta'}{2(n+2)}K^{2n}-\frac{a}{n+2}K^{-2}t}.
\]
Choosing $K$ sufficiently large depending only on $\delta', n, C$, we can ensure that
\[
\frac{1}{2}\ell_0\big|_{M_0} \leq \ell_0\big|_{M_{t}} \leq 2\ell_0\big|_{M_0}.
\]
Finally, we apply Lemma~\ref{lem: TYAimprove} to conclude that for $K$ sufficiently large depending on $C, \delta', n$, we have
\[
|A(t)|^2 \leq 2CK^{-2} + (n+2)^28CK^2e^{-\frac{\delta'}{2(n+2)^2}K^{2n}} \quad \text{ for all } t\in [0,T].
\]
Increasing $K$ if necessary, we obtain $|A(t)| < 3CK^{-2}$. It follows that $M_{t}$ has $(3C, K, \frac{\delta'}{n+2})$-bounded geometry on $[0,T_{max}]$.  By Lemma~\ref{lem: TYlmcfAest} it follows that $T_{max}= +\infty$ and $M_t$ has $(4C, K, \frac{\delta'}{n+2})$-bounded geometry for all time.  The estimate~\eqref{eq: TYexpDecay} holds for all time and hence $M_{t}$ converges smoothly and exponentially fast to a special Lagrangian with $(4C, K, \frac{\delta'}{n+2})$-bounded geometry.

\end{proof}

As an immediate consequence, we obtain Theorem~\ref{thm: main1Intro} for Tian-Yau spaces. 

\begin{proof}[Proof of Theorem~\ref{thm: main1Intro} for Tian-Yau spaces]
By Proposition~\ref{prop: TYalmostSLAG}, there are constants $C>1, \delta'>0, K_{0}\gg 1$ depending only on $N, (X,\omega_{TY})$ such that, for $K \geq K_{0}$ $(X,\omega_{TY})$ admits Lagrangians $M_{K}$ with $(C,K, \delta')$-bounded geometry and zero Maslov class.  By Theorem~\ref{thm: TYLMCFconv}, after possibly increasing $K_0$ depending only on $C, n, \delta'$ we can assume that the LMCF starting at $M_{K}$ converges to a special Lagrangian submanifold $M_{K,\infty}$ with $(4C, K, \frac{\delta'}{n+1})$-bounded geometry.  Define a sequence $K_{i}$, starting with $K_0$, having $K_{i}= 100C^2 K_{i-1}$.  Then we have
\[
4CK_{i} < (4C)^{-1}K_{i+1}
\]
and so, using the scale function $\ell_0$ we see that $M_{K_{i},\infty}$ is disjoint from $M_{K_{j}, \infty}$ for all $i\ne j$.  
\end{proof}

\begin{rk}\label{rk: disjoint}
Note that the proof, and in particular the exponential decay of the mean curvature, shows the following: for any $\epsilon, C, \delta' >0$, there is a constant $K(\epsilon, C, \delta')$ with the following effect.  If $M_0$ has $(C, K, \delta ')$-bounded geometry for $K \geq K(\epsilon, C, \delta')$, then the Lagrangian mean curvature flow starting from $M_0$ converges to an immersed special Lagrangian $M_{\infty} \subset (X,\omega_{TY}, g_{TY})$ and $M_{\infty} \subset B(M_0, \epsilon)$.  In particular, if $N, N'$ are two special Lagrangians in $D$, then for $\ell_0$ sufficiently large, the LMCF starting from the models $M_0, M_0'$ constructed in Proposition~\ref{prop: TYalmostSLAG} will converge to disjoint special Lagrangians.  Clearly the same result holds for the constructions in Section~\ref{sec: ACYL}.
\end{rk}

\section{Special Lagrangian Fibrations in dimension 2}\label{sec: sLagFib}

In this section we prove that, under fairly general assumptions in complex dimension 2, the existence of a single special Lagrangian torus with primitive homology class and zero self-intersection in a Calabi-Yau surface with controlled geometry implies the existence of a global special Lagrangian torus fibration.  The three main tools we use are the deformation theory of special Lagrangians, hyper-K\"ahler rotation and the moduli and compactness theory of holomorphic curves.

Recall that a hyper-K\"ahler manifold is a Riemannian manifold $(X,g)$ equipped with a triple of parallel, orthogonal, integrable complex structures $(I,J,K)$ satisfying the quaternion relations
\[
I^2=J^2=K^2=IJK=-1.
\]
This data yields an $S^2$ worth of complex structures given by $\{(aI+bJ+cK): a^2+b^2+c^2=1 \}$ on $(X,g)$ compatible with the Riemannian structure, inducing distinct K\"ahler structures on $(X,g)$.  Equivalently \cite{Don06}, in real dimension $4$, a hyper-K\"ahler structure on the oriented manifold $(X, dVol_0)$ is a triple of closed $2$-forms $(\omega_1, \omega_2, \omega_3)$ satisfying, for every $1\leq i\leq j \leq 3$
\[
\begin{aligned}
\frac{1}{2}\omega_i\wedge\omega_{j} &= Q_{ij} dVol_0,\\
\frac{1}{2}\omega_i\wedge\omega_{j} &= \frac{1}{6}\delta_{ij} (\omega_1^2+\omega_2^2+\omega_3^2)
\end{aligned}
\]
for $Q_{ij}$ a positive definite matrix.  The hyper-K\"ahler triple $(\omega_1,\omega_2,\omega_3)$ induces a Riemannian metric $g$ such that each $\omega_j$ is self-dual with respect to $g$.  Such a metric $g$ is called a hyper-K\"ahler metric.  Each form $\omega_i$ is symplectic and induces an integrable complex structure $J_i$ such that $\Omega_i = \omega_j +\sqrt{-1}\omega_k$ $i\ne j\ne k$ is a holomorphic $2$-form.  

In the present setting, we have a Calabi-Yau manifold $(X,g,J,\omega, \Omega)$ of complex dimension $2$ with K\"ahler form satisfying
\[
\omega^2= \frac{1}{2}\Omega \wedge\bar{\Omega}.
\]
Direct calculation shows that $({\rm Re}(\Omega), \omega, {\rm Im}(\Omega))$ is a hyper-K\"ahler triple, with associated complex structures $(I,J,K)$ satisfying the quaternion relations. Finally, associated to the hyper-K\"ahler structure is the twistor space $\mathcal{X}$, a smooth complex manifold diffeomorphic to $X\times \mathbb{P}^1$, but with complex structure on the fiber over $\zeta \in \mathbb{C} = \mathbb{P}^1\setminus\{\infty\}$ given by
\[
J_{\zeta} = \frac{\sqrt{-1}(-\zeta+\bar{\zeta})I - (\zeta + \bar{\zeta})K + (1-|\zeta|^2)J}{1+|\zeta|^2}.
\]
In particular, $\mathcal{X}$ has a non-trivial holomorphic fibration.  The holomorphic volume form on a fiber $(X, g, J_{\zeta})$ for $\zeta \in \mathbb{P}^1$ is given by
\begin{equation}\label{eq: twistVol}
\Omega_{\zeta} = \Omega + 2\zeta \sqrt{-1} \omega - \zeta^2\overline{\Omega}.
\end{equation}

 A crucial point for us is the observation that if $L \subset (X,g, J, \omega, \Omega)$ is special Lagrangian with $\omega|_{L}={\rm Im}(\Omega)|_{L} =0$ and ${\rm Re}(\Omega)|_{L} = dVol_{g}$, then Wirtinger's inequality implies $L$ is a {\em holomorphic} subvariety of the Calabi-Yau manifold $(X,g, I, {\rm Re}(\Omega), \omega-\sqrt{-1}{\rm Im}(\Omega))$.  We will denote this Calabi-Yau manifold by $(X,g, I)$.

We begin with the following lemma.
\begin{lem}\label{lem: openFib}
Suppose $L\subset (X,g, J, \omega_J)$ is a (possibly immersed) special Lagrangian torus and that $[L]^2=0$.  Then $L$ is embedded and there exists a neighborhood $U$ of $L$ and a fibration $\pi:U \rightarrow B$ to a complex manifold $B$, such that the fibers of $\pi$ are special Lagrangian tori.
\end{lem}
\begin{proof}
McLean's deformation theory  for (possibly immersed) special Lagrangian tori \cite{RMc, Hit} implies that, for each non-zero harmonic $1$-form representing a class in $H^{1}(L)$ we obtain a non-trivial deformation $L'$ of $L$.  By hyper-K\"ahler rotating, we obtain holomorphic curves $C\ne C'$ with $[C]=[C']= [L]$.  By assumption $0=[L]^2= [C]\cdot[C']=[L] .[L']$, so $C\cap C' = \emptyset$ and hence $L, L'$ are disjoint.  Furthermore, since $C^2=0$ and the canonical bundle of $(X,I)$ is trivial, the adjunction formula  \cite[Page 69]{BPV} implies that any immersed torus fiber $\pi$ is, in fact, a smooth embedded torus.   It follows that the deformations of $L$ are all disjoint smooth, embedded Lagrangian tori and hence we obtain an open set $U$ containing $L$ such that $\pi:U \rightarrow B$, where $B$ is an open neighborhood of $0 \in H^1(L)$, is a fibration whose fibers are embedded Lagrangian tori.  That the base of the fibration admits a natural complex structure is due to Hitchin \cite{Hit}.
\end{proof}

The remainder of this section is devoted to proving that this local fibration extends to a global fibration.  The basic idea is to prove that the set of points which lie on a (possibly singular) special Lagrangian  $L'$ deformable to $L$ is both open and closed.  We will make heavy use of the theory of holomorphic curves.  Since our manifold is not compact, we need a result to ensure that our holomorphic curves cannot escape to infinity.  We begin by noting the following lemma, which is likely well-known to experts in the field.

\begin{lem}\label{lem: holVol}
Suppose $(X,g, J)$ is a complete K\"ahler manifold.  For $p\in X$, let ${\rm inj}(p)$ denote the injectivity radius and
\[
K(p) = \sup_{B(p, {\rm inj}(p))}|Rm|.
\]
 There is a universal constant $C_1>0$ with the following effect; for any $r < \min\{{\rm inj}(x), C_1K(x)^{-1/2}\}$ if $u:\Sigma \rightarrow X$ is a $J$-holomorphic curve with $x \in u(\Sigma)$ and $u(\del \Sigma) \subset \del B(x, r)$, then
 \[
 {\rm Area}(u(\Sigma)\cap B(x,r)) \geq \frac{\pi}{4}r^2.
 \]
 \end{lem}
 \begin{proof}
This is a standard fact in symplectic geometry, which is based on the fact that holomorphic curves are absolutely area minimizing in their homology class; we refer the reader to \cite{Sik} for a transparent proof.  However, since the dependence on the geometry is not explicit there, we sketch the details.  First, using the Rauch theorems one can easily show that there is a universal constant $R>0$ so that, in normal coordinates centered at $x$, we have
 \[
 \frac{1}{2}g_{Euc} \leq g \leq 2 g_{Euc}
 \]
 on any ball of radius $r <  \min\{{\rm inj}(x), C_1K(x)^{-1/2}\}$.  In particular, it suffices to estimate the area of $u(\Sigma)\cap B(x,r)$ with respect to the Euclidean metric.  We now apply a comparison argument that goes back to Blaschke \cite[p. 247]{Blas}.  Since $u(\del \Sigma) \subset \del B(x,r)$, we can choose a point $x_0 \in u(\del \Sigma)$.  Taking the cone of $u(\del \Sigma) $ over $x_0$ yields a surface $D$ developable onto a disk.  One can then apply the isoperimetric inequality in the plane for non-simple curves \cite{BanPo} to conclude that, in the Euclidean space, the area and length satisfy
 \[
 4\pi{\rm Area_{Euc}}(D) \leq {\rm Length}_{Euc}(\del D) = {\rm Length}_{Euc}u(\del \Sigma).
 \] 
We refer the reader to \cite{Os} for a nice discussion of this argument.  We can now apply  \cite[Proposition 4.3.1]{Sik}.
\end{proof}

\begin{prop}\label{prop: cpctSet}
Let $(X,g)$ be a complete K\"ahler manifold. Fix a point $x_0 \in X$ and let $r(x)= d(x_0,x)$.  Suppose that
\begin{enumerate}
\item The sectional curvature of $(X,g)$ is bounded by a constant $C_2$.
\item There is a non-increasing function $f : [0,\infty) \rightarrow \mathbb{R}_{>0}$ such that $\int_{0}^{+\infty} f(s)ds = +\infty$ and
\[
{\rm inj}(x) \geq f(r(x)).
\]
\end{enumerate}
Let $K$ be a compact set in $X$.  If $\Sigma$ is a connected holomorphic curve with $\Sigma \cap K \ne \emptyset$, $\del\Sigma\subset K$ and ${\rm Area}(\Sigma) \leq A$, then there is a constant $e=e(X,K,A)>0$, independent of $\Sigma$, so that $\Sigma \subset B_{e}(K)$.
\end{prop}
\begin{proof}
Fix $x_0$ and a constant $A$ as in the statement of the proposition.  We may as well assume that $K=B(x_0, R_0)$ for some $R_0 \in \mathbb{N}$.  For $m\geq R_0$ let
\[
\Sigma_m =\Sigma \cap \big( B(x_0, m+1)\setminus B(x_0, m)\big)
\]
and note that each of these sets is either connected or empty.  We will show that there is a constant $N\in \mathbb{N}$, $N\geq R_0+1$ independent of $\Sigma$ and depending only on $X, A$ with the property the following property; if ${\rm Area}(\Sigma) \leq A$, then $\Sigma_m = \emptyset$ for $m\geq N$.

For the sake of contradiction, let us suppose that $\Sigma_N \ne \emptyset$. Cover $\Sigma$ by balls of radius $\delta(x) = 5^{-1}\min\{ {\rm inj}(x), \frac{1}{2}, C_1C_2^{-1/2}\}$ where $C_1$ is the constant appearing in Lemma~\ref{lem: holVol}.  By the Vitali covering lemma we can extract a countable collection of points $x_j$ such that $B(x_j, \delta(x_j))$ are disjoint and 
\[
\Sigma \subset \bigcup_{j=0}^{\infty} B(x_j, 5\delta(x_j)).
\]
Define
\[
N_m = \{j\in \mathbb{N} : x_j \in \Sigma_m \}, \qquad n_m = \#N_m.
\]
Since $\delta(x_j) <1$ for each $j\in N_m$, we have $B(x_j, \delta(x_j)) \cap K=\emptyset$ and hence, since $\del\Sigma \subset K$, $B(x_j, \delta(x_j))$ is disjoint from $\del \Sigma$.  We can therefore apply Lemma~\ref{lem: holVol}, in combination with the fact that the balls $B(x_j, \delta(x_j))$ are disjoint, to obtain
\begin{equation}\label{eq: areaBnd}
\begin{aligned}
{\rm Area}(\Sigma) &\geq \sum _{R_0< m}\sum_{j \in N_m} {\rm Vol}(B(x_j, \delta(x_j)) \cap \Sigma) \\
&\geq  \frac{\pi}{4}\sum _{R_0 <m}\sum_{j \in N_m} \delta(x_j)^2.
\end{aligned}
\end{equation}
Next, we claim that if $m' >m+1 > R_0$ and $n_{m'}\ne 0 \ne n_{m}$, then for every $m<m''<m'$, we have $n_{m''}\ne0$.  This follows easily from connectedness, since if $\Sigma_{m'} \ne \emptyset$ and $\Sigma_{m} \ne \emptyset$, then the same is true for every $m'' \in [m, m']$.  On the other hand, since $\delta(x_j)<\frac{1}{10}$, it is easy to see that 
\[
\Sigma_{m''} \not\subset \bigcup_{j \in N_m\cup N_{m'}} B(x_j, 5\delta(x_j)).
\]
Concretely, no point in $\Sigma \cap B(x_0, m''+\frac{1}{2})$ can be contained in the set on the right-hand side.  But, since $B(x_j, 5\delta(x_j))$ cover $\Sigma$, we conclude that $N_{m''}\ne \emptyset$.  Now, since we assumed that $\Sigma_N \ne \emptyset$, we get
\[
 {\rm Area}(\Sigma) \geq  \frac{\pi}{4}\sum _{R_0 <m <N}\min_{j\in N_m} \delta(x_j)^2.
\]
On the other hand, for each $j \in N_m$ we have
\[
\delta(x_j) \geq 5^{-1}\min\{ f(m+1), \frac{1}{2}, C_1C_2^{-1/2}\}.
\]
Thanks to the assumption that $\int_{0}^{\infty}f(s)ds =\infty$, we conclude that there is a constant $N'\in \mathbb{N}$, depending only on $C_1, C_2, f$ so that if $N\geq N'$ then
\[
{\rm Area}(\Sigma) \geq A+1,
\]
a contradiction.
\end{proof}

\begin{rk}
The integrability of the complex structure is not needed in the proof of Proposition~\ref{prop: cpctSet}.  The result holds even if $J$ is an almost complex structure and the symplectic form $\omega$ is ``uniformly" $J$-tame; see \cite{Sik}.  Y. Groman has pointed out to us that he independently obtained a similar result \cite[Theorem 4.10]{YGro}.
\end{rk}

We can now state the main theorem of this section, whose proof and consequences will occupy the remainder of this section.

\begin{thm}\label{thm: oneToMany}
Let $(X,g)$ be a complete hyper-K\"ahler surface. Fix a point $x_0 \in X$ and let $r(x)= d(x_0,x)$.  Suppose that
\begin{enumerate}
\item The sectional curvature of $(X,g)$ is bounded.
\item There is a non-increasing function $f : [0,\infty) \rightarrow \mathbb{R}_{>0}$ such that $\int_{0}^{+\infty} f(s)ds = +\infty$ and
\[
{\rm inj}(x) \geq f(r(x)).
\]
\item $X$ has finite Euler characteristic; $\chi(X) < +\infty$.
\end{enumerate}
Assume that there exists a (possibly immersed) special Lagrangian torus $L$ with $[L]\in H_{2}(X,\mathbb{Z})$ primitive and $[L]^2=0$.  Then
\begin{enumerate}
\item $X$ admits a special Lagrangian fibration with $L$ as one of the fibers.
\item There are at most $\chi(X)$ singular fibers, each classified by Kodaira and no fiber is multiple.
\item $L$ is a smooth embedded torus.
\end{enumerate}
\end{thm}

\begin{rk}\label{rk: multFib}
The assumption that $[L]$ is primitive in $H_{2}(X,\mathbb{Z})$ is not fundamental and can be weakened.  However, since it holds in all cases we have considered and streamlines parts of the argument, we have included it for convenience.
\end{rk}

Let us briefly recall the current state of affairs.  We have a complete non-compact Calabi-Yau surface $(X,g,J, \omega_J, \Omega_{J})$ with bounded curvature, which we can think of as a Tian-Yau space of either Type I or Type II, though our results apply in a rather general setting.  In addition, $(X,g, J, \omega_J, \Omega_{J})$ contains a special Lagrangian torus $L$, with $[L]^2=0$.  By Lemma~\ref{lem: openFib}, this special Lagrangian generates a local fibration.  We now hyper-K\"ahler rotate to a complex structure $I$ so that $(X,g,I, \omega_I, \Omega_{I})$ is again a Calabi-Yau surface, but now $L$ becomes a holomorphic submanifold of genus $1$ and $L$ has a neighborhood admitting a holomorphic genus $1$ fibration.  We are going to consider the moduli space of such submanifolds.

Let $\Sigma$ denote a smooth surface of genus $1$ and consider the space
\[
\mathcal{M}([L], I) = \{ u: (\Sigma, j) \rightarrow (X,g, I): u \text{ is holomorphic, } u_{*}[\Sigma] = [L] \},
\]
the moduli space of parametrized $I$-holomorphic maps from $\Sigma$ into $(X, g, I, \omega_I)$ having image homologous to our fixed elliptic curve $L$.  Note that we are not fixing the complex structure $j$ on $\Sigma$, which we allow to vary over the moduli space.   We let $\mathcal{M}_1$ denote the connected component of $\mathcal{M}$ containing the fixed holomorphic curve $L$ and let $X_1 \subset X$ be the set of points lying on $u(\Sigma)$ for some $u\in \mathcal{M}_1$.  By Lemma~\ref{lem: openFib}, $X_1$ is open and thanks to Hitchin \cite{Hit} $\mathcal{M}_1$ has a canonical complex structure.  The goal of the remainder of this section will be to prove that $X\setminus X_1$ consists of finitely many singular elliptic curves, each classified by Kodaira.  For simplicity, denote by $X_2= \del X_1$ and note that $X_2$ is closed.  As a first step we observe

\begin{lem}
If $p \in \del X_1$, then there is an $I$-holomorphic cusp curve  $u:\cup_{\alpha}\Sigma_{\alpha} \rightarrow (X,g,I)$ with $p \in \cup_{\alpha}u(\Sigma_{\alpha})$.  Furthermore, $u\notin \mathcal{M}_1$.
\end{lem}
\begin{proof}
This follows immediately from compactness theory for holomorphic curves.  If $u_i: \Sigma \rightarrow (X,g,I)$ is a sequence of holomorphic maps such with points $p_i\in u_i(\Sigma)$ and $p_i\rightarrow p$, then by Proposition~\ref{prop: cpctSet}, the holomorphic curves $u_i(\Sigma)$ all lie in a fixed compact set.  Since every curve in $\mathcal{M}_{1}$ has the same volume and hence energy, the standard Gromov-Sacks-Uhlenbeck compactness theory  \cite{McDSal, Ye, Gr, SacksUhl} implies that $u_i$ converges to a cusp curve (or stable curve) $u :\cup_{\alpha}\Sigma_{\alpha}\rightarrow (X,g,I)$.  Here $\cup_{\alpha}\Sigma_{\alpha}$ is some tree of Riemann surfaces; its precise structure is irrelevant for our current considerations.  If $\cup_{\alpha}\Sigma_{\alpha} = \Sigma$ irreducible, then $u:\Sigma \rightarrow (X,g,I)$ and by the deformation theory of holomorphic Lagrangians, Lemma~\ref{lem: openFib}, we obtain that $p$ is in the interior of $X_1$, a contradiction. 
\end{proof}

Let $\mathcal{M}_2$ denote the space parametrizing $I$-holomorphic cusp curves, or stable maps, appearing as limits of holomorphic curves in $X_1$.  Recall \cite{Ye} that a cusp curve in $(X,I)$ is a disjoint union $\cup_{\alpha}\Sigma_{\alpha}$ of finitely many Riemann surfaces $\Sigma_{\alpha}$, together with an identification of a finite number of points (called ``nodes") and a holomorphic curve $u:\cup_{\alpha}\Sigma_{\alpha}\rightarrow X$, compatible with the identification.  Any such holomorphic curve has connected image.  Recall that a holomorphic curve $u:\Sigma \rightarrow X$ is called multiply covered if $u= \hat{u}\circ \pi$ where $\pi:\Sigma \rightarrow \Sigma'$ is a holomorphic branched cover of degree larger than $1$. $u$ is called simple if it is not multiply covered.  If $u$ is a multiply covered holomorphic map from $\mathbb{P}^1$, then by the Riemann-Hurwitz formula, $\Sigma'= \mathbb{P}^1$ also.

Given a holomorphic cusp curve $C= u(\Sigma)$ we will denote by $C^{(k)}$ the irreducible components of $C$, $n_k$ their multiplicities and denote by
\[
u_{k}:\Sigma_k \rightarrow C^{(k)}
\]
the associated simple holomorphic curve.

\begin{lem}\label{lem: multBnd}
Suppose $x \in X_2$.  Then there is a number $N(x) \in \mathbb{N}$, depending only on $x$, such that any component of a cusp curve containing $x$ and corresponding to a point $u \in \mathcal{M}_2$ has number of components (counted with multiplicity) bounded by  $N(x)$.
\end{lem}
\begin{proof}
Suppose $u_{i} \in \mathcal{M}_1$ is a sequence of holomorphic maps Gromov-Sacks-Uhlenbeck converging to a stable map $u$ such that $x\in \mbox{Im}(u)$.  Write the image of $u$ as $C= \sum_{k} n_k C^{(k)}$ for $C^{(k)}$ reduced irreducible holomorphic curves.  Notice that $\mbox{Vol}_g\big(u_i(\Sigma_i)\big)=\int_{\Sigma}u_i^*\omega_I=\mbox{Vol}_g(L)$ is fixed since the $u_i$ are homotopic to each other. By Proposition~\ref{prop: cpctSet}, all the curves $u_i(\Sigma)$ fall in a fixed compact set $K \subset X$ and hence $C\subset K$.  From the convergence we have
\[
{\rm Vol}_{g}(L) = \int_{C} \omega_I = \sum_{k}n_{k} \int_{(C_k)_{reg}}\omega_{I}.
\]
On the other hand, since the sectional curvature and injectivity radius are bounded below, $\int_{(C_{k})_{reg}} \omega_I \geq \hbar >0$ for some constant $\hbar$ depending on $x$ \cite[Proposition 4.3.1]{Sik}.  The lemma follows.
\end{proof}

Since $L$ moves in a local fibration, the general fiber is disjoint from any singular curve $C$ obtained as a limit of curves in $\mathcal{M}_1$.  Thus, if we write $[C]=\sum_k n_k[C^{(k)}]$, then $0=[L].[C^{(k)}]=[C].[C^{(k)}]$ (see also \cite[Proposition III.8.2]{BPV}).  Thus, we obtain

\begin{lem}\label{lem: quadForm}
Suppose that $C= \sum_{k=1}^{m} n_k C^{(k)}$ is a singular holomorphic curve obtained as a Gromov-Sacks-Uhlenbeck limit of holomorphic curves in $\mathcal{M}_1$.  Let $Q$ denote the negative intersection form on the components $[C^{(k)}]$ with components $q_{ij} = -[C^{(i)}].[C^{(j)}]$.  Then $Q$ is positive semi-definite and the annihilator of $Q$ is one-dimensional and spanned by $[C] = \sum_{k}n_k[C^{(k)}]$.  In particular, if there is a component $[C^{(\ell)}]$ such that $[C^{(\ell)}]^2=0$, then $C =n_{\ell}C^{(\ell)}$ has only one component.

\end{lem}
\begin{proof}
Consider the vector space over $\mathbb{Q}$ spanned by the classes $[C^{(k)}]$ with the quadratic form defined as above.

Then we have $q_{pk} \leq 0$ for all $p \ne kl$.  Furthermore, since $C$ is connected, there is no partition of $\{1,\ldots, m\}$ into non-empty disjoint sets $P, K$ so that $q_{pk}=0$ for all $p\in P$, $k \in K$.  Finally, since
\[
[C] = \sum_{k=1}^{m}n_k [C^{(k)}]
\]
for $n_k >0$ and $[C]^2=0$, we can apply \cite[Lemma I.2.10]{BPV} to conclude that $Q\geq 0$ and the annihilator of $Q$ is spanned by $[C]$.  The lemma follows.
\end{proof}

We can now classify the singular fibers appearing in $\mathcal{M}_2$.

\begin{prop}\label{prop: singClass}
Suppose $C$ is a singular holomorphic curve obtained as a Gromov-Sacks-Uhlenbeck limit of holomorphic curves in $\mathcal{M}_1$.  Write $C= \sum_{k=1}^{m}n_k C^{(k)}$, with $C_{k}$ reduced and irreducible.  Then the components $C^{(k)}$ satisfy $[C^{(k)}]^2= 0,-2$, and
\begin{enumerate}
\item if $[C^{(k)}]^2=0$, for some $k$, then $C$ has one component and is a singular fiber of Kodaira type $I_1$ or $II$.
\item if $[C^{(k)}]^2=-2$, for all $k$, then $C$ is a singular fiber of Kodaira  type $III, I_n, IV, I_0^*, I_n^*, IV^*, III^*$ or $II^*$.
\end{enumerate}
\end{prop}

\begin{proof}

Suppose we have a sequence of $I$-holomorphic curves $u_{k}: \Sigma \rightarrow X$ converging in the sense of Gromov-Sacks-Uhlenbeck to a cusp curve $u: \cup_{k} \Sigma_{k} \rightarrow X$.  As discussed above we let 
\[
u_{k}: \Sigma_{k} \rightarrow X
\]
be the simple holomorphic curves, $[C^{(k)}] = (u_k)_{*}[\Sigma_k]$ (which may be zero if $u_k$ is constant) and write
\[
[L] = \sum_{k} n_k [C^{(k)}]
\]
for positive integers $n_k$. First, thanks to the Gromov-Sacks-Uhlenbeck compactness theorem \cite{McDSal, Ye, Gr, SacksUhl}, only Riemann surfaces with genus $1$ or $0$ appear in the limit. 

Suppose that the map $u_\alpha :\Sigma_\alpha \rightarrow X$ is non-constant.  By the Riemann-Hurwitz formula, $u_\alpha$ is either simple or factors through a branched covering $\pi: \Sigma_\alpha \rightarrow \hat{\Sigma}_{\alpha}$; in order to lighten notation we will denote by $v_\alpha: \hat{\Sigma}_{\alpha} \rightarrow X$ the simple holomorphic map. Keep in mind that $\hat{\Sigma}_{\alpha}$ can be either the sphere or the torus, with the latter case occurring if $u_\alpha$ factors through an unramified cover of the torus (thanks to the Riemann-Hurwitz formula).  Consider the curve $C^{(\alpha)} = v_\alpha(\hat{\Sigma}_{\alpha})$ and let 
\[
\nu: \tilde{C}_{\alpha} \rightarrow C^{(\alpha)}
\]
be the normalization.  Since $\hat{\Sigma}_{\alpha}$ has genus $0$ or $1$, it follows easily from the universal property of the normalization and Riemann-Hurwitz that $\tilde{C}_{\alpha}$ has genus $0$ or $1$.  Furthermore, if $\hat{\Sigma}_{\alpha}$ has genus zero, then so does $\tilde{C}_{\alpha}$.  We now appeal to the adjunction formula \cite[Page 69]{BPV} which gives
\begin{equation}\label{eq: adjunct}
{\rm genus}(\tilde{C}_{\alpha}) + \delta = 1 + \frac{1}{2} (K_{X}+[C^{(\alpha)}]).[C^{(\alpha)}]
\end{equation} 
where the $\delta$ invariant is given by
\[
\delta =  \sum_{x\in C_{\alpha}} {\rm dim}_{\mathbb{C}}(\nu_{*}\mathcal{O}_{\tilde{C}_{\alpha}}/\mathcal{O}_{C_{\alpha}}). 
\]
Recall that $X$ has $K_X= \mathcal{O}_{X}$ and $[C^{(\alpha)}]^2\leq 0$ by Lemma~\ref{lem: quadForm}.  

Let us first treat the case that $\tilde{C}_{\alpha}$ has genus $1$ for some $\alpha$.  In this case we must have that $[C^{(\alpha)}]^2= 0$, $\delta =0$ and $\hat{\Sigma}_{\alpha}$ has genus $1$.  In particular, by Lemma~\ref{lem: quadForm}, $C$ has only one component and $u$ factors through an unramified cover of the torus.  But since $[L] = (u)_{*}[\cup_{\alpha}\Sigma_{\alpha}]$ is primitive in $H_{2}(X,\mathbb{Z})$, we must have that $\cup_{\alpha}\Sigma_{\alpha}$ is a single Riemann surface of genus $1$. Since $\delta =0$, $\nu$ is an isomorphism and $C$ is smooth.  An application of Riemann-Hurwitz implies that $\hat{\Sigma}_{1}$ is biholomorphic to $\tilde{C}$ and hence $u_1$ is an embedding.  Thus $C= C^{(1)}$ is a smooth elliptic curve, which yields a contradiction. 

We may therefore assume that $\tilde{C}_{\alpha}$ has genus $0$ for all $\alpha$.  Equation~\eqref{eq: adjunct} becomes
\[
\delta = 1+ \frac{1}{2}[C^{(\alpha)}]^2
\]
and so either $[C^{(\alpha)}]^2=0$ or $-2$.  If $[C^{(\alpha)}]^2=0$, then by Lemma~\ref{lem: quadForm}, $C= n_{\alpha}C^{(\alpha)}$ has only one component, which has $\delta =1$.  By an exercise in algebraic geometry (see, for example \cite[Chapter 1, Exercise 4]{Fried}), $C^{(\alpha)}$ has either a single ordinary double point or a single cusp and hence is (a positive integer multiple of) a fiber of Kodaira type $I_1$ or $II$.

It remains to consider the case when $[C^{(\alpha)}]^2=-2$ for all $\alpha$.  In this case each $C^{(\alpha)}$ is a smooth rational curve by the adjunction formula and $C$ must have more than one component.  Furthermore, by Lemma~\ref{lem: quadForm}, for any $\alpha\ne \beta$ we have
\[
0 \geq \left([C^{(\alpha)}]+[C^{(\beta)}]\right)^2 = -4 + 2[C^{(\alpha)}].[C^{(\beta)}]
\]
and so $[C^{(\alpha)}].[C^{(\beta)}]\leq 2$ and by Lemma~\ref{lem: quadForm} equality is achieved if and only if $[C] = n([C^{(\alpha)}] + [C^{(\beta)}])$.  In particular, $C$ is a multiple of a fiber of Kodaira type $I_2$ or $III$.  We are reduced to considering the case when $0\leq [C^{(\alpha)}].[C^{(\beta)}] \leq 1$ for all $\alpha\ne \beta$.  We can now apply directly \cite[Lemma 2.12]{BPV} to conclude that the intersection matrix is of type $\tilde{A}_{n}, \tilde{D}_{n}$, or $\tilde{E}_k$ for $k=6,7,8$.    By inspection these yield singular fibers of type $I_n$ (or type IV if $n=3$), or of type $I_0^*, I_n^*, IV^*, III^*, II^*$. 

It only remains to rule out the case of multiple fibers, but this follows from the assumption that $[C]=[L]$ is primitive in $H_{2}(X,\mathbb{Z})$.
\end{proof}

\begin{rk}\label{rk: redCompExist}
Note that every singular fiber appearing in the Kodaira classification result Proposition~\ref{prop: singClass} contains a component with multiplicity one.
\end{rk}

Next we will show that, in fact, $X_1\cup X_2 = X$.  The main technical issue is to prove that the set of points lying on singular elliptic curves is ``discrete" in an appropriate sense.  Geometrically, we will prove that, for any singular elliptic curve $C$ lying in $X_2$, there is an $\epsilon>0$, such that the $\epsilon$-neighborhood $B(C, \epsilon) \subset (X,g)$ contains no other singular elliptic curve in $X_2$.  We begin by proving this statement when all the singular curves under consideration have multiple components.

\begin{lem}\label{lem: multAcc}
Suppose $C \subset X_2$ has $m$ irreducible components for some $m\geq 2$.  Then there exists $\epsilon>0$ such that $B(C, \epsilon)$ does not contain any singular curve $C'\ne C$ homologous to $[L]$ with more than one component.
\end{lem}
\begin{proof}
Suppose not.  Write $C = \sum_{k=1}^{m}n_k C^{(k)}$.  For every $\ell \in \mathbb{N}$ there is a singular elliptic curve $C_\ell$ in $B(C, \ell^{-1})$.  By Lemma~\ref{lem: multBnd}, we can assume that each curve $C_\ell$ has $m \geq 2$ components with multiplicity $n'_k$.  That is, we can write
\[
C_{\ell} = \sum_{k =1}^{m} n'_{k} C^{(k)}_{\ell},
\]
for $C^{(k)}_{m}$ smooth, irreducible rational curves, thanks to Proposition~\ref{prop: singClass}.  It follows from the Mayer-Vietoris theorem that, for all $\ell$ sufficiently large, $H_{2}(B(C, \ell^{-1}), \mathbb{Z})$ is generated (over $\mathbb{Z})$ by $[C^{(k)}]$ for $1\leq k \leq m$.  Thus, we can assume that $[C_{\ell}^{(k)}] = [C^{(k)}]$ for all $m$ sufficiently large.  On the other hand, since $C_\ell \ne C$, we must have that, for some $k$ and some $\ell$ sufficiently large, $C^{(k)}_{\ell} \ne C^{(k)}$.  Thus
\[
[C^{(k)}]^2= [C^{(k)}].[C^{(k)}_{\ell}] \geq 0
\]
But by Lemma~\ref{lem: quadForm} we have $[C^{(k)}]^2=-2$, a contradiction.
\end{proof}

The next step is to rule out the accumulation of singular curves with only one component at a singular curve with only one component.  By Proposition~\ref{prop: singClass}, singular curves with only one component correspond to reduced and irreducible divisors obtained as holomorphic images of $\mathbb{P}^1$ with either nodal or cuspidal singularities.  So, suppose we have a simple holomorphic curve
\[
u : \mathbb{P}^1\rightarrow C\subset (X,g, I, \omega_I)
\]
such that $[C] = [L]$.  Suppose that, for all $\ell\geq 0$, $B(C, \ell^{-1})$ contains a reduced, irreducible rational curve $C_\ell \subset X_2$, corresponding to a simple holomorphic map
\[
u_{\ell}: \mathbb{P}^1 \rightarrow C_{\ell}\subset (X,g, I, \omega_I),
\]
with $[C_\ell] = [L]$. We will address the nodal and cuspidal cases separately, but first we prove a general lemma.

\begin{lem}\label{lem: GromConv}
Suppose that $u:\mathbb{P}^1 \rightarrow C\subset (X,g,I, \omega_I)$ is a simple, holomorphic curve with $[C]= [L]$ and $C$ is reduced and irreducible.  Suppose that for every $\ell\in \mathbb{N}$ there is a simple holomorphic curve $u_{\ell}:\mathbb{P}^1 \rightarrow B(C, \ell^{-1})$ such that $(u_{\ell})_{*}[\mathbb{P}^1]= [L]$.  Then $u_{\ell}$ converges to $u$ in the sense of Gromov-Sacks-Uhlenbeck.
\end{lem}
\begin{proof}
By Proposition~\ref{prop: cpctSet} up to taking a subsequence, $u_{\ell}$ Gromov-Sacks-Uhlenbeck converges to a cusp curve, or stable map $u': \cup_{\alpha} \Sigma_{\alpha}\rightarrow (X,g, I, \omega_I) $, with each $\Sigma_{\alpha} = \mathbb{P}^1$, whose image is contained in $C$ \cite{Gr, McDSal, Ye}.  Since $C$ is irreducible, the image of $u'$ must be $C$.  Furthermore, since $(u_{\ell})_{*}[\mathbb{P}^1] = [C_\ell]=[C]$, we have that $(u')_{*}[\cup_{\alpha} \Sigma_{\alpha}] = [C]$. It follows that $u'$ is constant on all but one component of $\cup_{\alpha} \Sigma_{\alpha}$.  Forgetting the constant components of the map, we obtain $u':\mathbb{P}^1 \rightarrow C$ and since $(u')_*[\mathbb{P}^1]=[C]$, the map $u'$ is simple.  Therefore, $u' = u\circ \tau$ for some $\tau \in PSL(2,\mathbb{C})$.  Therefore $u_\ell$ converges in the sense of Gromov-Sacks-Uhlenbeck to $u$.  
\end{proof}

\begin{prop}\label{prop: nodalAcc}
Suppose $C\subset X_2$ is a reduced and irreducible rational curve with a nodal singularity obtained as a Gromov-Sacks-Uhlenbeck limit of $I$-holomorphic curves in $\mathcal{M}_1$.  Then there exists an $\epsilon >0$ such that $B(C, \epsilon)$ does not contain any irreducible singular curve in $\mathcal{M}_2$ distinct from $C$.
\end{prop}
\begin{proof}
We begin with a simple calculation.  By assumption, the rational curve $u: \mathbb{P}^1 \rightarrow X$ is nodal and hence $du$ is injective.  Therefore, we have an injection of holomorphic vector bundles on $\mathbb{P}^1$ by
\[
0\rightarrow T\mathbb{P}^1 \rightarrow u^{*}TX.
\]
By Grothendieck's theorem \cite{Grot} and the fact the $c_1(u^{*}TX)=0$, $u^{*}TX$ splits as a direct sum $u^{*}TX = \mathcal{O}_{\mathbb{P}^1}(a) \oplus  \mathcal{O}_{\mathbb{P}^1}(-a)$ for some $a\geq 0$.  Since $T\mathbb{P}^1 = \mathcal{O}_{\mathbb{P}^1}(2)$, there is a nowhere vanishing section of $\mathcal{O}_{\mathbb{P}^1}(a-2) \oplus  \mathcal{O}_{\mathbb{P}^1}(-a-2)$.  This immediately implies $a=2$ and hence
\[
 u^{*}TX = \mathcal{O}_{\mathbb{P}^1}(2) \oplus  \mathcal{O}_{\mathbb{P}^1}(-2).
\]
We now consider the twistor space of $X$, which we denote by $\mathcal{X}$.  By the fiber exact sequence we have
\[
0 \rightarrow TX \rightarrow T\mathcal{X}|_{X} \rightarrow \mathcal{O}_{X}\rightarrow 0.
\]
Composing the holomorphic map $u$ with the inclusion, we obtain
\[
0\rightarrow  \mathcal{O}_{\mathbb{P}^1}(2) \oplus  \mathcal{O}_{\mathbb{P}^1}(-2)\rightarrow u^{*}T\mathcal{X}|_{X} \rightarrow \mathcal{O}_{\mathbb{P}^1} \rightarrow 0.
\]
We need the following lemma
\begin{lem}\label{lem: nonSplit}
In the above notation, we have
\[
u^{*}T\mathcal{X}|_{X} = \mathcal{O}_{\mathbb{P}^1}(2) \oplus  \mathcal{O}_{\mathbb{P}^1}(-1) \oplus \mathcal{O}_{\mathbb{P}^1}(-1).
\]
\end{lem}
Let us assume the lemma for now and finish the proof.  Let $\iota: X \hookrightarrow \mathcal{X}$ be the inclusion of $X$ into the twistor space by the complex structure $I$.  Combining Lemma~\ref{lem: nonSplit} with \cite[Lemma 3.3.1]{McDSal} and \cite[Theorem 3.1.6]{McDSal} we conclude that the moduli space of parametrized, simple holomorphic rational curves in the twistor space homologous to $\iota_*[C]$ is a smooth manifold with Gromov-Sacks-Uhlenbeck topology. Moreover, it has real dimension $6$ by the Riemann-Roch theorem (see also \cite[Theorem 3.1.6]{McDSal}).  On the other hand, the $3$-complex dimensional group $PSL(2,\mathbb{C})$ acts on the holomorphic rational curves by reparametrization.  It follows that there is an open neighborhood $U$ of $u$ in the Gromov-Sacks-Uhlenbeck topology such that, if $u' \in U$, then $u'$ is obtained from $u$ by pre-composing with a M\"obius transformation of $\mathbb{P}^1$; in particular, $u(\mathbb{P}^1) = u'(\mathbb{P}^1)$.  Since any $I$-holomorphic rational curve into $X$ homologous to $[C]$ induces a holomorphic rational curve into $\mathcal{X}$ homologous to $\iota_*[C]$, we deduce that the same result holds true for $X$.

Assume for the sake of a contradiction that there are rational curves $u_{\ell}:\mathbb{P}^1 \rightarrow (X,g,I)$ in $\mathcal{M}_2$ such that $C_\ell= u_{\ell}(\mathbb{P}^1) \subset B_{\ell^{-1}}(C)$, but $C_\ell \ne C$.  By Lemma~\ref{lem: GromConv}, the rational curves $u_{\ell}$ Gromov-Sacks-Uhlenbeck converge to $u$.  In particular, for $\ell$ sufficiently large $u_\ell = u\circ \tau$ for some $\tau \in PSL(2,\mathbb{C})$.  But this implies $u_{\ell}(\mathbb{P}^1) = u(\mathbb{P}^1)$, a contradiction.

\end{proof}

It only remains to prove Lemma~\ref{lem: nonSplit}.
\begin{proof}[Proof of Lemma~\ref{lem: nonSplit}]
Consider the exact sequence of vector bundles
\begin{equation}\label{eq: twistorSeq}
0 \rightarrow TX \hookrightarrow T\mathcal{X}\big|_{X} \rightarrow \mathcal{O}_{X} \rightarrow 0.
\end{equation}
Restricting to $C$ and pulling back gives the exact sequence
\[
0\rightarrow  \mathcal{O}_{\mathbb{P}^1}(2) \oplus  \mathcal{O}_{\mathbb{P}^1}(-2)\rightarrow u^{*}T\mathcal{X}|_{X} \rightarrow \mathcal{O}_{\mathbb{P}^1} \rightarrow 0.
\]
An easy computation shows that ${\rm dim}_{\mathbb{C}}{\rm Ext}^1(\mathcal{O}_{\mathbb{P}^1},  \mathcal{O}_{\mathbb{P}^1}(2) \oplus  \mathcal{O}_{\mathbb{P}^1}(-2))=1$ and so it suffices to show that the exact sequence is not split.  Taking the long exact sequence in cohomology, this question is reduced to understanding the connecting homomorphism 
\[
\delta: H^{0}(\mathbb{P}^1, \mathcal{O}_{\mathbb{P}^1}) \rightarrow H^{1}(\mathbb{P}^1, u^{*}TX).
\]
In particular, it is easy to see that if $\delta$ is not the zero map, then the exact sequence cannot be split.

On the other hand, since the twistor family $\mathcal{X} \rightarrow \mathbb{P}^1$ is a non-trivial deformation of complex structures, the Kodaira-Spencer map
\[
\delta: H^{0}(X,\mathcal{O}_{X}) \rightarrow H^{1}(X, TX)
\]
of the long exact sequence associated with~\eqref{eq: twistorSeq} is non-trivial.  By a well-known computation (see, for example \cite[Lemma 7.2]{Tian}), the contraction of the image of the Kodaira-Spencer map (viewed as a $(TX, I)$-valued $(0,1)$ form) with the holomorphic $2$-form is the $(1,1)$ component of $\frac{d}{d\zeta}|_{\zeta= \sqrt{-1}} \Omega_{\zeta}$.  By~\eqref{eq: twistVol} we have
\[
\frac{d}{d\zeta}\bigg|_{\zeta=\sqrt{-1}} \Omega_{\zeta} = 2\sqrt{-1}\omega- 2\sqrt{-1}\,\overline{\Omega} = 2\sqrt{-1}(\omega+\sqrt{-1}{\rm Im}(\Omega)) -2\sqrt{-1}{\rm Rm}(\Omega).
\]
Since $(\omega+\sqrt{-1}{\rm Im}(\Omega))$ is holomorphic on $X_{I}$, we are reduced to considering $-2\sqrt{-1}{\rm Re}(\Omega)$.

It suffices to show that the restriction of the deformation of complex structures to the nodal rational curve $C$ is non-trivial.  In other words, it suffices to show that
 \[
 u^{*}{\rm Re}(\Omega) \in H^{1}(\mathbb{P}^1, u^{*}\Lambda^{1,1}T^{*}X_{I})
 \]
 is non-trivial.  But this is clear, since $u^{*}{\rm Re}(\Omega)$ is a K\"ahler metric on $\mathbb{P}^1$ and hence cannot be in the image of $\dbar: u^{*}\Lambda^{1,0}T^*X_{I} \rightarrow u^{*}\Lambda^{1,1}T^*X_{I}$ for otherwise we would have
 \[
 0 < \int_{C}{\rm Re}(\Omega) = \int_{\mathbb{P}^1} \dbar(u^{*}\beta) = 0.
 \]
\end{proof}

We next rule out the accumulation of irreducible rational curves at a cuspidal rational curve.

\begin{prop}\label{prop: cuspAcc}
Suppose $C\subset X_2$ is a reduced and irreducible rational curve with cuspidal singularities obtained as a Gromov-Sacks-Uhlenbeck limit of $I$-holomorphic curves in $\mathcal{M}_1$.  Then there exists an $\epsilon >0$ such that $B(C, \epsilon)$ does not contain any rational curve in $X_2$ homologous to $[C]= [L]$ and distinct from $C$.
\end{prop}
\begin{proof}
The idea is to show that if there exist irreducible rational curves $C_{\ell} \ne C$ in $X_1$ contained in $B_{\ell^{-1}}(C)$, then in fact $C$ deforms in a real $2$ dimensional family of holomorphic curves sweeping out a neighborhood of the generic point of $C$.  In particular, there will be rational curves homologous to $[C]$ intersecting the smooth elliptic curves homologous to $[C]$ in $X_1$, which is impossible since $[C]^2=0$.

In order to do this we must first study the deformation theory of the cuspidal rational curve $C$.  Recall \cite[Chapter 2]{McDSal} that if $u$ is an $I$-holomorphic curve (for some almost complex structure $I$) with $u_{*}[\mathbb{P}^1] = C$, then we can deform $u$ for $\lambda$ sufficiently small by
\[
u_{\lambda}= \exp_{u}(\lambda \xi),
\]
where $\xi$ is any smooth (or more generally $W^{k,p}$) section of $u^{*}TX \rightarrow \mathbb{P}^1$.  Note that we are considering {\em parametrized} $I$-holomorphic curves.  Under this identification, one obtains the linearized $I$-holomorphic curve operator
\begin{equation}\label{eq: linDBAR}
D_{(u,I)}: W^{k,p}(\mathbb{P}^1, u^{*}TX) \rightarrow W^{k-1,p}(\mathbb{P}^1, \Lambda^{0,1}T^{*}\mathbb{P}^1\otimes u^{*}TX)
\end{equation}
where $k, p$ are chosen sufficiently large.  The moduli space of $I$-holomorphic curves with $u_{*}[\mathbb{P}^1] = [C]$ will be a smooth manifold near $(u,I)$ provided $D_{(u,I)}$ is surjective.

We begin by showing that, in the present case $D_{(u,I)}$ is in fact {\em not} surjective by applying (the proof of) \cite[Lemma 3.3.1]{McDSal}. By assumption, $u: \mathbb{P}^1 \rightarrow (X,g, I)$ is a cuspidal rational curve and so $du$ has a simple zero at some point $p\in\mathbb{P}^1$.  In particular, $du$ induces an injective map
\[
du: T\mathbb{P}^1 \otimes \mathcal{O}_{\mathbb{P}^1}(1) \rightarrow u^{*}TX.
\]
Since $u^{*}TX$ has degree zero, Grothendieck's theorem implies that $u^{*}TX = \mathcal{O}_{\mathbb{P}^1}(3)\oplus \mathcal{O}_{\mathbb{P}^1}(-3)$ and Serre duality yields
\[
H^{1}(\mathbb{P}^1,  \mathcal{O}_{\mathbb{P}^1}(3)\oplus \mathcal{O}_{\mathbb{P}^1}(-3)) = H^0(\mathbb{P}^1,  \mathcal{O}_{\mathbb{P}^1}(-5))\oplus H^0(\mathbb{P}^1,\mathcal{O}_{\mathbb{P}^1}(1)).
\]
By \cite[Lemma 3.3.1]{McDSal}, $D_{(u,I)}$ has a $4$ real dimensional cokernel. By \cite[Proposition 3.1.11]{McDSal} the index of $D_{(u,I)}$ is $4$ and hence the kernel is $8$ dimensional.  

In order to obtain a smooth moduli space, we must expand the set of almost complex structures under consideration.  By a standard argument, we can construct a smooth family $\pi: \mathcal{I} \rightarrow D\subset \mathbb{R}^4$ of (not necessarily integrable) almost complex structures compatible with the given symplectic form $\omega_0 := \omega_{I_0}$ over a disk $D\subset \mathbb{R}^4$ with $I_0:= \pi^{-1}(0)= I$ such that the expanded moduli space
\[
\mathcal{M}([C], \mathbb{P}^1; \mathcal{I}) = \{ (u, I_t) : t \in D, \, u:\mathbb{P}^1 \rightarrow (X, \omega_{0}, I_t) \text{ is $I_t$-holomorphic}\}
\]
becomes a smooth manifold in a neighborhood of $(u,I_0)$.  More precisely, the extended linearized operator
\begin{equation}\label{eq: extDU}
\hat{D}_{(u,I_t)}(\xi, Y) = D_{(u, I_t)}\xi + \frac{1}{2}Y(u)\circ du \circ j_{\mathbb{P}^1}
\end{equation}
regarded as a map
\[
\hat{D}_{(u,I_t)}: W^{k,p}(\mathbb{P}^1, u^{*}TX) \times T_{I_{t}}\mathcal{I} \rightarrow W^{k-1,p}(\mathbb{P}^1, \Lambda^{0,1}T^{*}\mathbb{P}^1\otimes u^{*}TX)
\]
is surjective.  Since this operator is homotopic to the operator $(\xi, Y) \mapsto D_u\xi$, one can easily show that $\mathcal{M}([C], \mathbb{P}^1; \mathcal{I})$ has dimension $8$, with tangent space at $(u,I_t)$ given by the kernel of~\eqref{eq: extDU}.  The $6$ real dimensional group $PSL(2,\mathbb{C})$ acts on $\mathcal{M}([C], \mathbb{P}^1; \mathcal{I})$ by reparametrization.  Taking the quotient, we obtain a smooth manifold 
\[
 \widetilde{\mathcal{M}}([C], \mathbb{P}^1; \mathcal{I}) :=  \mathcal{M}([C], \mathbb{P}^1; \mathcal{I})/PSL(2,\mathbb{C})
\]
of real dimension $2$ consisting of {\em unparametrized} holomorphic maps.

Combining the assumption with Lemma~\ref{lem: GromConv}, $\widetilde{\mathcal{M}}([C], \mathbb{P}^1; \mathcal{I})$ contains a sequence of disjoint $I_0$ holomorphic curves converging to the $I_0$ holomorphic curve $u$.  Thus, there is a non-zero smooth section $w$ of $u^{*}TX\rightarrow \mathbb{P}^1$  satisfying 
\[
\hat{D}_{(u, I_0)}(w, 0) = D_{(u,I_0)}w =0,
\]
and giving rise to a non-trivial deformation of $u$.  On the other hand, by \cite[Remark 3.2.6]{McDSal}, the operator $D_u = \hat{D}_{(u,I)}$ is complex linear since $I=I_0$ is integrable. Hence $(Iw,0) \ne (w,0)$ is also in the kernel of $\hat{D}_{(u,I)}$.  We claim that $Iw$ gives rise to a non-trivial deformation.  To see this observe that if $Iw$ gives rise to a trivial deformation, then $Iw= du(V_{\tau})$ where $V_{\tau}$ is the vector field induced by a $1$-parameter subgroup of $PSL(2,\mathbb{C})$.  But by the $I$-holomorphic curve equation,
\[
w= -I(Iw) = -Idu(V_{\tau}) = -du(j_{\mathbb{P}^1} V_{\tau})
\]
and hence $w$ gives rise to a trivial deformation, a contradiction.

Fix a point $p \in C_{reg}$, where $w(p)\ne 0$. Let $z$ be a local coordinate on $C_{reg}$ near $p$ identifying a neighborhood of $p$ with $B_1\subset \mathbb{C}$. For $\epsilon \ll 1$ consider the map
\begin{equation}\label{eq: defEmbedding}
B_1 \times \{ (s,t) \in \mathbb{R}^2: |s|< \epsilon, |t|<\epsilon\} \mapsto u_{s,t}(z),
\end{equation}
where $u_{s,t}$ is the deformation of $u$ generated by $sw+tIw$. Let $C_{s,t}=u_{s,t}(\mathbb{P}^1)$.  Since $p$ is a regular point of $C$ and $w(p)\ne 0$,~\eqref{eq: defEmbedding} is an immersion in a neighborhood of the origin.  Furthermore, since the deformation is non-trivial and $[C].[C]=0$, $C_{s,t}$ is disjoint from $C_{s',t'}$ for $(s,t)\ne (s',t')$.  The map~\eqref{eq: defEmbedding} is therefore an embedding. In particular, deformations of $C$ sweep out a neighborhood of $p$.

On the other hand, since $C\subset X_2$, for any $\epsilon >0$ there are smooth elliptic curves homologous to $[C]$ and intersecting $B_{\epsilon}(p)$.  It follows that there is an $I=I_0$ holomorphic rational curve $\tilde{C}$ homologous to $C$ such that $\tilde{C}$ intersects $X_1$ non-trivially.  Therefore, we can choose a smooth elliptic curve $\hat{C} \subset X_1$ homologous to $C$ such that
\[
0< \tilde{C}.\hat{C} = [C]^2 =0
\]
a contradiction.
\end{proof}

\begin{rk}
The proof of Proposition~\ref{prop: cuspAcc} could be used to prove Proposition~\ref{prop: nodalAcc} as well.  The main advantage of the argument given to prove Proposition~\ref{prop: nodalAcc} is that the extraneous family of complex structures is constructed explicitly using the twistor space construction
\end{rk}

Finally, we only need to rule out the accumulation of singular rational curves with one component at singular curves with several components and vice versa.  Note if $C$ is an irreducible (nodal or cuspidal) rational curve, then it is a simple consequence of Mayer-Vietoris that, for some $\epsilon>0$, $H_{2}(B(C, \epsilon), \mathbb{Z}) = \mathbb{Z}[C]$.  In particular, singular curves $\sum_{k}n_kC^{(k)}$ with more than one component cannot accumulate at singular curves with only one component.  The converse will be a corollary of the following proposition

\begin{prop}\label{prop: topFib}
The set $X_1$ consisting of points lying on smooth elliptic curves is open, dense and path connected in $X$. In particular, $X=X_1\cup X_2$. 
\end{prop}
\begin{proof}

Choose a smooth elliptic curve $C$ passing through a point $p\in X_1$.  Let $q\in X$ and choose $R>0$ so that $q\in B_{R}(p)$.  We can assume that $q \notin X_2$, for otherwise we are finished.  Consider $\overline{B_{2R+e}}(p)$, where $e$ is the constant from Proposition~\ref{prop: cpctSet} for $K=B_{2R}(p)$ and $A= [\omega].[C]$ .  We claim that, by Lemma~\ref{lem: multAcc} and Propositions~\ref{prop: nodalAcc} and~\ref{prop: cuspAcc}, the set $X_2\cap \overline{B_{2R}}(p)$ is a finite union of sets $C_i\cap  \overline{B_{2R}}(p)$ where $C_{i}\subset \overline{B_{2R+e}}(p)$ are singular holomorphic curves.  Suppose that this is not the case.  By Proposition~\ref{prop: cpctSet} any connected holomorphic curve  intersecting $\overline{B_{2R}(p)}$ is contained in $\overline{B_{2R+e}}(p)$.  Thus, for the sake of contradiction, we can assume
\begin{equation}\label{eq: X2union}
X_2\cap \overline{B_{2R}}(p) = \bigcup_{\alpha \in \mathcal{A}} C_\alpha\cap  \overline{B_{2R}}(p),
\end{equation}
where the $C_{\alpha}$ are singular holomorphic curves in $\overline{B_{2R+e}}(p)$ and $\mathcal{A}$ is an infinite index set.  There is an infinite subset $\mathcal{A}'\subset \mathcal{A}$ such that all the curves $C_{\alpha'}, \alpha'\in \mathcal{A}'$ are either all reduced, irreducible rational curves or all curves with $m\geq 2$ irreducible components.  Now, since $\overline{B_{2R+e}}(p)$ is compact, the Hausdorff distance is compact, and so there is a sequence of curves $C_{\alpha'_i}, i\in \mathbb{N}$ converging in the Hausdorff distance to $C_{\alpha'_{\infty}}$ for some $\alpha'_{\infty}\in \mathcal{A}'$. If $\mathcal{A}'$ consists only of irreducible rational curves, this contradicts  Propositions~\ref{prop: nodalAcc} and~\ref{prop: cuspAcc}, while if $\mathcal{A}'$ consists of reducible curves then this contradicts Lemma~\ref{lem: multAcc}.  Therefore, the index set $\mathcal{A}$ in~\eqref{eq: X2union} is finite.  Since each $C_{\alpha}$ has only finitely many components by Lemma~\ref{lem: multBnd} each with Hausdorff dimension $2$, $X_{2}\cap B_{2R}(p)$ has Hausdorff dimension $2 <3 =4-1$.  Therefore, the complement $B_{2R}(p)\setminus X_2$ is path connected (see, for example \cite{Sem}).  Hence we can find a smooth curve $\gamma(t)$ such that $\gamma(0)=p, \gamma(1)=q$ and $\gamma(t)\notin X_2$ for any $t\in[0,1]$.  The set $I:= \{ t\in[0,1]: \gamma(t) \in X_1\}$ is non-empty and open by Lemma~\ref{lem: openFib}.  Since $\gamma(t)\notin X_2$ for any $t\in[0,1]$, it follows that $A$ is closed and hence $q\in X_1$ as desired.

\end{proof}

The following lemma is straightforward, but we state it for completeness

\begin{lem}\label{lem: eulerChar}
	Let $Y$ be a del Pezzo surface or a rational elliptic surface and $D\in |-K_{Y}|$ a smooth divisor with $D^2=d$.  Then $X= Y\setminus D$ has $\chi(X)=9-d$.
\end{lem}
\begin{proof}
	This follows immediately from the fact that, topologically, $Y$ is obtained by blowing-up $\mathbb{P}^2$ at $9-d$ points.  Thus $\chi(Y)= 12$.  Therefore $\chi(X)= 12-d$.
\end{proof}

Finally, we prove the following lemma, which in combination with the previous results, establishes Theorem~\ref{thm: oneToMany}.

\begin{lem}\label{lem: fibration}
There is a complex manifold $B$ of complex dimension $1$ such that $(X,g,I, \omega_I)$ admits an elliptic fibration $\pi: X\rightarrow B$, the number of singular fibers bounded by $\chi(X) < +\infty$.  Furthermore, this fibration is minimal in the sense that no fiber contains a rational curve with self-intersection $(-1)$. 
\end{lem}
\begin{proof}
First we show that $X$ contains only finitely many singular fibers. Notice that one can get compute the Euler characteristic of $X$ from the Mayer-Vietoris sequence. Since $X$ is a torus fibration we have
\[
+\infty > \chi(X)  = \sum_{C: \text{ singular fibers}} \chi(C).
\]
On the other hand, by Proposition~\ref{prop: singClass} the singular fibers are classified by Kodaira's list and each singular fiber $C$ has $\chi(C) \geq 1$ (see, for example \cite{Mir}).  Therefore there can only be a finite number of singular fibers.  In particular, this rules out the accumulation of singular rational curves with only one component at a singular curves with several components.

Define a fibration $\pi: X \rightarrow B$ by sending $x \rightarrow [x]$ where we say that $x\sim y$ if $x,y$ lie on the same connected holomorphic curve homologous to $[L]$.  Note that by Proposition~\ref{prop: topFib} this equivalence relation is well defined on all of $X$.  Let $B_1 = \pi(X_1)$ and recall that by a result of Hitchin \cite{Hit}, $B_1$ has a natural complex structure making $\pi: X_1 \rightarrow B_1$ a holomorphic fibration with fibers smooth genus $1$ curves.

For a torus fiber $C$, choose a point $p \in C$.  Then the normal exponential map $\nu : T_{p}C^{\perp} \rightarrow X$ defines a local section of $\pi$ which is smooth with respect to the smooth structure on $B_1$.  We extend this structure to all of $B$ in the following way.  If $C$ is a singular curve in $X_2$, choose a point $p\in C_{reg}$ lying in a component with multiplicity $1$; this is possible by Proposition~\ref{prop: singClass} and Remark~\ref{rk: redCompExist}.  Since the singular fibers of $\pi$ are isolated, there is a small ball $B(p, \epsilon)$ such that $B(p, \epsilon)\setminus C$ consists only of points lying on smooth torus fibers.  Furthermore, since $p$ lies on a component with multiplicity $1$, the normal exponential map from $p$ defines a local section of $\pi$ intersecting each smooth fiber in one point and hence induces local coordinates in a neighborhood of $\pi(p)\in B$.  Since the disk has a unique smooth structure, the resulting smooth structure on $B$ is well-defined and independent of any choices.

It only remains to prove that the holomorphic structure on $B_1$ extends to all of $B$.  Choose a point $b\in B\setminus B_1$ and let $D \ni b$ be a disk with local coordinates $(x_1, x_2)$ centered at $b$ and such that $D^{*} := D\setminus\{b\} \subset B_1$.  By Hitchin's result \cite{Hit}, $D^{*}$ has a complex structure.  By a result from complex analysis (see, for example \cite[Corollary 1.2.7a]{Thur}) $D^{*}$ is biholomorphic to either the punctured disk $\Delta^{*} = \{ z \in \mathbb{C} : 0< |z|<1\}$, $\mathbb{C}^{*}$, or an annulus $\{ z \in \mathbb{C} : 1 < |z| < R\}$ for some $R>1$. We claim that in fact $D^{*}$ is biholomorphic to $\Delta^{*}$.  Take $p \in C_{reg}$ a smooth point in a component of multiplicity $1$; note that such a component always exists; see Proposition~\ref{prop: singClass} and Remark~\ref{rk: redCompExist}.  We can find holomorphic coordinates $(z_1,z_2)$ on an open ball $B(p, \epsilon) \subset X$ so that $\{z_2=0\}=C\cap B(p,\epsilon) $ and $p=\{z_1=z_2=0\}$. We claim that for $N\in \mathbb{N}$ sufficiently large, the set $\{(0,z_2): |z_2|<N^{-1}\epsilon\}$ will intersect each fiber of $\pi$ exactly in one point.  Suppose that this is not the case.  Then, for all $i\in \mathbb{N}$ there exist smooth torus fibers $C^i \ne C$, Hausdorff converging to $C$ as $i\rightarrow \infty$ and having the following property: for each $i \in \mathbb{N}$ there exists points $p_1^i, p_2^i \in C^i \cap B(p,\epsilon)$, with $p_1^i \ne p_2^i$ such that $z_1(p_1^i)=z_1(p_2^i)=0$ and $|z_{2}(p_1^i)|<i^{-1}\epsilon, |z_2(p_2^i)|<i^{-1}\epsilon$.  Clearly $p_1^i,p_2^i \rightarrow p$ as $i\rightarrow \infty$. 

If there is a subsequence $i_{k}\rightarrow +\infty$ such that $C^{i_k}\cap B(p, \epsilon)$ has $2$ or more connected components then necessarily the component of $C$ containing $p$ has multiplicity at least $2$.  But this contradicts the fact that $p$ is in a component of $C$ with multiplicity $1$.  Thus, we can assume that for $i$ sufficiently large $C^{i}\cap B(p, \epsilon)$ has only one connected component.  Thanks to the connectedness and the fact that $z_1(p_j^i)=0$ for $j=1,2$, the mean value theorem gives the existence of a point $q^i \in C^i\cap B(p, \epsilon)$ such that $dz_1=0$ at $q^i$, but this contradicts the fact that $(z_1,z_2)$ are coordinates.

 Thus, after possibly shrinking $D$, the map $z_2 \mapsto \pi(0,z_2)$ gives a biholomorphic map from a punctured disk to $D^{*}$; in particular $D$ inherits a complex structure extending the one on $D^{*} \cong \Delta^{*}\subset \mathbb{C}$.  Thus, we have a holomorphic map $\pi: \pi^{-1}(D^{*}) \rightarrow D^{*} \cong \Delta^{*} \subset \mathbb{C}$.  By the Riemann extension theorem, $\pi$ extends to a holomorphic map $\pi: \pi^{-1}(D) \rightarrow \Delta$ and hence $\pi: X\rightarrow B$ is a holomorphic fibration.  By the adjunction formula, we conclude that the smooth fibers of $\pi$ are tori and no smooth rational curve with self intersection $(-1)$ can occur in any fiber (this also follows from Proposition~\ref{prop: singClass}).  Thus, $\pi:(X,I) \rightarrow B$ is a minimal holomorphic torus fibration.
\end{proof}

Finally, we can apply Theorem~\ref{thm: oneToMany}, in conjunction with Theorem~\ref{thm: main1Intro} to prove Theorem~\ref{thm: wtf}.  In order to apply Theorem~\ref{thm: oneToMany} we need to first check that the conditions apply.  We begin with the following result.

\begin{lem}\label{lem: prim}
Let $Y$ be a compact K\"ahler surface and $D\in|-K_{Y}|$ a smooth anti-canonical divisor.  Let $N$ be a tubular neighborhood of $D$ and $X=Y\setminus D$.  Assume that $[L]\in H_{2}(N,\mathbb{Z})$ is primitive, $H_{2}(Y,\mathbb{Z})$ is torsion-free and $H_{1}(Y,\mathbb{Z})=0$.  Then $[L] \in H_{2}(X,\mathbb{Z})$ is primitive.
\end{lem}
\begin{proof}
This is a direct consequence of the Mayer-Vietoris sequence.  Since $H_{1}(Y)=0$, Poincar\'e duality implies $H_3(Y)=0$ and by Mayer-Vietoris
\[
\begin{aligned}
0= H_{3}(Y) \rightarrow H_{2}(N\setminus D) &\overset{\alpha}{\longrightarrow}H_{2}(X) \oplus H_{2}(D) \overset{\beta}{\longrightarrow} H_{2}(Y)\\
[L]&\longrightarrow(\alpha([L]), 0).
\end{aligned}
\]
If $\alpha([L])$ is not primitive in $X$, then we can write $\alpha([L]) = m[L']$ for some primitive homology class $[L'] \in H_{2}(X)$ not in the image of $\alpha$.  But then $\beta([L'])$ is a non-zero class in $H_{2}(Y)$ with $m\beta([L'])=0$, contradicting the assumption the $H_2(Y)$ has no torsion.
\end{proof}

\begin{thm}\label{thm: sLagFib}
Let $Y$ be a del Pezzo surface or a rational elliptic surface and $D\in |-K_{Y}|$ a smooth anti-canonical divisor.  Then $X=Y\setminus D$ admits a special Lagrangian fibration $\pi:X\rightarrow \mathbb{R}^2$ with at most finitely many singular fibers.  Furthermore, after hyper-K\"ahler rotation with respect to the Tian-Yau metric, the fibration $\pi:X \rightarrow \mathbb{C}$ is holomorphic.
\end{thm}
\begin{proof}
Let $Y$ be a del Pezzo surface or a rational elliptic surface.  Then, topologically, $Y$ is obtained by blowing up $\mathbb{P}^2, \mathbb{P}^1\times\mathbb{P}^1$, or the second Hirzebruch surface $\mathbb{F}_2$. It follows that $H_1(Y,\mathbb{Z})=0$, $H_2(Y,\mathbb{Z})$ is torsion-free and $Y$ has finite Euler characteristic by Lemma~\ref{lem: eulerChar}.

Since the divisor $D$ is a flat torus, $D$ contains infinitely many smooth special Lagrangian circles.  By Theorem~\ref{thm: main1Intro} we obtain infinitely many disjoint, possibly embedded, special Lagrangian tori in $X$ with ${\rm Im}(\Omega)\big|_{L}=0$ (after possibly rotating $\Omega$) within a fixed homology class $[L]$ which is primitive in a tubular neighborhood of $D$.  Therefore, by Lemma~\ref{lem: prim}, $[L]$ is primitive in $H_{2}(X,\mathbb{Z})$ and since $[L]$ can be represented by disjoint embedded special Lagrangians, $[L]^2=0$.   

Now, $X$ equipped with its the Tian-Yau, or asymptotically cylindrical Calabi-Yau metric, satisfies the assumptions of Theorem~\ref{thm: oneToMany} and hence we obtain a special Lagrangian torus fibration $\pi: (X,g, J, \omega_J) \rightarrow B$ with finitely many singular fibers, each classified by Kodaira and having no multiple fibers.  We only need to prove that $B= \mathbb{R}^2$.  After hyper-K\"ahler rotating, we have a holomorphic fibration $\pi:(X,g,I, \omega_I) \rightarrow B$ and $B$ is a non-compact Riemann surface by Hitchin's result \cite{Hit}.  We need the following lemma

\begin{lem}\label{lem: torsionFun}
The manifold $X$ has torsion first homology group.
\end{lem}
\begin{proof}
	Consider the long exact sequence of relative homologies
	\begin{align*}
	H_2(X)\rightarrow H_2(Y)\rightarrow H_2(Y,X)\rightarrow H_1(X)\rightarrow H_1(Y)=0. 
	\end{align*} The last equality can be seen from the fact that any del Pezzo surface $Y$ is either a blow-up of $\mathbb{P}^2$ in points or $\mathbb{P}^1\times \mathbb{P}^1$.
	Under the duality $H_2(Y,X)\cong H^2(D)\cong \mathbb{Z}$, the second map is given by 
	\begin{align*}
	     H_2(Y)&\rightarrow H^2(D)  \\ 
	        [C]&\mapsto ([D]\mapsto [C].[D]).
	\end{align*} This is surjective after tensoring with $\mathbb{R}$ because $D$ is ample.

\end{proof}

Let $B_{s}$ be the image of the singular fibers, $X_{s} = \pi^{-1}(B_s)$ and recall that by Lemma~\ref{lem: fibration}, $B_{s}$ is a finite set.  Given any closed curve $\gamma$ in $B$, after possibly a small homotopy, we can assume that $\gamma$ avoids $B_{s}$.  By lifting $\gamma$ along the smooth fibration $\pi:X\setminus X_s\rightarrow B\setminus B_s$, we obtain a surjection $\pi_1(X)\rightarrow \pi_1(B)$.  On the other hand, it is a classical result (see, for example \cite[Theorem 44A]{AS}) that the fundamental group of a non-compact Riemann surfaces is a free group.  In particular, if $H_1(X)$ is torsion, then $B$ must be simply connected.  By the uniformization theorem $B$ is biholomorphic to either $\mathbb{C}$ or the unit disk.  If $B$ is biholomorphic to a disk, then pulling back the bounded holomorphic function $z$ along the holomorphic fibration $\pi$, we obtain (after taking real and imaginary parts) a bounded harmonic function on $(X,g)$.  But $(X,g)$ is complete and Ricci-flat and a well-known theorem of Yau says that no such function can exist \cite{Y1}. Thus $B$ is biholomorphic to $\mathbb{C}$.
\end{proof}

Finally, we prove that for Tian-Yau surfaces of Type I or II the special Lagrangian fibration in a neighborhood of $\infty$ is obtained from the Lagrangian mean curvature flow of the model fibration, in a very precise sense.  To fix notation, let $L_0$ be a Lagrangian in the model fibration and denote by $F_{t}(L_0)$ the solution of the Lagrangian mean curvature flow starting from $L_0$ at time $t \in [0, \infty)$.  Sections~\ref{sec: ACYL} and~\ref{sec: TY} show that in either the Type I or Type II cases we can fix a compact set $K_0\subset X$ such that $X\setminus K_0$ is diffeomorphic to the corresponding model geometry and such that if $L_0\subset X\setminus K_0$ is a Lagrangian obtained from a special Lagrangian in the model geometry, then Lagrangian mean curvature flow starting at $L_0$ converges smoothly and exponentially fast to a special Lagrangian which we denote by $F_{\infty}(L_0)$.

\begin{prop}\label{prop: MCFhomotope}
In the above setting, there are compact sets $K_0\subset K_1 \subset K_2\subset X$ with the following property; let $\tilde{L}\subset X\setminus K_1$ be a fiber of the special Lagrangian fibration of $(X,\omega_{TY})$.  Then there exists a unique Lagrangian $L\subset (X,\omega_{TY})$ contained in $X\setminus K_0$ which is a fiber of the model Lagrangian fibration such that 
\[
\tilde{L} = F_{\infty}(L).
\]
Furthermore, the mean curvature flow $F_t(\cdot), t\in[0,\infty]$ induces a continuous family of continuous maps $F_t: X\setminus K_1 \rightarrow X\setminus K_0$ such that $F_{\infty}$ is injective and $F_{\infty}(X\setminus K_1) \supset X\setminus K_2$.
\end{prop}

The proposition will be the result of the following two lemmas.

\begin{lem} \label{lem: HF}
   	Let $L_1,L_2$ be two disjoint model Lagrangian submanifolds contained in $X\setminus K_0$ and denote by $\tilde{L}_i = F_{\infty}(L_i)$ the limits of the LMCF for $i=1,2$.  Then $\tilde{L}_1$ and $\tilde{L}_2$ are disjoint.
   \end{lem}
   \begin{proof}
   By our choice of $K_0$ above, the Lagrangian mean curvature flow starting from any model Lagrangian contained in $X\setminus K_0$ converges smoothly and exponentially fast to a special Lagrangian torus $\tilde{L}$.  Let $L_1, L_2$ be two such model Lagrangians in $X\setminus K_0$ and let $\tilde{L}_1, \tilde{L}_2$ be the corresponding special Lagrangians obtained as limits of the LMCF.
   
   Hyper-K\"ahler rotate so that $\tilde{L}_1, \tilde{L}_2$ are holomorphic.  Since $0 = [\tilde{L}_1].[ \tilde{L}_2$] we see that $\tilde{L}_1, \tilde{L}_2$ are either disjoint or equal.  We only need to rule out the case $\tilde{L}_1=\tilde{L}_2$.  We do this by using a Floer homology-theoretic argument.  Recall that LMCF preserves the Hamiltonian isotopy class \cite{Smo}. Since $L_1,L_2,\tilde{L}_1,\tilde{L}_2$ are all special Lagrangians (though with respect to different holomorphic volume forms), they are all spin and have Maslov index $0$.  Since $X$ has complex dimension $2$, a standard index calculation shows that the moduli space of holomorphic disks with boundary on any of $L_1,L_2,\tilde{L}_1,\tilde{L}_2$ has virtual dimension $-1$.  Thus, by a generic small perturbation of almost complex structures, we can assume that they don't bound any pseudo-holomorphic disks and hence are all unobstructed Lagrangians; see \cite{SolVer} for an even stronger result. In particular, the Floer homology between any pair of $L_1,L_2,\tilde{L}_1,\tilde{L}_2$ is well-defined; see for instance \cite{Aur3}. If $\tilde{L}_1=\tilde{L}_2$ coincide, then the standard argument of Floer \cite{Fl} yields
   	\begin{align*}
   	H^*(\tilde{L}_1)\cong HF^*(\tilde{L}_1,\tilde{L}_1)=  HF(\tilde{L}_1,\tilde{L}_2)\cong HF(L_1,L_2)=0.
   	\end{align*}
	But since $\tilde{L}_1$ is a torus, this is absurd. Therefore, $\tilde{L}_1,\tilde{L}_2$ are disjoint and the result follows.
   \end{proof}	

Next we have

\begin{lem}\label{lem: mcfContinuous}
There exist compact sets $K_1\subset K_2$ with $K_0\subset K_1\subset X$ and having the following property; suppose $\tilde{L}_i = F_{\infty}(L_i)$ are special Lagrangians in $X\setminus K_2$ converging in the Hausdorff topology to $\tilde{L}_{\infty}\subset X\setminus K_2$.  Then the sequence $\{L_{i}\}_{ i\in \mathbb{N}}$ of model Lagrangian submanifolds is contained in $X\setminus K_1$ and converges in the Hausdorff topology to a model Lagrangian $L_{\infty} \subset X\setminus K_1$ with $F_{\infty}(L_\infty)= \tilde{L}_{\infty}$. 
\end{lem}
\begin{proof}
By Theorem~\ref{thm: TYLMCFconv} (or Theorem~\ref{Liconverge} in the Type II case) the LMCF starting from any model Lagrangian $L$ in $X\setminus K_0$ converges to a special Lagrangian $\tilde{L}$ which is contained in a ball of radius $\epsilon$ around $L$, where $\epsilon=\epsilon(K_0)$ depends only on $K_0$.  Thus, we can choose $K_0\subset K_1\subset K_2$ such that
\begin{itemize}
\item[(i)] the LMCF starting from any model Lagrangian in $X\setminus K_1$ converges to a special Lagrangian in $X\setminus K_0$,
\item[(ii)] if $L$ is a model Lagrangian such that the LMCF starting from $L$ converges to $\tilde{L} =F_{\infty}(L) \subset X\setminus K_2$, then $L \subset X\setminus K_1$.
\end{itemize}
In fact, we may as well just take
\[
K_i= \overline{B(K_0, 100^i \epsilon(K_0))}
\]
for $i=1,2$.

Let $\tilde{L}_i\subset X\backslash K_2$ be as in the statement of the lemma. Since ${\tilde{L}_i\subset X\setminus K_2}$ for all $i$ (including $i=\infty$) and the sequence $\{\tilde{L}_i\}$ converges in the Hausdorff topology, we get that the $L_i$ are contained in a compact subset of \linebreak $X\setminus K_1$. From the explicit description of the model fibration we can pass to a subsequence (not relabelled)  converging in the Hausdorff topology (and even smoothly) to a limit model Lagrangian $L_{\infty}\subset X\setminus K_1$.  It suffices to show that $F_{\infty}(L_{\infty})= \tilde{L}_{\infty}$; this follows from the continuous dependence of the LMCF on initial conditions, together with the exponential decay of the mean curvature established in Theorem~\ref{thm: TYLMCFconv}.

First, since $L_i \subset X\setminus K_0$ for all $i$ (including $i=\infty$), the proof of Theorem~\ref{thm: TYLMCFconv} (see~\eqref{eq: TYexpDecay}) shows that (after possibly enlarging $K_0$) we can assume that the mean curvature along the LMCF satisfies
\[
|H(t)|^2 \leq e^{-ct}
\]
on $[1,\infty]$ for a uniform constant $c>0$.  Thus, for any $\epsilon >0$ we can choose $T_{\epsilon}$ large so that
\[
\int_{T_{\epsilon}}^{\infty} |H(t)| dt < \epsilon.
\]
Since the $L_i$ converge smoothly to $L_{\infty}$, the continuous dependence of the mean curvature flow on initial data shows that we can choose $N$ large so that, if $i\geq N$, then
\[
d(F_{t}(L_i), F_{t}(L_{\infty})) < \epsilon
\]
for all $t\in [0, T_{\epsilon}]$.  Combining these estimates we see that
\[
d(F_{t}(L_i), F_{t}(L_{\infty})) < 3\epsilon
\]
 for all $t\in [0,\infty]$ provided $i \geq N$.  Since the flows $F_{t}(L_i)$ converge smoothly to $\tilde{L}_i$ we get $d(\tilde{L}_i, F_{\infty}(L_{\infty})) < 3\epsilon$. Now since $\epsilon$ was arbitrary and $\tilde{L}_i$ converge to $\tilde{L}_{\infty}$, we conclude that $F_{\infty}(L_{\infty}) = \tilde{L}_{\infty}$ as desired.  Furthermore, since $\tilde{L}_{\infty}$ is a smooth fiber of a smooth torus fibration, the convergence $\tilde{L}_i \rightarrow \tilde{L}_{\infty}$ is smooth.

\end{proof}

\begin{proof}[Proof of Proposition~\ref{prop: MCFhomotope}]
Let $K_0\subset K_1 \subset K_2$ be as in Lemmas~\ref{lem: HF} and~\ref{lem: mcfContinuous}. For $t\in[0,\infty]$,  let $F_t : X\setminus K_2 \rightarrow X\setminus K_1$ be the map sending a point $x$ to its time $t$ flow under the LMCF.  By (the proof of) Lemma~\ref{lem: mcfContinuous}, $F_t$ is continuous for all $t\in[0,\infty]$.  By Lemma~\ref{lem: HF}, $F_{\infty}$ is injective and by Lemma~\ref{lem: mcfContinuous}, $F_{\infty}(X\setminus K_1)\cap X\setminus K_2$ is relatively closed in $X\setminus K_2$.  By invariance of domain, $F_{\infty}(X\setminus K_1)\cap X\setminus K_2$ is relatively open in $X\backslash K_2$.  Since $K_2$ is compact and $X$ has only one end, $X\backslash K_2$ is connected, and so $F_{\infty}(X\setminus K_1) \cap X\backslash K_2= X\backslash K_2$ as claimed.
\end{proof}

\section{Applications to Mirror Symmetry}\label{sec: mirror}

In this section we apply the results from Section~\ref{sec: sLagFib}, together with the classification of compact complex surfaces, to prove Corollaries~\ref{cor: introAurConj1} and~\ref{cor: introAurConj2} and Theorem~\ref{thm: main3Intro}.  This will be obtained by compactifying the elliptic fibrations obtained by hyper-K\"ahler rotating our special Lagrangian fibrations.  We begin with the following lemma, which says that there is a section in a neighborhood of $\infty$.  The reader may wish to compare with \cite[Proposition 5.3.1]{CosDol}.

\begin{lem}\label{lem: section}
Let $Y$ be a del Pezzo surface or a rational elliptic surface and $D\in |-K_{Y}|$ a smooth anti-canonical divisor.  Let $\pi:(X,g, J)\rightarrow \mathbb{R}^2$ be the special Lagrangian fibration whose existence in guaranteed by Theorem~\ref{thm: sLagFib}.  Then, after hyper-K\"ahler rotation, the genus one fibration $\pi: (X,g, I)\rightarrow \mathbb{C}$ admits a local holomorphic section in a neighborhood of $\infty$.
\end{lem}

\begin{proof}
By Theorem~\ref{thm: sLagFib}, after hyper-K\"ahler rotation, we have an elliptic fibration $\pi: (X,g, I) \rightarrow \mathbb{C}$ with no singular fibers in a neighborhood of infinity.  Let $\Delta^{*}$ be a punctured disk neighborhood of $\infty$ and let $X^{*} = \pi^{-1}(\Delta^{*})$.  Since the fibers of $\pi$ are smooth elliptic curves, the map $\pi$ is flat \cite[p. 158]{Fis}.  In particular, the direct image sheaves $R^{i}\pi_{*}\mathcal{O}_{X}$ are locally free.  The fiber of $R^{i}\pi_{*}\mathcal{O}_{X}$ over $b\in \Delta^{*}$ is, by definition, $H^{i}(\pi^{-1}(b), \mathcal{O}_{\pi^{-1}(b)})$.  When $i=1$, Serre duality implies $H^{1}(\pi^{-1}(b), \mathcal{O}_{\pi^{-1}(b)}) = H^{0}(\pi^{-1}(b), \mathcal{O}_{\pi^{-1}(b)})= \mathbb{C}$ and so $R^{1}\pi_{*}\mathcal{O}_{X}$ is a line bundle.  Thanks to the fact that  $\Delta^{*}$ is Stein, Cartan's Theorems A and B imply that $H^{1}(\Delta^{*}, R^{p}\pi_{*}\mathcal{O}_{X}) =0$.  By a theorem of Grauert \cite{Gr} and R\"ohrl \cite{Ro}, $R^{1}\pi_{*}\mathcal{O}_{X}$ is the trivial line bundle.

Let $X^{\#}$ denote the sheaf of holomorphic sections of $\pi:X^{*}\rightarrow \Delta^{*}$.   Since $\pi:X^{*}\rightarrow \Delta^{*}$ is a smooth fibration without multiple fibers, there is an exact sequence of commutative groups (see, for example, \cite[Chapter V, Section 9]{BPV})
\[
0\rightarrow R^{1}\pi_{*}\mathbb{Z} \rightarrow R^{1}\pi_{*}\mathcal{O}_{X^{*}} \rightarrow X^{\#} \rightarrow 0.
\]
Taking the long exact sequence in cohomology yields
\[
0 \rightarrow H^{0}(R^{1}\pi_{*}\mathbb{Z}) \rightarrow H^{0}(R^{1}\pi_{*}\mathcal{O}_{X^*}) \rightarrow H^{0}( X^{\#}) \rightarrow H^{1}(R^{1}\pi_{*}\mathbb{Z})\rightarrow 0.
\]
Since $H^{0}(\Delta^{*}, R^{1}\pi_{*}\mathcal{O}_{X^*})$ is the sheaf of global sections of a trivial bundle over $\Delta^{*}$, it is uncountable.  On the other hand, $H^{i}(\Delta^{*}, R^{1}\pi_{*}\mathbb{Z}), i=0,1$ is a lattice and hence countable.  Therefore $H^{0}(\Delta^{*}, X^{\#})$ is infinite dimensional and hence we obtain a section.
\end{proof}

We now have a elliptic fibration over a punctured disk with a section.  In order to extend this elliptic fibration over $0$, we need to control the monodromy of the fibration.  We have

\begin{lem}\label{lem: monodromy}
Let $Y$ be a del Pezzo surface or a rational elliptic surface and $D\in |-K_{Y}|$ a smooth anti-canonical divisor with $D^2=d$.  Let $\pi:(X,g, J)\rightarrow \mathbb{R}^2$ be the special Lagrangian fibration whose existence in guaranteed by Theorem~\ref{thm: sLagFib}.  Then, around $\infty$ the torus fibration has monodromy
\[
m_{\infty, d} := \begin{pmatrix}
1 & d
\\ 0&1
\end{pmatrix}.
\]
\end{lem}
\begin{proof}
Let $\gamma$ be a simple closed loop in the base of the special Lagrangian fibration $\pi:X\rightarrow \mathbb{R}^2$ circling $\infty$ with positive orientation.  By Proposition~\ref{prop: cpctSet},  we can choose $\gamma$ so that the torus bundle $\pi^{-1}(\gamma)$ is contained in $X\setminus K_2$, where $K_2$ is the compact set from Proposition~\ref{prop: MCFhomotope}.  By Proposition~\ref{prop: MCFhomotope}, we can find a loop $\hat{\gamma}$ in the base of the model fibration such that the model torus fibration over $\hat{\gamma}$ is carried by LMCF to $\pi^{-1}(\gamma)$.  Thus $\pi^{-1}(\gamma)$ has the same homotopy type as the model torus fibraton over $\hat{\gamma}$ and hence has the same monodromy.  But the monodromy of the model fibration in the Calabi model, or asymptotically cylindrical model, is $m_{\infty,d}$ \cite{Sco}.
\end{proof}

We now obtain the following corollary, which follows from Kodaira's classification of singularities of elliptic fibrations \cite{Kod2, BPV} and the theory of stable reduction.
\begin{cor}\label{lem: compact}
There is a compact complex surface $\check{Y}$ equipped with a relatively minimal elliptic fibration $\check{\pi}:\check{Y}\rightarrow \mathbb{P}^1$ without multiple fibers such that $\check{Y}\setminus \check{\pi}^{-1}(\infty) \cong (X, I)$.  
\end{cor}
\begin{proof}
As usual, $(X,I)$ denotes the hyper-K\"ahler rotation so we have a fibration $\pi:(X, I) \rightarrow \mathbb{C}$. Identify a neighborhood of $\infty$ with the punctured disk $\Delta^{*}\subset \mathbb{C}$; let $\pi^*: X^{*} \rightarrow \Delta^{*}$ be the induced elliptic fibration.  By Lemma~\ref{lem: section}, $X^{*}\rightarrow \Delta^{*}$ has a section,  and hence we get a map $ f^*: \Delta^{*} \rightarrow \mathcal{M}_{1,1}$, the moduli space of elliptic curves with a marked point. By Lemma~\ref{lem: monodromy} the fibration $X^{*}\rightarrow \Delta^{*}$ has monodromy $m_{\infty, d}$ and so, by \cite[Proposition 5.9]{Hain} (and its proof), $f$ extends to a holomorphic map $f: \Delta \rightarrow \overline{\mathcal{M}}_{1,1}$.  Thus, we can identify $X^{*}\rightarrow \Delta^*$ with the universal family away from the central fiber. We can therefore fill in the fiber over $0\in \Delta$ and obtain a holomorphic  family
\[
f:X\rightarrow \Delta
\]
extending $X^{*}$ and having reduced central fiber.  By taking a minimal resolution and blowing down any $(-1)$ curves contained in the fiber we obtain a relatively minimal family of elliptic curves $\bar{\pi}: W\rightarrow \Delta$ agreeing with the fibration over $\Delta^*$.  Since $W$ is isomorphic to $X$ away from the fiber over $0$, the fibration has monodromy corresponding to a fiber of type $I_{d}$.  By Kodaira's classification \cite{Kod2, BPV}, this implies that the central fiber is of type $I_d$. 

Now, since $W\setminus \bar{\pi}^{-1}(0)$ is isomorphic to $X^{*}$ we can glue $W$ to $(X, I)$ along $X^*$ to obtain a compact complex surface $\check{Y}$ with a relatively minimal fibration $\check{\pi}:\check{Y}\rightarrow \mathbb{P}^1$.

\end{proof}

We are now in a position to prove

\begin{thm}\label{thm: RES}
Let $Y_{d}$ be a del Pezzo surface or rational elliptic surface and $D\in |-K_{Y_{d}}|$ a smooth divisor with $D^2=d$.  Let $X_d = Y_d\setminus D$ and equip $X_d$ with the Tian-Yau metric $g_{TY}$ and let $\pi:(X_d, g_{TY}, J)\rightarrow \mathbb{R}^2$ be the special Lagrangian torus fibration of Theorem~\ref{thm: sLagFib}.  Then after hyper-K\"ahler rotating to a complex structure $I$ so that $\pi:(X_d, g_{TY}, I)\rightarrow \mathbb{C}$ is a holomorphic elliptic fibration the following holds:  There is a rational elliptic surface $\check{\pi}: \check{Y}\rightarrow \mathbb{P}^1$ with a singular fiber of Kodaira type $I_{d}$ so that so that $(X_{d}, I)$ is biholomorphic to $\check{Y}\setminus I_{d}$. 
\end{thm}

\begin{proof}
Let $\check{\pi}: \check{Y}\rightarrow \mathbb{P}^1$ be the surface constructed in Lemma~\ref{lem: compact}. Then $\check{Y}$ admits a genus $1$ fibration with an $I_d$ fiber over $\infty \in \mathbb{P}^1$.  It follows from Lemma~\ref{lem: eulerChar} that $\chi(\check{Y})=12$.  From the Mayer-Vietoris sequence and Lemma~\ref{lem: torsionFun}, we have $b_1(\check{Y})=0$.  Since $\check{Y}\rightarrow \mathbb{P}^1$ is a genus $1$ fibration without multiple fibers, the canonical bundle formula (see, for example \cite[Chapter 7, Theorem 15]{Fried} or \cite[Chapter V, Theorem 12.1]{BPV}) gives
\begin{equation}\label{eq: classCanonForm}
K_{\check{Y}} = \pi^{*}(K_{\mathbb{P}}^1 \otimes \mathcal{O}_{\mathbb{P}^1}(k)) 
\end{equation}
for some $k\geq 0$.  Thus $c_1(\check{Y})^2=0$.  Furthermore, applying \cite[Chapter 7, Corollary 17]{Fried} we conclude that, since $\check{\pi}:\check{Y}\rightarrow \mathbb{P}^1$ has no multiple singular fibers we must have $k>0$ in~\eqref{eq: classCanonForm}. We can now appeal to the classification of compact complex surfaces.

To begin with, assume $\check{Y}$ is minimal.  Since $c_1(\check{Y})^2=0, b_1(\check{Y})=0$, by the Enriques-Kodaira classification (see, for example \cite[Chapter VI, Table 10]{BPV}) $\check{Y}$ must be an Enriques surface, a $K3$ surface or a minimal properly elliptic surface.  Since $\chi(\check{Y})=12$, $\check{Y}$ is not a $K3$ surface.  If $\check{Y}$ is an Enriques surface, then \cite[Chapter VIII, Lemma 17.1]{BPV} gives that $\pi: \check{Y}\rightarrow \mathbb{P}^1$ has two multiple fibers.  But by construction the fibration $\check{\pi}$ has no multiple fibers.  Thus $\check{Y}$ is not an Enriques surface.  It only remains to rule out the possibility that $\check{Y}$ is a minimal properly elliptic surface.  We apply Noether's formula in combination with $K_{\check{Y}}^2=0, \chi(\check{Y})=12$ to obtain
\[
\chi(\mathcal{O}_{\check{Y}}) = \frac{1}{12}(K_{\check{Y}}^2 + \chi(\check{Y})) = 1.
\]
By definition, a properly elliptic surface has Kodaira dimension $1$, which implies that in equation~\eqref{eq: classCanonForm}, we must have $k\geq 3$.  In particular, we have $h^{0}(\check{Y},K_{\check{Y}}) >0$.  In combination with Serre duality and $b_1(\check{Y})=0$, we obtain
\[
1= \chi(\mathcal{O}_{\check{Y}}) = h^0(\check{Y}, \mathcal{O}_{\check{Y}})-h^1(\check{Y}, \mathcal{O}_{\check{Y}})+h^2(\check{Y}, \mathcal{O}_{\check{Y}})= 1-0+h^0(\check{Y}, K_{\check{Y}}) >1, 
\]
a contradiction.

It follows that $\check{Y}$ is not minimal.  Let $C$ be a rational curve in $Y$ with $C^2=-1$.  Since the genus $1$ fibration is relatively minimal, $C$ must intersect the generic fiber of $\check{\pi}:\check{Y}\rightarrow \mathbb{P}^1$ positively; in particular, $C$ is a multi-section of the fibration.  Let $F$ be a generic fiber of $\pi$.  Then we have $(C+F)^2=-1 +2C.F + F^2 = 2C.F-1 >0$.  Thus by \cite[Chapter IV, Theorem 5.2]{BPV}, $\check{Y}$ is projective.  By the canonical bundle formula~\eqref{eq: classCanonForm} (and the remarks following it) we have
\[
K_{\check{Y}} = \pi^{*}(\mathcal{O}_{\mathbb{P}^1}(k')) 
\]
for some $k' \geq -1$.  If $k'\geq 0$, then $\check{Y}$ has ${\rm Kod}(\check{Y}) \geq 0$ and so $K_{\check{Y}}$ is effective.  But by the adjunction formula $K_{\check{Y}}.C=-1$, a contradiction.  Thus, we have $k'=-1$ and hence $h^1(\check{Y},\mathcal{O}_{\check{Y}}) =0, h^0(\check{Y}, K_{\check{Y}}^2)=0$ and $C$ intersects the generic fiber in one point.  Therefore, by Castelnuovo's rationality criterion \cite[Chapter VI, Theorem 2.1]{BPV} we conclude that $\check{Y}$ is rational.  Thus, $\check{Y}$ is a rational elliptic surface and $C: \mathbb{P}^1 \rightarrow \check{Y}$ is a section.
\end{proof}

\begin{cor}\label{cor: SYZsection}
Let $Y_{d}$ be a del Pezzo surface or rational elliptic surface and $D\in |-K_{Y_{d}}|$ a smooth divisor with $D^2=d$.  Let $X_d = Y_d\setminus D$, equip $X_d$ with the Tian-Yau metric $g_{TY}$ and let $\pi:(X_d, g_{TY}, J)\rightarrow \mathbb{R}^2$ be the special Lagrangian torus fibration of Theorem~\ref{thm: sLagFib}.  Then $\pi$ admits a global section.
\end{cor}

Note that Corollary~\ref{cor: SYZsection}, together with Theorem~\ref{thm: sLagFib} establishes Theorem~\ref{thm: wtf}, modulo the statement that, near infinity, the special Lagrangians are $S^1$-bundles over special Lagrangians in the divisor at $\infty$.  But this statement is an immediate consequence of Proposition~\ref{prop: MCFhomotope} and the fact that the model special Lagrangians are $S^1$-bundles over special Lagrangians in the divisor at $\infty$.

The next result says that, at least in the special case of $\mathbb{P}^2$, we can identify the rational elliptic surface obtained by hyper-K\"ahler rotation.  Together with Theorem~\ref{thm: sLagFib}, this result establishes Corollary~\ref{cor: introAurConj1} and a conjecture of Auroux \cite[Conjecture 2.9]{Aur}.

\begin{prop}
Let $D \in |-K_{\mathbb{P}^2}|$ be a smooth cubic and let $\check{\pi}:\check{Y}\rightarrow \mathbb{P}^1$ be the rational elliptic surface obtained via Theorem~\ref{thm: RES}.  Then $\check{\pi}^{-1}(\infty)$ is a singular fiber of type $I_9$ and  $\check{\pi}: \check{Y}\setminus \pi^{-1}(\infty)\rightarrow \mathbb{C}$ has exactly three singular fibers of type $I_1$.
\end{prop}

\begin{proof}
Recall that $\check{\pi}: \check{Y}\setminus \pi^{-1}(\infty)$ is an elliptic fibration with no multiple fibers.  Since $\chi(\check{Y})=12$ and $\check{Y}$ has a singular fiber of type $I_{9}$ over $\infty \in \mathbb{P}^1$ with monodromy at $\infty$ given by
\[
m_{\infty} := \begin{pmatrix}
1 & 9
\\ 0&1
\end{pmatrix},
\]
monodromy considerations \cite[Chapter V, Table 6]{BPV} imply that $\check{Y}\setminus \check{\pi}^{-1}(\infty)$ must have more than one singular fiber.  Thus, there are only three possible configurations for the singular fibers in $\check{Y}\setminus \check{\pi}^{-1}(\infty)$; they are $\{ I_1, I_1, I_1\}$, $\{I_1, II\}$ and $\{I_1, I_2\}$.  If the configuration is $\{I_1, I_2\}$, then there are $a,b \in \mathbb{Z}$ with $\gcd(a,b)=1$ so that $m_{\infty}$ is conjugate in $SL(2, \mathbb{Z})$ to
\[
\begin{pmatrix}
1-ab & a^2
\\ -b^2&1+ab
\end{pmatrix}
\begin{pmatrix}
1 & 2
\\ 0&1
\end{pmatrix}.
\]
Since ${\rm Tr}(m_{\infty})=2$, we find $b=0$ and hence $a= \pm1$, which implies that $m_{\infty}$ is conjugate to
\[
\begin{pmatrix}
1 & 3
\\ 0&1
\end{pmatrix}
\]
which is absurd.  If instead the configuration is  $\{I_1, II\}$ then we conclude that $m_{\infty}$ is conjugate to
\[
\begin{pmatrix}
1-ab & a^2
\\ -b^2&1+ab
\end{pmatrix}
\begin{pmatrix}
0 & 1
\\ -1&1
\end{pmatrix}.
\]
Again, since ${\rm Tr}(m_{\infty})=2$, we obtain that $a,b$ solve $a^2+b^2-ab+1=0$, which has no real solutions.  The result follows. 
\end{proof}

In the same vein, we have the following lemma, which together with Theorem~\ref{thm: sLagFib} proves Corollary~\ref{cor: introAurConj2} and a conjecture of Auroux \cite[Conjecture 2.10]{Aur}.

\begin{lem}
Consider the moduli space $\mathcal{Y}$ consisting of $4$-tuples $(Y,D, [\omega], \Omega)$, where $Y$ is a rational elliptic surface, $D\in|-K_{Y}|$ is a smooth divisor, $[\omega]\in H^2(Y,\mathbb{R})$ a K\"ahler class, and $\Omega$ is a meromorphic $2$-form on $X=Y\setminus D$ with simple pole along $D$. For a generic $4$-tuple $(Y,D,[\omega],\Omega) \in \mathcal{Y}$, the special Lagrangian fibration of $X$ with respect to any asymptotically cylindrical metric in $[\omega]|_{Y\setminus D}$ produced by Theorem~\ref{thm: sLagFib} has 12 singular fibers, each of which is a nodal special Lagrangian sphere.
\end{lem}
\begin{proof}
Recall that a generic rational elliptic surface has 12 singular fibers of type $I_1$.  The idea is to show that for generic choice of data $(Y,D,[\omega], \Omega)$, the hyper-K\"ahler rotation along the special Lagrangian fibration produced in Theorem~\ref{thm: sLagFib} will produce a generic rational elliptic surface. Fix a rational elliptic surface $\pi: Y\rightarrow \mathbb{P}^1$, a smooth divisor $D \in |-K_{Y}|$ and a class $[{\gamma}] \in H_{1}({D}, \mathbb{Z})$ and let $M_{\gamma}$ be the model Lagrangian induced by $\gamma$. Let ${\Omega}$ be the holomorphic volume form on ${X}:= {Y}\setminus {D}$ with a simple pole along ${D}$ and normalized such that $\int_{M_\gamma}\Omega=1$. Let ${\omega}$ be the asymptotically cylindrical Tian-Yau metric such that $2\omega^2=\Omega\wedge \bar{\Omega}$.  By Theorem~\ref{thm: sLagFib}, there exists a special Lagrangian fibration ${\pi}: {X} \rightarrow \mathbb{R}^2$.  Denote by $\check{X}$ the complex surface obtained from hyper-K\"ahler rotation with the K\"ahler form and holomorphic volume form 
\begin{equation} \label{HK}
\begin{aligned} 
\check{\Omega} &= {\omega} -\sqrt{-1}{\rm Im}{\Omega},\\
\check{\omega }&= {\rm Re}{\Omega}.
\end{aligned}
\end{equation}
Note that these choices are compatible with our normalizations so that after this rotation the fibration ${\pi}:X \rightarrow \mathbb{R}^2$ becomes an elliptic fibration. By Theorem~\ref{thm: RES}, there is a rational elliptic surface $\check{Y}$ and a smooth divisor $\check{D}\in|-K_{\check{Y}}|$ so that $X = \check{Y}\setminus \check{D}$.  By direct computation in the cylindrical model together with the estimates~\eqref{formdecay},~\eqref{Jdecay},~\eqref{holvoldecay}, one can check that $\check{\Omega}$ has a simple pole along $\check{D}$. 
 
By the Torelli theorem for pairs of rational surfaces with smooth anti-canonical divisors \cite{Mc}, a deformation $(Y',D')$ of the pair $(Y,D)$ is determined by the cohomology class of the meromorphic $2$-form $\Omega'$ in
\[
H^2(Y'\setminus D',\mathbb{C})\cong H^2(Y\setminus D,\mathbb{C})
\]
up to $\mathbb{C}^*$ scaling, but we have fixed the scaling by the $\int_{M_\gamma}\Omega=1$ (see also \cite[Proposition 3.12]{Fr2} to go from markings to periods). Since all the rational elliptic surfaces are deformation equivalent and generic fibres are smooth, all such pairs are in the same deformation family. Assume that $Y$ is a generic rational elliptic surface. Then for a generic small deformation $Y'$ of $Y$, the cohomology class $[\omega]\in H^2(Y,\mathbb{R})$ is transported by the Gauss-Manin connection to a K\"ahler form $[\omega'] \in H^2(Y',\mathbb{R})$ \cite[Theorem 0.9]{Demailly} (since $H^{2,0}(Y)=0$). In particular, a small deformation of the K\"ahler class $[\omega']|_{X'} \in H^{2}(X', \mathbb{R})$ within the subspace $\mbox{Im}\big(H^2(Y',\mathbb{R})\rightarrow H^2(X',\mathbb{R})\big)$ is K\"ahler.  We now apply the discussion in the preceding paragraph to $Y'$, obtaining a hyper-K\"ahler rotation as in~\eqref{HK}.  Denote by $(\check{Y}', \check{D}')$ the pair of a rational elliptic surface and a smooth anti-canonical divisor so obtained.  Let $\check{X}'=\check{Y}'\setminus \check{D}'$. By the Torelli theorem \cite{Mc}, any small deformation of $(\check{Y},\check{D})$ can be achieved by a suitable choice of $(Y',D')$ and its K\"ahler class $[\omega']$. The theorem follows from the fact that the generic rational elliptic surface $\check{Y}'$ has $12$ singular fibres.

\end{proof}
\begin{rk}
More generally, the above lemma shows that the special Lagrangian fibrations constructed by Theorem~\ref{thm: sLagFib} on the complement of a smooth divisor in a rational elliptic surface can have all possible singularities in the Kodaira's list, since this is true of elliptic fibrations on rational elliptic surfaces.
\end{rk}

\begin{rk}
In \cite[Remark 2.5]{HSVZ}, Hein-Sun-Viaclovsky-Zhang note that for special choices of elliptic curves $D$ and homology classes $[\gamma]\in H_1(D,\mathbb{Z})$, hyper-K\"ahler rotating the Calabi model along the model special Lagrangian fibration over $[\gamma]$ produces the semi-flat ansatz in a neighborhood of a type $I_{d}$ fiber.  The semi-flat model was used by Hein \cite{Hein} to construct Ricci-flat metrics on complements of $I_d$ fibers in rational elliptic surfaces.  The authors suggest that this could be used to identify the metrics by global hyper-K\"ahler rotation, which could lead to a completely different proof of Theorem~\ref{thm: main3Intro} in these special cases.
\end{rk}

\appendix
\section{Some Analysis Lemmas}

In this appendix we record, with proofs, several results which were needed for the analysis in Sections~\ref{sec: perturb}, ~\ref{sec: ACYL} and~\ref{sec: TY}.  These results are surely well-known.  However, since we have not been able to find references containing exactly the statements we need with proofs, we include complete proofs here for the reader's convenience.  The first result is a variational formula for the second fundamental form of a submanifold under variations of the metric; see \cite{Lott} for a related formula in codimension $1$.

\begin{lem}\label{lem: secFunVar}
Let $M^k \subset X^{n+k}$ be a smooth submanifold of a Riemannian manifold $(X,g_0)$.  Suppose that $g(t)$ is a smooth variation of Riemannian metrics for $t\in (-\epsilon, \epsilon)$.  Consider the product manifold $\overline{X}:=X\times (-\epsilon, \epsilon)$ equipped with the Riemannian metric $\bar{g} = dt^2 + g(t)$.  Let $A(t)$ denote the second fundamental form of $M \subset (X, g(t))$ and let $\overline{\nabla}$ denote the covariant derivative of $\bar{g}$.  For $p \in M \times \{0\}$, let $(x_1,\ldots, x_{k})$ be local coordinates on $M$centered at $p$ which are normal for $g(0)$.  Let $\{E_1, \ldots, E_n\}$ be a local orthonormal frame for $(TM)^{\perp} \subset (X,g(0))$. Then we have
\[
 \begin{aligned}
 \overline{\nabla_{t}}A_{ij}(0) &= \sum_{\alpha=1}^{n} \left( \left(\nabla^0_{i}\del_tg_{\alpha j} + \nabla^0_{j} \del_tg_{\alpha i} - \nabla^0_{\alpha}\del_tg_{ij}\right)  +  \frac{1}{2}\sum_{\beta=1}^{n}\del_t g_{\alpha\beta}A_{ij}^{\beta}(0)\right) E_{\alpha}(0)\\
 &\quad- \sum_{\alpha=1}^{n}\sum_{\ell=1}^{k}\frac{1}{2}\del_tg_{\alpha \ell} A_{ij}^{\alpha}(0) \del_{x_{\ell}},
 \end{aligned}
 \]
 where $\nabla^0$ denotes the covariant derivative on $(X,g(0))$.
 \end{lem}
 \begin{proof}
As in the statement of the lemma, let $\overline{X} = X\times (-\epsilon,\epsilon)$ and equip $\overline{X}$ with the metric $\overline{g} := dt^2 + g(t)$.  For $ p\in M$ and let $(x_1,\ldots, x_{k})$ be local coordinates on $M$ centered at $p$ which are normal for $g(0)=g_0$.  Let $\{E_1, \ldots, E_n\}$ be a local orthonormal frame for $(TM)^{\perp}$ with respect to $g(0)$. Extend $\{E_{\alpha}\}_{1\leq \alpha \leq n}$ smoothly in time to a local $g(t)$-orthonormal frame of $TM^{\perp} \subset (X,g_{t})$.  Choose local functions $(y_{1},\ldots, y_{n})$ vanishing at $p$ so that $\del_{y_i} = E_{i}(0)$ holds at $p$.  Then $(x_1,\ldots,x_k, y_1,\ldots, y_n)$ form local coordinates for $X$ and $g_0$ is the identity at $p$.  The second fundamental form is
\[
A_{ij}(t) = \sum_{\alpha=1}^{n} \langle \nabla^{t}_{\del_{x_i}} \del_{x_j}, E_{\alpha}(t)\rangle_{g_{t}} E_{\alpha}(t).
\]
 Let $\overline{\nabla}$ denote the covariant derivative of $\overline{g}$.  Then we have
 \[
 \begin{aligned}
 \overline{\nabla_{t}}A_{ij}(t) &= \sum_{\alpha=1}^{n} \langle \overline{\nabla}_{\del_t}\nabla^{t}_{\del_{x_i}} \del_{x_j}, E_{\alpha}(t)\rangle_{g_{t}} E_{\alpha}(t)\\
 &+ \sum_{\alpha=1}^{n} \langle \nabla^{t}_{\del_{x_i}} \del_{x_j}, \overline{\nabla}_{\del_t}E_{\alpha}(t)\rangle_{g_{t}} E_{\alpha}(t)\\
 &+\sum_{\alpha=1}^{n} \langle \nabla^{t}_{\del_{x_i}} \del_{x_j}, E_{\alpha}(t)\rangle_{g_{t}} \overline{\nabla}_{\del_t}E_{\alpha}(t).
 \end{aligned}
 \]
 At $t=0$ we have $\nabla^{0}_{\del_{x_i}} \del_{x_j} = (\nabla^{0}_{\del_{x_i}} \del_{x_j})^{\perp}$ and so
 \[
 \begin{aligned}
  \langle \nabla^{0}_{\del_{x_i}} \del_{x_j}, \overline{\nabla}_{\del_t}E_{\alpha}(0)\rangle_{g_{0}} &=    \langle \nabla^{0}_{\del_{x_i}} \del_{x_j}, (\overline{\nabla}_{\del_t}E_{\alpha})^{\perp}(0)\rangle_{g_{0}}\\
  &=\sum_{\beta} \langle \nabla^{0}_{\del_{x_i}} \del_{x_j}, E_{\beta}(0)\rangle_{g_0}\langle E_{\beta}(0),  \overline{\nabla}_{\del_t}E_{\alpha}(0)\rangle_{g_{0}}.
  \end{aligned}
\]
Since $E_{\beta}(t)$ is orthogonal to $TM$ with respect to $g(t)$, we can also write
\[
\begin{aligned}
\overline{\nabla}_{\del_t}E_{\alpha}(0) &= \langle E_{\beta}, \overline{\nabla}_{\del_t}E_{\alpha}(0)\rangle_{g_0} E_{\beta} + \sum_{\ell=1}^{k}\langle \del_{x_{\ell}}, \overline{\nabla}_{\del_t}E_{\alpha}(0)\rangle \del_{x_{\ell}}\\
&= \langle E_{\beta}, \overline{\nabla}_{\del_t}E_{\alpha}(0)\rangle_{g_0} E_{\beta} - \sum_{\ell=1}^{k}\langle \overline{\nabla}_{\del_t}\del_{x_{\ell}}, E_{\alpha}(0)\rangle \del_{x_{\ell}}.
\end{aligned}
\]
Putting these formulae together, we obtain
\[
 \begin{aligned}
 \overline{\nabla_{t}}A_{ij}(t) &= \sum_{\alpha=1}^{n} \langle \overline{\nabla}_{\del_t}\nabla^{t}_{\del_{x_i}} \del_{x_j}\big|_{t=0}, E_{\alpha}(0)\rangle_{g_{t}} E_{\alpha}(0)\\
 &+ \sum_{1\leq \alpha, \beta \leq n} \langle \nabla^{0}_{\del_{x_i}} \del_{x_j}, E_{\beta}(0)\rangle_{g_0}\langle E_{\beta}(0),  \overline{\nabla}_{\del_t}E_{\alpha}(0)\rangle_{g_{0}} E_{\alpha}(0)\\
 &+\sum_{1\leq \alpha, \beta \leq n} \langle \nabla^{0}_{\del_{x_i}} \del_{x_j}, E_{\alpha}(0)\rangle_{g_{0}}  \langle E_{\beta}, \overline{\nabla}_{\del_t}E_{\alpha}(0)\rangle_{g_0} E_{\beta}(0)\\
 &- \sum_{\alpha=1}^{n}\sum_{\ell=1}^{k}\langle \nabla^{0}_{\del_{x_i}} \del_{x_j}, E_{\alpha}(0)\rangle_{g_{0}}\langle \overline{\nabla}_{\del_t}\del_{x_{\ell}}, E_{\alpha}(0)\rangle \del_{x_{\ell}}.
 \end{aligned}
 \]
Swapping $\alpha, \beta$ in the second line and using that $\overline{\del_t}\langle E_{\alpha}, E_{\beta}\rangle_{g_t}=0$, the second and third lines cancel and we obtain
\[
 \begin{aligned}
 \overline{\nabla_{t}}A_{ij}(t) &= \sum_{\alpha=1}^{n} \langle \overline{\nabla}_{\del_t}\nabla^{t}_{\del_{x_i}} \del_{x_j}\big|_{t=0}, E_{\alpha}(0)\rangle_{g_{0}} E_{\alpha}(0)\\
 &- \sum_{\alpha=1}^{n}\sum_{\ell=1}^{k}\langle \nabla^{0}_{\del_{x_i}} \del_{x_j}, E_{\alpha}(0)\rangle_{g_{0}}\langle \overline{\nabla}_{\del_t}\del_{x_{\ell}}, E_{\alpha}(0)\rangle \del_{x_{\ell}}.
 \end{aligned}
 \]
We now compute
\[
\begin{aligned}
\langle \overline{\nabla}_{\del_t}\nabla^{t}_{\del_{x_i}} \del_{x_j}\big|_{t=0}, E_{\alpha}(0)\rangle_{g_{0}} &=  \ddt \Gamma_{ij}^{\alpha} + \sum_{\ell=1}^{k} \Gamma_{ij}^{\ell} \langle \overline{\nabla}_{\del_t}\del_{x_{\ell}}, E_{\alpha}\rangle_{g_0} + \sum_{\beta=1}^{n}\Gamma_{ij}^{\beta}\langle \overline{\nabla}_{\del_t} \del_{y_{\beta}}, E_{\alpha}(0)\rangle_{g_{0}}\\
&=  \ddt \Gamma_{ij}^{\alpha}  + \sum_{\beta=1}^{n}A_{ij}^{\beta}(0) \overline{\Gamma}_{t\beta}^{\alpha},
\end{aligned}
\]
where $\overline{\Gamma}$ denotes the Christoffel symbols of $\overline{g}$ in coordinates $(x_{\ell}, y_{\alpha}, t)$ and we have used that $\Gamma_{ij}^{\ell}(0)$ vanish at $p$.  By straightforward calculation we have
\[
\overline{\Gamma}_{t\ell}^{\alpha}= \frac{1}{2}\del_t g_{\alpha \ell}, \qquad \overline{\Gamma}_{t\beta}^{\alpha} = \frac{1}{2}\del_t g_{\alpha\beta}.
\]
Therefore,
\[
 \begin{aligned}
 \overline{\nabla_{t}}A_{ij}(t) &= \sum_{\alpha=1}^{n} \left(\ddt \Gamma_{ij}^{\alpha}  +  \frac{1}{2}\sum_{\beta=1}^{n}\del_t g_{\alpha\beta}A_{ij}^{\beta}\right) E_{\alpha}(0)\\
 &- \sum_{\alpha=1}^{n}\sum_{\ell=1}^{k}\frac{1}{2}\del_tg_{\alpha \ell} A_{ij}^{\alpha}(0) \del_{x_{\ell}}.
 \end{aligned}
 \]
By the well-known formula
\[
\ddt \Gamma_{ij}^{\alpha} = \left(\nabla_{i}^0\del_tg_{\alpha j} + \nabla_{j}^0 \del_tg_{\alpha i} - \nabla_{\alpha}^0\del_tg_{ij}\right),
\]
we obtain
\[
 \begin{aligned}
 \overline{\nabla_{t}}A_{ij}(t) &= \sum_{\alpha=1}^{n} \left( \left(\nabla^0_{i}\del_tg_{\alpha j} + \nabla^0_{j} \del_tg_{\alpha i} - \nabla^0_{\alpha}\del_tg_{ij}\right)  +  \frac{1}{2}\sum_{\beta=1}^{n}\del_t g_{\alpha\beta}A_{ij}^{\beta}(0)\right) E_{\alpha}(0)\\
 &- \sum_{\alpha=1}^{n}\sum_{\ell=1}^{k}\frac{1}{2}\del_tg_{\alpha \ell} A_{ij}^{\alpha}(0) \del_{x_{\ell}},
 \end{aligned}
 \]
 which is the desired result.
\end{proof}

Our next result concerns smoothing estimates to the mean curvature flow.  Such estimates are essentially standard in the theory.  However, we have been unable to a find reference for these estimates in a scale invariant form in a non-flat background.  We refer the reader to \cite[Chapter 3]{Ecker} for a proof when the background geometry is Euclidean and \cite[Theorem 1.2]{Smo} for similar estimates but which are not manifestly scale invariant.

\begin{prop}
Suppose that $M^k_0 \subset (X^{n+k},g)$ is compact submanifold and let $M_t$ be the mean curvature flow starting at $M_0$.  Suppose that there is a constant $K>0$ so that, for all $t\in [0, \frac{\alpha}{K})$, we have the following estimates:
\begin{itemize}
\item[(i)] the second fundamental form $A_t$ of $M_t$ satisfies
\[
|A_t|^2 \leq K.
\]
\item[(ii)] For all $0\leq \ell \leq m+1$, the curvature tensor $Rm$ of $(X,g)$ satisfies
\[
\sup_{M_{t}} |\nabla^{\ell}Rm|^2 \leq K^{2+\ell}.
\]
\end{itemize}
Then there exists a constant $C>0$ depending only on $\alpha, n, k, m$ so that
\[
|\nabla^{m}A|^2 \leq \frac{CK}{t^{m}}.
\]
\end{prop}
\begin{proof}
The proof is based on Shi's well-known estimates for the Ricci flow \cite{Ban, Shi1, Shi2}, see also \cite[Chapter 7]{ChKn}.  We will only prove the case $m=1$, the remaining cases being essentially identical and following from an easy induction argument.  For tensors $S,T$ we write $S*T$ for various contractions using the metric and multiplication by dimensional constants; the precise form will be irrelevant for our considerations.  We begin by recalling the well-known evolution equations for the second fundamental form along the mean curvature flow.  We have
 \[
 \begin{aligned}
 \ddt A &= \Delta A + A*A*A + A*Rm + \nabla Rm.
  \ddt \nabla A &= \nabla \ddt A + A*A*\nabla A.
  \end{aligned}
 \]
 Furthermore, we have
 \[
 \ddt A = \Delta A + A*A*A + A*Rm + \nabla Rm.
 \]
Therefore
\[
\begin{aligned}
\nabla \ddt A &= \nabla \Delta A + \nabla A*A*A + \nabla A*Rm + A*\nabla Rm + \nabla\nabla Rm\\
&= \Delta \nabla A + (\nabla A)*A*A + \nabla A*Rm + A*\nabla Rm + \nabla\nabla Rm,
\end{aligned}
\]
 where we recall that everything is taken up to dimensional constants.  Then we have
 \[
 \ddt|\nabla A|^2 = 2 \langle \nabla A, \Delta \nabla A\rangle + (\nabla A)^{*2}*A*A + \nabla A^{*2}*Rm + \nabla A*A*\nabla Rm + \nabla A*\nabla\nabla Rm.
 \]
 Thus
 \[
 \begin{aligned}
 \left(\ddt - \Delta\right)|\nabla A|^2 &\leq -2|\nabla \nabla A|^2\\
 & +C\left(|\nabla A|^2|A|^2 + |\nabla A|^2|Rm| + |\nabla A||A||\nabla Rm| + |\nabla A||\nabla\nabla Rm|\right)
 \end{aligned}
 \]
 for a dimensional constant $C$.  We also have
 \[
 \left(\ddt - \Delta\right)| A|^2 \leq -2|\nabla A|^2 + C\left(|A|^4 + |A|^2|Rm|+|A|\nabla Rm|\right).
 \]
 Now, as before, assume that, on the interval $t\in[0, \frac{\alpha}{K})$ we have
 \[
 |A|^2 \leq K, \quad |Rm| \leq K, \quad  |\nabla Rm| \leq K^{3/2},\quad  |\nabla\nabla Rm| \leq K^{2}.
 \]
Consider the quantity $F= t|\nabla A|^2+\beta|A|^2$.  Then we have
\[
\begin{aligned}
\left(\ddt - \Delta\right)F &\leq |\nabla A|^2 + \\
&\quad Ct\left(|\nabla A|^2|A|^2 + |\nabla A|^2|Rm| + |\nabla A||A||\nabla Rm| + |\nabla A||\nabla\nabla Rm|\right)\\
&\quad - 2\beta|\nabla A|^2 + C\beta\left(|A|^4 + |A|^2|Rm|+|A|\nabla Rm|\right).\\
\end{aligned}
\]
Now, we want to estimate the $|\nabla A|^2$ term.  Write
\[
\begin{aligned}
|\nabla A||\nabla\nabla Rm| &= (|\nabla A||\nabla\nabla Rm|^{\alpha})(|\nabla \nabla Rm|^{1-\alpha})\\
&\leq |\nabla A|^2|\nabla\nabla Rm|^{2\alpha})+ |\nabla \nabla Rm|^{2(1-\alpha)}.
\end{aligned}
\]
By considering the scaling of each term, we are lead to take $2(1-\alpha)\times 4 = 6$ or in other words, $\alpha = \frac{1}{4}$.  Then we get
\[
\begin{aligned}
\left(\ddt - \Delta\right)F &\leq (1+ CtK -2\beta)|\nabla A|^2 +CtK^3\\
&\quad + C\beta\left(|A|^4 + |A|^2|Rm|+|A||\nabla Rm|\right).
\end{aligned}
\]
Choosing $\beta$ large depending only on $\alpha, C$, we obtain
\[
\left(\ddt - \Delta\right)F \leq C'\beta K^2,
\]
for a uniform constant $C'$.  It follows that $F-C'\beta K^2t \leq F(0)\leq \beta K$ and so
\[
t|\nabla A|^2 \leq \beta K +C'\beta K^2t \leq C''K,
\]
which is the desired estimate.
\end{proof}

\end{document}